\documentclass[11pt]{article}

\usepackage{amsmath}
\usepackage{amssymb} 
\usepackage{caption}
\usepackage{graphicx}
\usepackage{authblk}
\usepackage[hypertexnames=false,colorlinks=true,linkcolor=blue,citecolor=blue]{hyperref}
\usepackage[numbers,comma,square,sort&compress]{natbib}
\usepackage[a4paper,text={6.5in,10in},centering]{geometry}
\usepackage{enumitem}

\usepackage{subfig}
\usepackage{bm} 

\graphicspath{{eps/}{pdf/}{images/}}

\makeatletter\@addtoreset{equation}{section}\makeatother

 {\begin{trivlist} \item[]{\bf Proof. }}%
 {\hspace*{\fill}$\rule{.4\baselineskip}{.4\baselineskip}$\end{trivlist}}

 {\begin{trivlist}\item[]\textbf{Acknowledgments.}}{\end{trivlist}}
 
\newtheorem{lemma}{Lemma}

\newtheorem{theorem}{Theorem}

\newtheorem{remark}{Remark}

\newenvironment{proof}[1][.]%
{\begin{trivlist}\item[]\textbf{Proof#1 }}%
 {\hspace*{\fill}$\rule{0.3\baselineskip}{0.35\baselineskip}$\end{trivlist}}

\DeclareMathOperator{\sech}{sech}

\DeclareMathOperator{\arccosh}{arccosh}

\newcommand{\bu}{\mathbf{u}}

\title{Existence of stationary fronts in a system of two coupled wave equations with spatial inhomogeneity}
\author{Jacob Brooks, Gianne Derks, David J.B. Lloyd}
\date{\today}
 
\begin{document}
\maketitle

\begin{abstract} 

  We investigate the existence of stationary fronts in a coupled
  system of two sine-Gordon equations with a smooth, ``hat-like"
  spatial inhomogeneity. The spatial inhomogeneity corresponds to a
  spatially dependent scaling of the sine-Gordon potential term. The
  uncoupled inhomogeneous sine-Gordon equation has stable stationary
  front solutions that persist in the coupled system.  Carrying out a
  numerical investigation it is found that these inhomogeneous sine-Gordon fronts loose stability, provided the coupling
  between the two inhomogeneous sine-Gordon equations is strong
  enough, with new stable fronts bifurcating. In order to analytically study the bifurcating fronts, we first
  approximate the smooth spatial inhomogeneity by a piecewise constant
  function. With this approximation, we prove analytically the existence of a
  pitchfork bifurcation. To complete the argument, we prove that transverse
  fronts for a piecewise constant inhomogeneity persist for the smooth
  ``hat-like" spatial inhomogeneity by introducing a fast-slow structure and using geometric singular
  perturbation theory.
\end{abstract}
	
\begin{section}{Introduction} 
  In this paper, we study the existence of front solutions in the
  following system of two spatially inhomogeneous sine-Gordon
  equations with coupling
\begin{align}
\begin{split}
	\theta_{xx} - \theta_{tt} &= (1-d\rho(x))\sin \theta -\alpha\sin(\theta-\psi), 
	\\
	\psi_{xx} - \psi_{tt} &= (1-d\rho(x))\sin \psi -\alpha \sin (\psi-\theta),
\end{split}
\label{coupledsystem}
\end{align}
where $\theta = \theta(x,t),\psi = \psi(x,t)$, $\alpha\in\mathbb{R}$
is the coupling parameter and $d \in \mathbb{R}$ measures the strength
of the spatial inhomogeneity $\rho(x).$ We consider
the ``hat-like'' spatial inhomogeneity
\begin{equation}\label{e:smooth_rho}
\rho(x;\Delta, \delta) := \frac{\tanh((x+\Delta)/\delta)+\tanh((-x+\Delta)/\delta)}{2},
\end{equation}
with $0<\delta\ll1$, $\Delta>0$.  Since $0<\delta \ll 1,$ apart from a
small region near $|x| =\Delta$, the inhomogeneity is near zero for
$|x|>\Delta$ and near~1 for $|x|<\Delta$; see Figure~\ref{step1}(a).
Thus the variable $\Delta$ measures the half width of the
$\rho(x;\Delta, \delta)$ ``hat''.  The small parameter $\delta$
determines the steepness of the inhomogeneity's jump. As $\delta \to 0,$ $\rho(x; \Delta, \delta)$ converges pointwise to the piecewise constant function 
\begin{equation}
\rho_{0}(x;\Delta)= 
\begin{cases}
0,&  \vert x \vert > \Delta,  \\
1,     &  \vert x \vert < \Delta,
\end{cases} 
\label{step}
\end{equation}
(see Figure~\ref{step1}(b)) which will also be considered in this paper.
\begin{figure}
	\centering
	\includegraphics[scale=0.75]{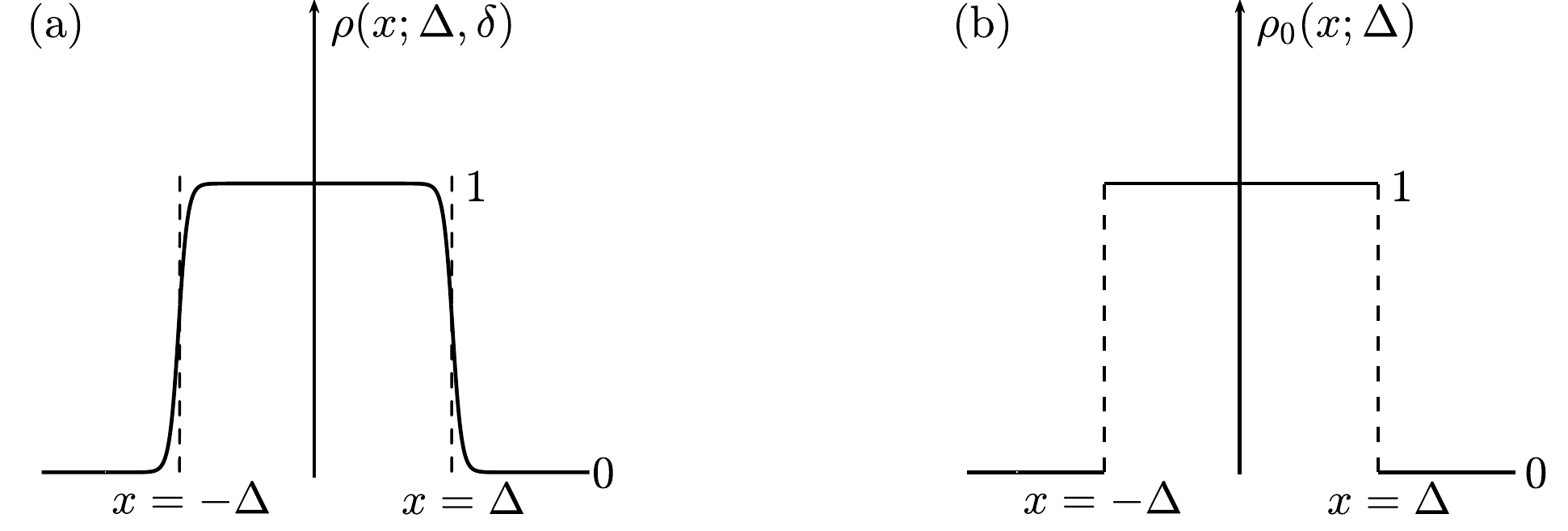}
	\caption{(a) is a sketch of the smooth ``hat-like" spatial inhomogeneity $\rho(x;\Delta,\delta)$ whilst (b) is a sketch of the piecewise constant inhomogeneity $\rho_{0}(x;\Delta).$ Note the dashed lines correspond to $x=\pm\Delta,$ illustrating $\rho(\pm \Delta;\Delta,\delta) \approx 1/2.$}
	\label{step1}
\end{figure}

The coupled system~(\ref{coupledsystem}) can be interpreted as a
continuous approximation of two pendulum chains interacting
with one another where the mass of the pendulums is allowed to change. The dependent variables $\theta$ and $\psi$ represent the angles of the two pendulum chains, the parameter $\alpha$ corresponds to the coupling strength between the two chains and the spatial inhomogeneity $\rho(x)$ represents a change in mass of the pendulums. The coupled system without spatial inhomogeneity $(d=0)$ was proposed as an elementary model for two parallel adatomic
chains with small local interaction in \cite{Braun1988}. Additionally the coupled system has been studied as a simple model of the DNA double
helix~\cite{Homma1984, Yakushevich1989, Derks2011}, where the DNA chain is represented as a coupled pendulum chain. Furthermore, in the context of DNA it was proposed in \cite{Derks2011}, that the inhomogeneity $(d \ne 0)$ in the coupled system represents the presence of an RNA protein, an important mediator in DNA copying.

When $u(x,t)= \theta(x,t) = \psi(x,t)$ and $d=0$, the coupling and
inhomogeneous terms in the system (\ref{coupledsystem}) vanish. As a
result the system~(\ref{coupledsystem}) reduces to the
celebrated sine-Gordon equation~\cite{Bour1862,Frenkel1939}
\begin{equation*}
u_{xx}-u_{tt} = \sin (u).
\end{equation*}
The sine-Gordon equation is fully integrable and possesses a family of
travelling front solutions
\begin{equation} \label{travelling} u_{\rm sG}^{\pm}(x,t;c) =
  4\arctan\bigg(\exp\bigg(\frac{\pm1}{\sqrt{1-c^2}}{(x-ct+x^*)}\bigg)\bigg),
  \quad \vert c \vert <1, \quad x^* \in \mathbb{R}.
\end{equation}
Here $u_{\rm sG}^{+}$ represents the monotonic increasing front,
whilst $u_{\rm sG}^{-}$ the monotonic decreasing one. Both fronts are
centred at $x=-x^*$ when $t=0$ and move with constant speed $c.$ Thus when $c=0$ the fronts are stationary. Note that
$u_{\rm sG}^-= 2\pi - u_{\rm sG}^{+}$, which reflects the
$u\mapsto 2\pi-u$ symmetry of the sine-Gordon equation. 

From an application point-of-view, understanding front solutions and
their dynamics is of special interest. Recent research into the
interaction of travelling sine-Gordon fronts with finite length
spatial inhomogeneities has produced fascinating
results. In~\cite{Piette2007} Piette and Zakrzewski studied the scattering
of (\ref{travelling}) in the
inhomogeneous sine-Gordon equation
\begin{equation}\label{e:sine_Gordon}
u_{xx}-u_{tt} = (1-d\rho_{0}(x;\Delta))\sin (u)
\end{equation}
with the piecewise constant spatial inhomogeneity (\ref{step}). Starting the travelling front far away from the inhomogeneity they noted several different phenomena dependent on the initial speed and $d,\Delta.$ Fix $d, \Delta>0$, then for values of $c$ less than some critical one the travelling front would not pass and get stuck in the inhomogeneity. For higher speeds the front could pass through the inhomogeneity. Interestingly they noted some speed values less than the critical one that fronts could bounce back out of the inhomogeneity. More recently, Goatham \textit{et al.} studied the scattering of the travelling sine-Gordon fronts~(\ref{travelling}) in~(\ref{e:sine_Gordon}) with smooth non-steep spatial inhomogeneities; see \cite{Goatham2011}.

It has also been shown the existence of stationary fronts plays a role
in studying the interaction of travelling fronts with spatial
inhomogeneities~\cite{Derks2012}. This is because stationary front
solutions correspond to fixed points in the dynamical systems approach
to the wave equation. The existence of stationary front solutions to the
inhomogeneous sine-Gordon equation~(\ref{e:sine_Gordon}), with
boundary conditions $u(-\infty)=0$ and $u(+\infty)=2\pi,$ for all
$d\in\mathbb{R}$ and $\Delta>0$ was established in~\cite{Derks2012}. We
denote these fronts by $u_0 (x;d,\Delta)$, hence
$u_0 (x;d=0,\Delta)=u_{\rm sG}^+ (x,t;0)$. In the special case, $d=1$, Derks {\it et
  al.} \cite{Derks2012} also gave the explicit expression for the front solutions,
\begin{equation} \label{q1front}
u_0(x;d=1,\Delta)= 
\begin{cases}
4\arctan(e^{x+x^*}),&  x < -\Delta \\
\pi +\sqrt{2h}x ,     & \vert x \vert < \Delta \\
4\arctan(e^{x-x^*}), & x> \Delta,
\end{cases} 
\end{equation}
where $0<h(\Delta)< 2$ is uniquely determined by 
\begin{equation}
  \label{eq:Delta_xstar}
  \Delta = \arccos (h-1)/\sqrt{2h}	\quad\mbox{and}\quad
  x^* = \ln (\tan(\arccos(1-h)/4))+\Delta.
\end{equation}
These solutions persist in the coupled system (\ref{coupledsystem}) when $\delta = 0$.

Returning to the full coupled system~(\ref{coupledsystem}), when there
is no spatial inhomogeneity, i.e. $d=0$, the sine-Gordon fronts
$\theta(x,t)=\psi(x,t)=u_{\rm sG}^\pm (x,t;c)$ are stable if
$-1/2<\alpha<0$ and unstable if $0<\alpha<1/2;$ see
\cite{Braun1988}. We illustrate this instability for the stationary
front $u_{\rm sG}^{+}(x,t;0)$ in a numerical time simulation of
(\ref{coupledsystem}) with $d=0$ and $\alpha=0.1$ in
Figure~\ref{spacetimeplots}(a). The instability manifests itself by
the stationary fronts travelling apart. We now consider a numerical
time simulation of the stationary front $u_{\rm sG}^{+}(x,t;0)$ in
(\ref{coupledsystem}) when $d\ne0;$ see
Figure~\ref{spacetimeplots}(b-c). Fixing $0<\delta\ll 1, d>0$ and
$\Delta>0$ the stationary fronts with $\theta=\psi=u_{\rm sG}^+$ first
adapt themselves to account for the presence of a spatial
inhomogeneity then two different phenomena can occur. For small values
of $\alpha>0$ the stationary fronts are stable; see
Figure~\ref{spacetimeplots}(b). On the other hand, for larger values
of $\alpha<1/2$ the stationary fronts are unstable and bifurcate to
new stationary fronts with $\theta\neq \psi$; see
Figure~\ref{spacetimeplots}(c). When plotting in the coordinates
 \begin{equation}\label{e:uv_coords}
   u=\frac{\theta+\psi}{2} \quad\mbox{and}\quad v = \frac{\theta-\psi}{2}
\end{equation}
one starts to see how this bifurcation occurs. The case
$\theta(x) = \psi(x)$ in the original variables corresponds to
$(u(x),v(x))=(\theta(x),0)$ in the new ones. For fixed $d>0,$ we see
in Figure~\ref{spacetimeplots}(b) that for small values of $\alpha$,
$(u(x),v(x))=(\theta(x),0)$ is stable, i.e. the inhomogeneity has
stabilized the sine-Gordon front in the coupled system. For larger
values of $\alpha$, Figure~\ref{spacetimeplots}(c) shows that a
bifurcation has happened: the effect of the coupling initially
dominates the stabilizing effect of the inhomogeneity and the $\theta$
and $\psi$ components start to travel apart as in (a), but soon
afterwards, the inhomogeneity dominates again and the fronts get
stopped. This results in $v$ becoming a small localised pulse.
\begin{figure}
	\captionsetup[subfigure]{labelformat=empty}
	\centering
	\subfloat[(a)]{\includegraphics[scale=0.8]{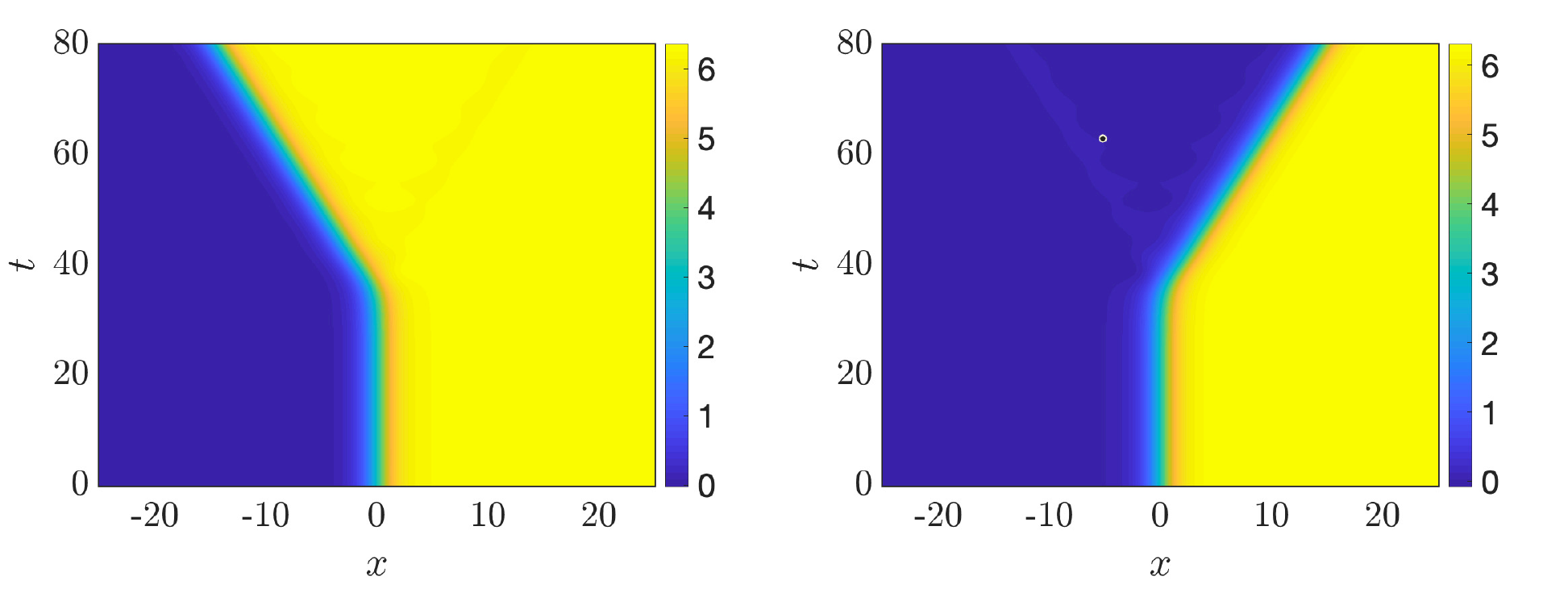}}\hspace{-10pt}\\
	\subfloat[(b)]{\includegraphics[scale=0.8]{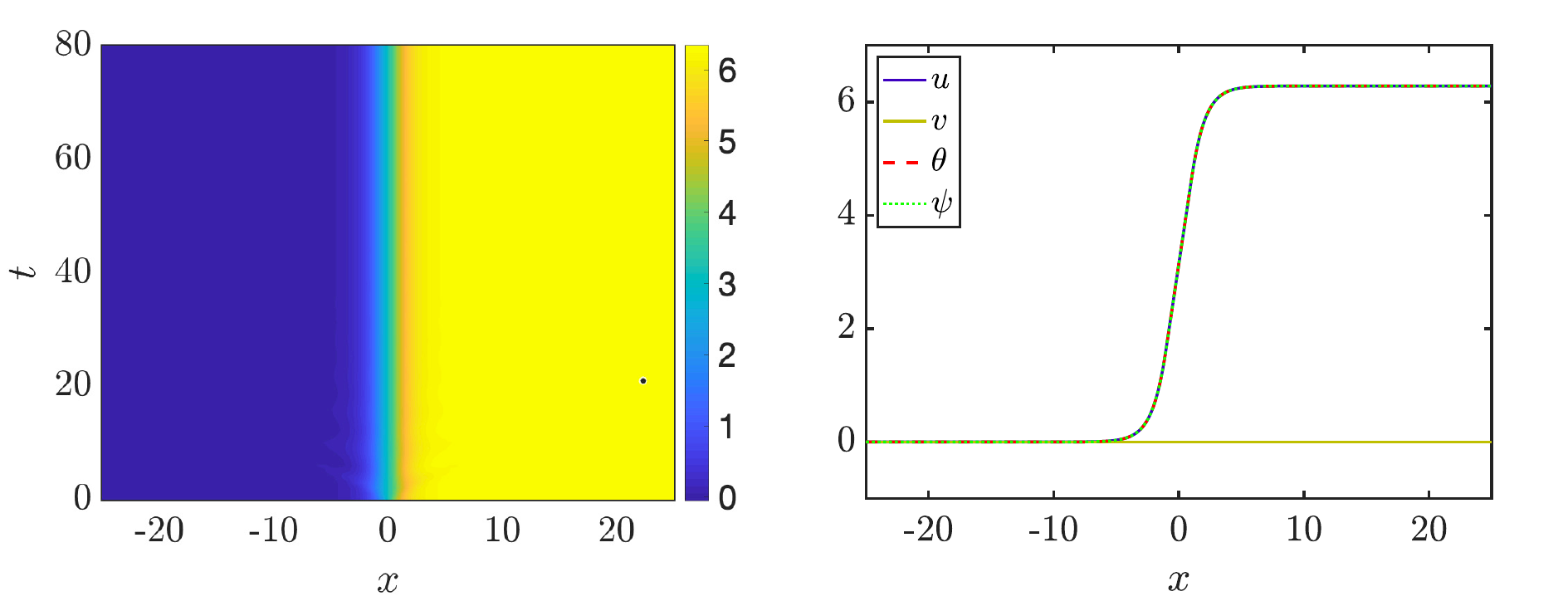}} \hspace{-10pt} \\ 
	\subfloat[(c)]{\includegraphics[scale=0.8]{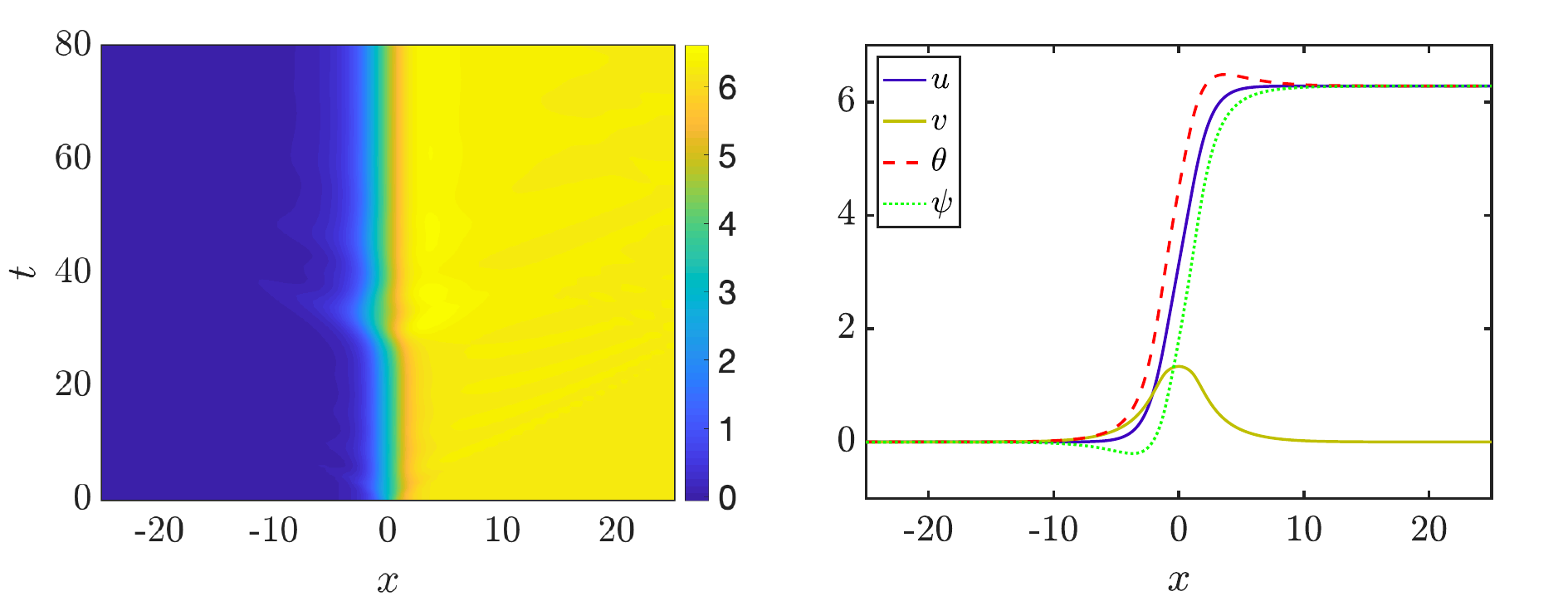}}\\
	\caption{(a) corresponds to space-time plots of the dynamics
          of the initial condition consisting of the stationary sine-Gordon
          front solutions in the system (\ref{coupledsystem}) with
          $\alpha=0.1$ and no spatial inhomogeneity, i.e. $d=0$. The
          left panel corresponds to $\theta(x,t)$ and the right
          $\psi(x,t)$. After a while, the fronts loose stability and
          begin to travel apart.
          The left panel of~(b) correspond to a space-time plot of the
          $\theta(x,t)$ dynamics of the same initial condition, now
          with the spatial inhomogeneity $\delta=1/15$ and $\Delta= d =1$, while keeping  $\alpha=0.1$. The right panel corresponds to the
          solution profile at $t=80$.
          The left panel of~(c) corresponds to a space-time plot of
          the $\theta(x,t)$ dynamics, again with the same initial
          condition, now for a stronger coupling $\alpha=0.4$, and the
          same inhomogeneity $\delta = 1/15$ and $\Delta= d =1$.
          The right panel corresponds to the solution profile at
          $t=80.$ Note that as we are interested in stationary
          solutions a small damping term was added in (b) and (c) to
          suppress the additional radiation generated by the initial
          adaptation in the Hamiltonian system.}
	\label{spacetimeplots}
\end{figure}

The main aim of this paper is to provide a detailed numerical and
analytical understanding of the bifurcation shown in Figure
(\ref{spacetimeplots})(b-c). We will do this by studying the existence
of stationary fronts to the coupled system~(\ref{coupledsystem}) when
$0<\alpha<1/2$, $d>0$, $\Delta>0$ and $0\leqslant \delta \ll 1$. The
restriction on $\alpha$ is due to the fact that the steady state
$(\theta,\psi)=(2\pi,2\pi)$ is temporally unstable for $\alpha> 1/2$. Note that the $\delta = 0$ case corresponds to the piecewise constant inhomogeneity, i.e. $\rho(x) = \rho_0(x;\Delta)$ given by (\ref{step}). Since the
sine-Gordon symmetry persists for the full coupled system
(\ref{coupledsystem}) we can restrict ourselves to the monotonic
increasing fronts. It is helpful to keep the change of
variables~(\ref{e:uv_coords}). Consequently, the existence of
stationary front solutions in the coupled inhomogeneous
system~(\ref{coupledsystem}) is equivalent to the existence of
solutions to the Boundary Value Problem (BVP)
\begin{gather}
\begin{split}
u_{xx} &= (1-d\rho(x))\sin u \cos v, 
\\
v_{xx} &= (1-d\rho(x)) \sin v \cos u -\alpha \sin 2v.\label{alphasys}
\end{split}
	\\
	\lim_{x\to-\infty}(u(x),v(x))=(0,0) \quad \text {and} \quad \lim_{x\to+\infty}(u(x),v(x))=(2\pi,0). \label{BCS}
\end{gather}
When $v(x) = 0,$ the system (\ref{alphasys}) reduces to the stationary
inhomogeneous sine-Gordon equation
\begin{equation}\label{e:p_sine_Gordon}
u_{xx} = (1-d\rho(x))\sin u, \quad
\lim_{x\to-\infty}u(x)=0 \quad \text{and} \quad \lim_{x\to\infty}u(x)=2\pi.
\end{equation} 

An obvious starting point for the analysis to understand the
bifurcation occurring in Figure~\ref{spacetimeplots} is to build on
the work on the uncoupled inhomogeneous sine Gordon equation
in~\cite{Derks2012} and consider the case $d=1$, $\delta=0$ (the
piecewise constant inhomogeneity $\rho_0(x;\Delta)$ given by (\ref{step})) and
carry out a Lyapunov-Schmidt reduction analysis for the explicit front
solution~(\ref{q1front}). As the front is a non-constant state, this poses some challenges
to be overcome. The next step would be to extend the existence for the
piecewise constant inhomogeneity~$\rho_0(x;\Delta)$ to the smooth inhomogeneity $\rho(x; \Delta, \delta)$~(\ref{e:smooth_rho}). Whilst for fixed~$\Delta$ the function
$\rho(x;\Delta, \delta)$ converges pointwise to $\rho_0(x;\Delta)$ as
$\delta \to 0$, the link between the front~(\ref{q1front}) and front
solutions in (\ref{coupledsystem}) is not immediately obvious.

In order to overcome this issue, we adapt an approach by Goh and
Scheel~\cite{Goh2016}. 
They study fronts in the complex Ginzburg-Landau equation with a
smooth single step inhomogeneity and characterise this
inhomogeneity with an additional Ordinary Differential Equation (ODE).  Following this approach, we
extend the coupled system with the following additional ODE for the inhomogeneity
$\rho$:
\[
\delta^2\rho_{xx} = 4\rho^3 - (6+4\epsilon)\rho^2 + 2(1+\epsilon)\rho.
\]
When $0 < \epsilon, \delta \ll 1$, this ODE has explicit solutions
where~(\ref{e:smooth_rho}) is the leading order approximation
and~$\epsilon$ can be expressed in terms of $\Delta$ and $\delta$
($\epsilon = \mathcal{O}(e^{-2\sqrt{2}\Delta/\delta})$).  Including this ODE
in the system (\ref{alphasys}) turns the problem into a fast-slow dynamical system
where geometric singular perturbation theory can be applied and
existence of solutions can be proved for fixed $\Delta$ and
$0<\delta\ll1$. The key part of the geometric singular perturbation
theory is to understand the singular limit $\delta \to 0$, where
$\rho$ is determined by an algebraic equation, then use Fenichel's
theorems \cite{Fenichel1979} to prove persistence when
$0<\delta \ll1.$

We have four main results. The first is a systematic numerical
investigation of the bifurcation illustrated in
Figure~\ref{spacetimeplots}(b--c) using numerical path following in the
$(\alpha,d,\Delta)$ parameter space. In particular, this numerical
investigation allows us to explore several limiting cases where
analysis is possible. With this analysis, we obtain two theorems about
the location of the bifurcation from the sine-Gordon front and the
emerging bifurcating states using the piecewise constant
$\rho_0(x;\Delta)$~(\ref{step}). Finally, we prove that the fronts
found for a piecewise constant inhomogeneity $\rho_0(x;\Delta)$
persist for the smooth inhomogeneity
$\rho(x;\Delta, \delta)$~(\ref{e:smooth_rho}) for $0<\delta\ll1$.

%
%

The structure of this paper is as follows. Section
\ref{numericalinvestigation} presents a numerical investigation into
the BVP~(\ref{alphasys}-\ref{BCS}). Starting with a solution of the
form $(u(x),v(x))=(u(x),0)$, we show the existence of a pitchfork
bifurcation at which $v(x)$ becomes non-zero in the parameter space
with $\alpha\in (0, 1/2)$ and $d,\,\Delta>0$. In Section
\ref{diagramanalysis} we use the piecewise constant inhomogeneity
$\rho_0(x;\Delta)$ to determine an analytical expression for the
bifurcation locus in case $d=1$ and derive approximations for the
bifurcation locus observed in Section
\ref{numericalinvestigation} in the cases $d$ large and $\Delta$
large. Using the bifurcation locus expression and the front
solution (\ref{q1front}), we employ Lyapunov-Schmidt reduction to show
the existence of a pitchfork bifurcation and approximate the bifurcating
solutions in Section \ref{coupledsys} for the case $d=1$. In Section
\ref{hamslowfast} we use regular and singular perturbation theory to
show that if solutions exist with the inhomogeneity $\rho_0(x;\Delta)$,
then they persist for the smooth inhomogeneity
$\rho(x;\Delta, \delta).$ This result rigorously justifies
comparisons between the numerics and the analysis made throughout the
paper. Finally, in Section \ref{discussion} we end with a summary of
the main results and a discussion of further research.
\end{section}


\begin{section}{Numerical bifurcation investigation} \label{numericalinvestigation}

  In this section we numerically investigate a bifurcation in the
  inhomogeneous coupled sine-Gordon BVP (\ref{alphasys}-\ref{BCS})
  from the solution state $(u(x),v(x))=(u(x),0)$ to one where
  $v(x) \ne 0$. Recall (\ref{alphasys}) has four parameters: the
  coupling parameter~$\alpha$, the strength of the inhomogeneity~$d$,
  the steepness~$\delta$ and the width $\Delta$. Throughout this section we keep
  the steepness parameter $\delta=1/15$ fixed. First fixing
  $d=\Delta=1$, we determine a bifurcation point in the remaining
  parameter $\alpha$ whereby $v(x)$ undergoes a pitchfork
  bifurcation. After this, we keep $d=1$ fixed but consider any $\Delta>0$
  and plot the corresponding bifurcation diagram in the
  $(\alpha,\Delta)$ plane. We finish this section by showing the
  pitchfork bifurcation occurs for any $d> 0$ and give plots in the
  $(d,\Delta)$ plane for various fixed values of $\alpha\in(0,1/2)$,
  illustrating the existence of a two dimensional bifurcation manifold
  in the $(\alpha,d, \Delta)$ parameter space. 

  To start this numerical bifurcation section, we discuss how to set
  up the problem for numerical investigation.

\subsection{Implementation}
We will study the BVP~(\ref{alphasys}-\ref{BCS}) using
AUTO07p~\cite{AUTO}. AUTO07p requires us to re-write~(\ref{alphasys})
as a first order ODE system. Hence we consider
	\begin{align} \label{numericsystem}
	u_{x} &= p, \nonumber
	\\
	p_{x} &= (1-d\rho(x;\Delta, \delta))\sin u \cos v, \nonumber
	\\
	v_{x} &= q, 
	\\
	q_{x} &= (1-d\rho(x;\Delta, \delta))\sin v \cos u -\alpha\sin 2v, \nonumber
\nonumber
	\end{align}
        where the smooth inhomogeneity~$\rho(x;\Delta, \delta)$ is
        defined in~(\ref{e:smooth_rho}). Note that AUTO07p is unable
        to deal with the piecewise constant approximation
        $\rho_0(x;\Delta)$ of $\rho(x;\Delta, \delta)$.  The dynamics
        of (\ref{numericsystem}) are centred at $x=0$ however AUTO07p
        requires us to consider the dynamics on a positive spatial
        interval. Thus we apply the spatial translation
        $\tilde{x}=x+50$ to (\ref{numericsystem}) which centres the
        dynamics at $\tilde{x}=50.$ We consider the dynamics over the
        finite interval $\tilde{x}\in [0,100]$ with boundary
        conditions $(u,p,v,q)(\tilde{x}=0) = (0,0,0,0)$ and
        $ (u,p,v,q)(\tilde{x}=100) = (2\pi,0,0,0).$ When plotting the data we
        have reverted the shift transformation (consider
        $x=\tilde{x}-50$) so that it is once again centred at $x=0$
        and satisfies (\ref{numericsystem}). Due to the spatial
        inhomogeneity no phase condition is needed. Finally, we use standard
        AUTO07p tolerances as detailed in~\cite{AUTO}.
    
    \begin{figure}
    	\centering
    	\includegraphics[scale=0.57]{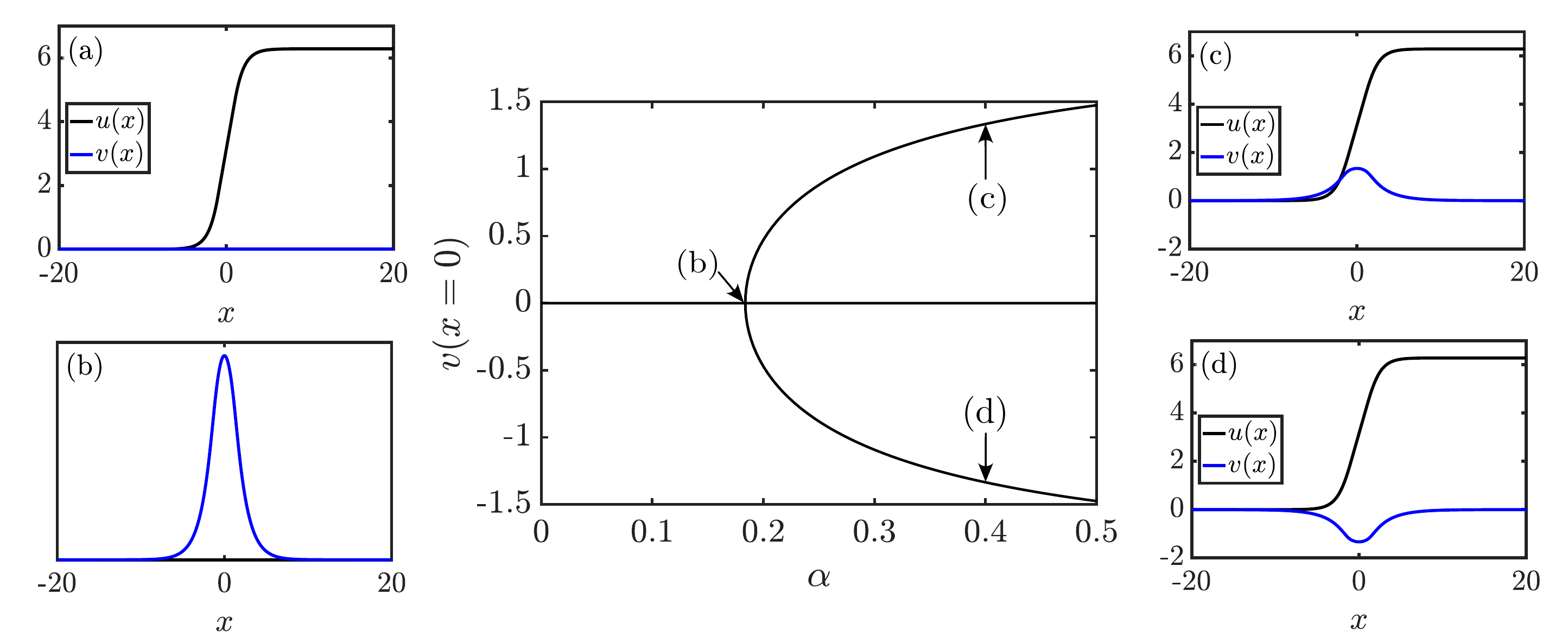}
    	  	\caption{A plot of the bifurcation branches and evolution of $v(x)$ solution in the
    	        system~(\ref{numericsystem}) with $d=1=\Delta$ and
    	      varying $\alpha\in (0,1/2)$. Panel (a) corresponds to the solution with $v(x)=0$ which exists for all $\alpha\in (0,1/2)$. The bifurcation
    	     locus is at $\alpha\approx 0.18$. Panel~(b) shows the $u$ (black) and $v$ (blue) eigenfunctions with the eigenvalue~0 respectively. Panel~(c) displays the
    	   $u$ and $v$ components of the solution on the positive
    	   bifurcation branch at $\alpha=0.4$. Panel~(d) displays the
    	   $u$ and $v$ components of the solution on the negative
    	   bifurcation branch at $\alpha=0.4$.}
    		\label{q=1pitchfork}
    
    \end{figure}

\subsection[]{The bifurcation when $\bm{d=1}$}

Consider $d=1.$ Then, in the limit $\delta \to0,$ it follows from~\cite{Derks2012} that for any $\Delta>0$,
the system~(\ref{numericsystem}) has stationary front solutions
$(u(x),u_x(x),0,0)$ with $u(x)$ given by~\eqref{q1front}.  When $0<\delta \ll 1,$ AUTO07p
shows that there are nearby stationary solutions
of~(\ref{numericsystem}) that satisfy the boundary conditions
$u(x=-50)=0$ and $u(x=50)=2\pi$.  Considering $\Delta$ fixed and
varying $\alpha\in (0,1/2)$, one can find a pitchfork
bifurcation point at some $\alpha=\alpha^*$ whereby the $(v,q)$
component can become non-zero. For example, fixing $\Delta=1$ we find
a pitchfork bifurcation at $\alpha^* \approx 0.18$. Figure
\ref{q=1pitchfork} shows this bifurcation for $\Delta =1$ and the new
emerging branches where $v(x)$ is non-zero in the region
$\alpha \in (0,1/2)$. Figure \ref{q=1pitchfork}(a-b)
show the~$u$ and~$v$ components of the solution and the eigenfunctions
for the zero eigenvalue, respectively, at the bifurcation point
$\alpha^*\approx 0.18$. In particular, we see that the eigenfunction
for the $u$ component is zero whereas for the $v$ component it is a
localised function. Fixing $\alpha = 0.4>\alpha^*$, the~$u$
and~$v$ components of the solution on the positive and negative branches of the
system~(\ref{numericsystem}), with $d=1$ and $\Delta=1$, are plotted in
Figures~\ref{q=1pitchfork}(c) and \ref{q=1pitchfork}(d) respectively. Here we observe the emergence of a
localised $v$ component that steadily grows as we move away from the
bifurcation point.

In Figure \ref{q=1pitchforklocus} we trace out the locus of the
pitchfork bifurcation point $\alpha^*(\Delta)$ in the
$(\alpha,\Delta)$ plane. For parameters $(\alpha,\Delta)$ chosen on
the right side of the bifurcation diagram there exists
non-zero $v(x)$ solutions. Whilst on the left only solutions with $v=0$
exist. This figure shows that the largest value of $\alpha$ that the pitchfork
bifurcation can occur is $\alpha^* \approx 0.185$. Hence, when $d=1$,
the bifurcation occurs when the coupling between $u(x)$ and $v(x)$
components is small.

\begin{figure}
	\centering
	\includegraphics[scale=0.4]{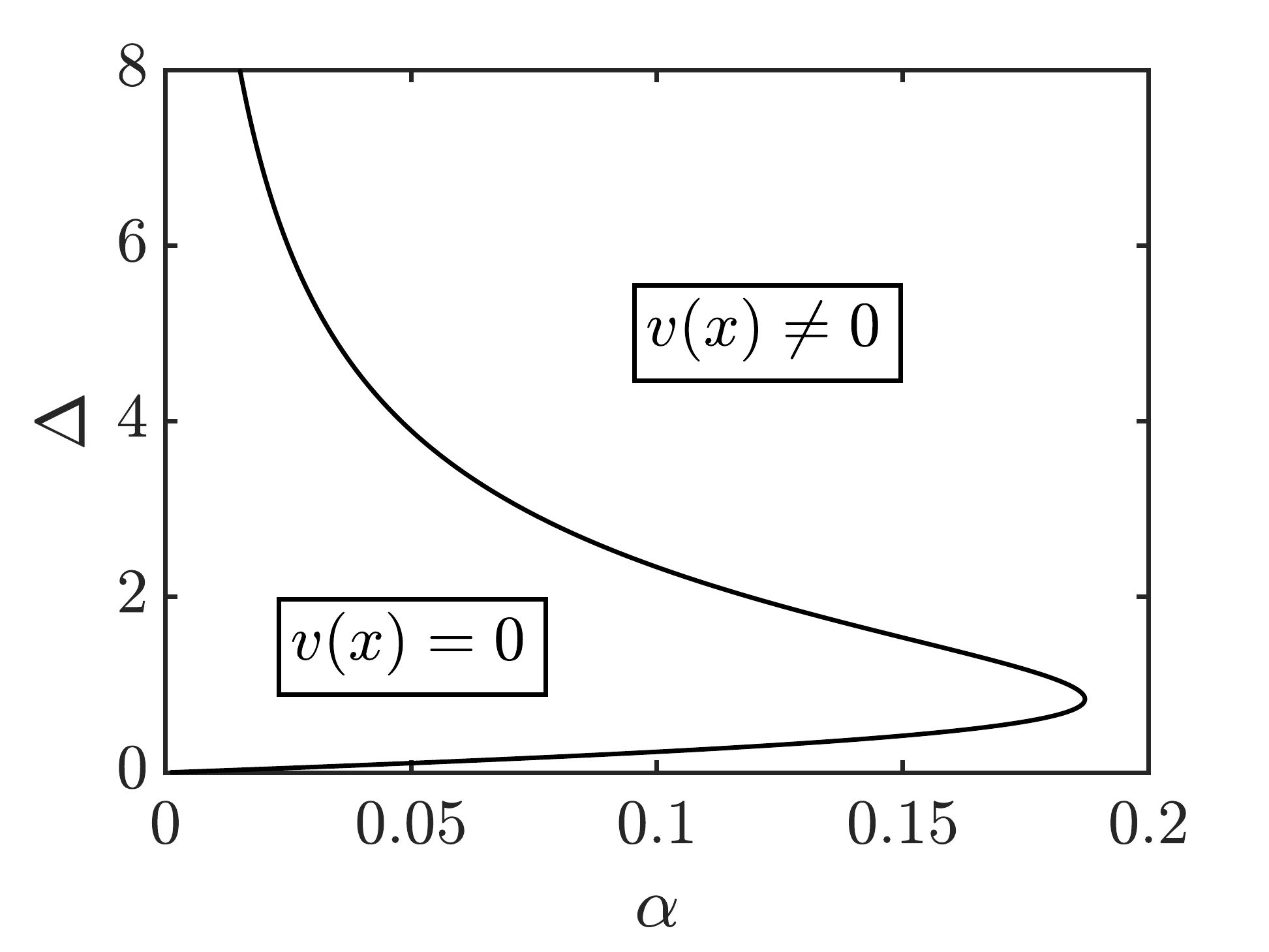}
	\caption{The locus of the pitchfork bifurcation in the $(\alpha,\Delta)$ plane when $d=1$. The $v(x)=0$ solution exists for all parameter choices but for choices on the right of the curve $v(x)\ne0$ solutions also exist.}
	\label{q=1pitchforklocus}
\end{figure}

\begin{figure}
	\captionsetup[subfigure]{labelformat=empty}
	\centering
	\subfloat[]{\includegraphics[scale=0.4]{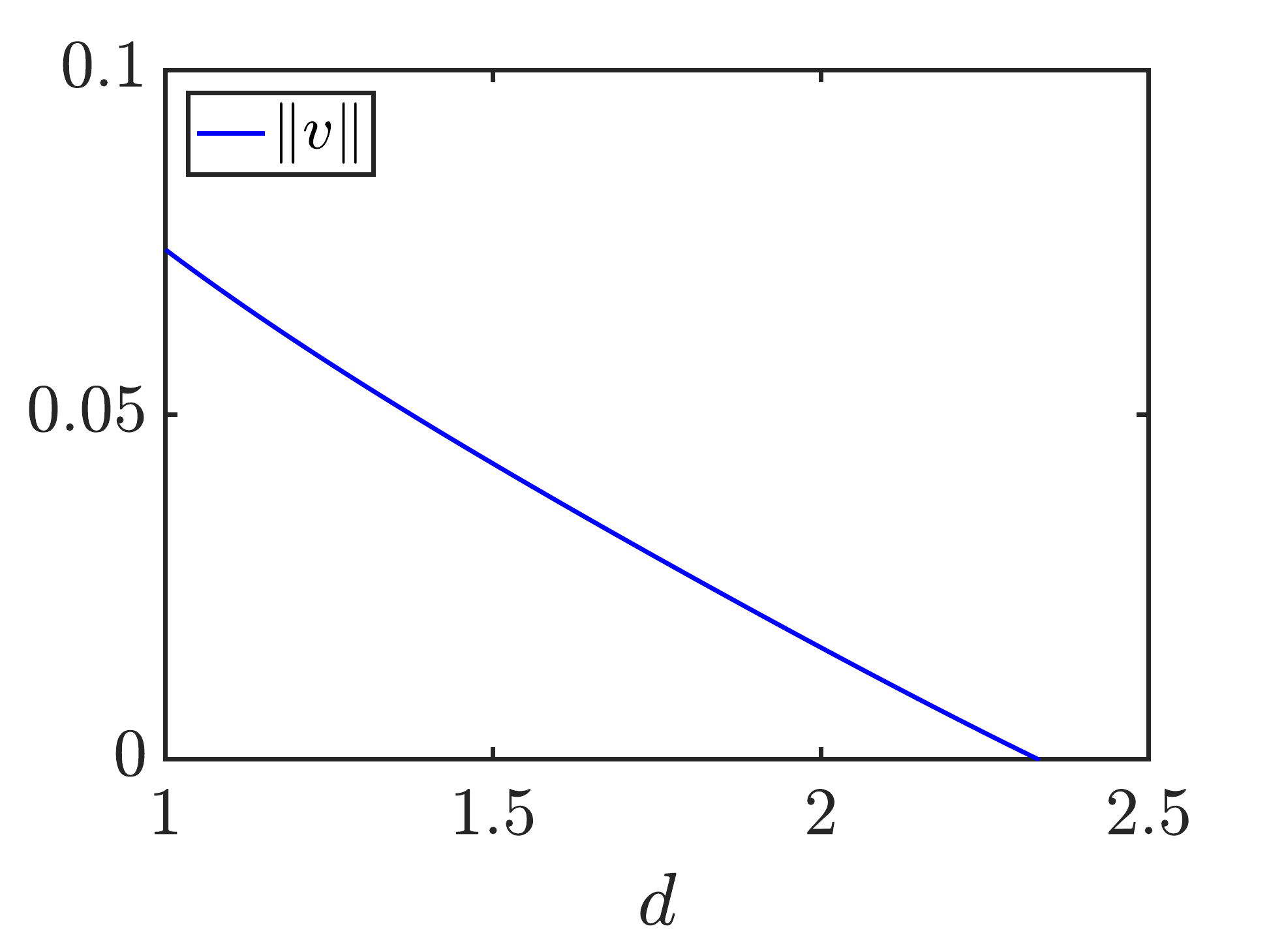}} 
	\subfloat[]{\includegraphics[scale=0.4]{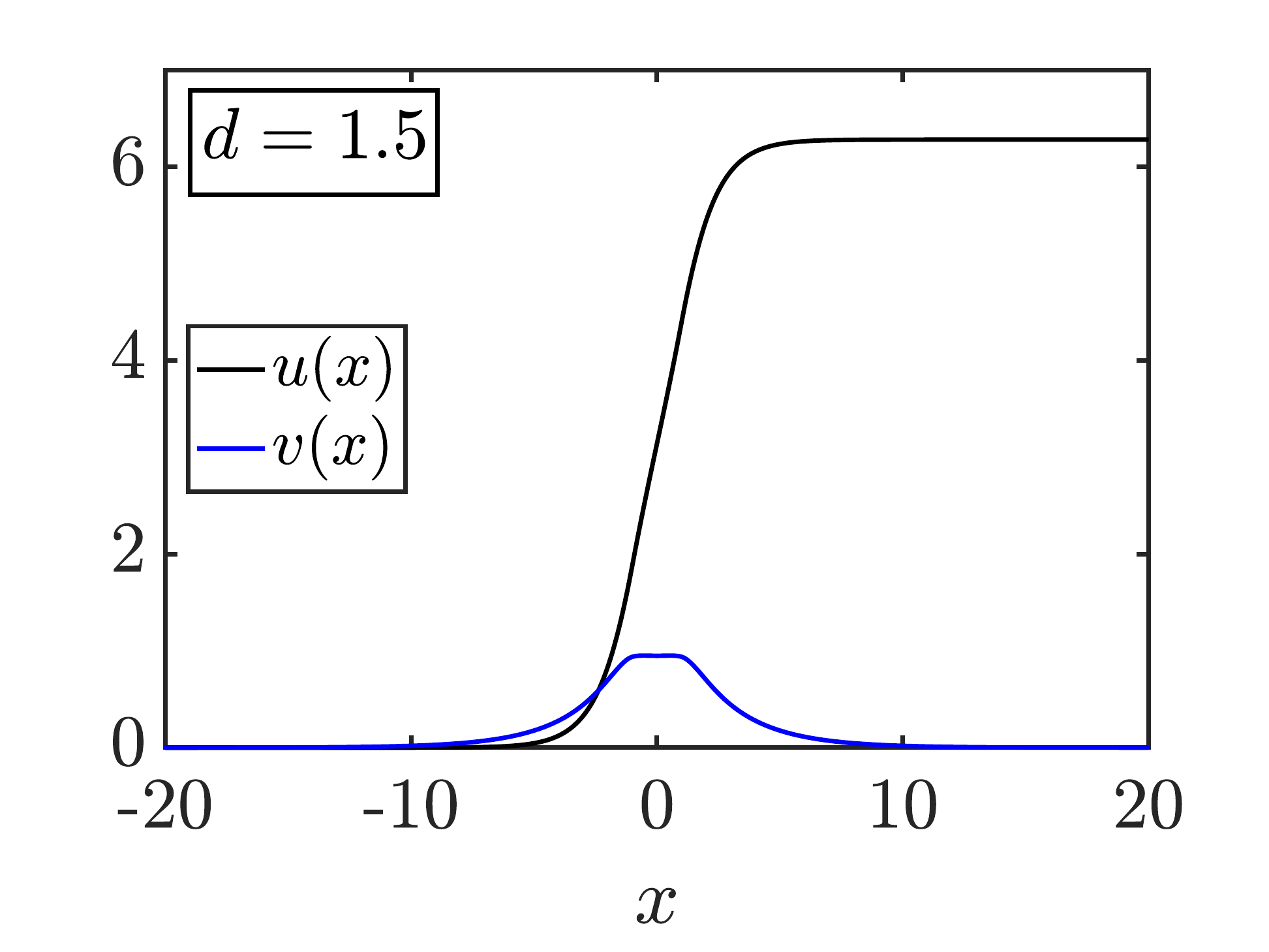}}\vspace{-14pt}\\
	\subfloat[]{\includegraphics[scale=0.4]{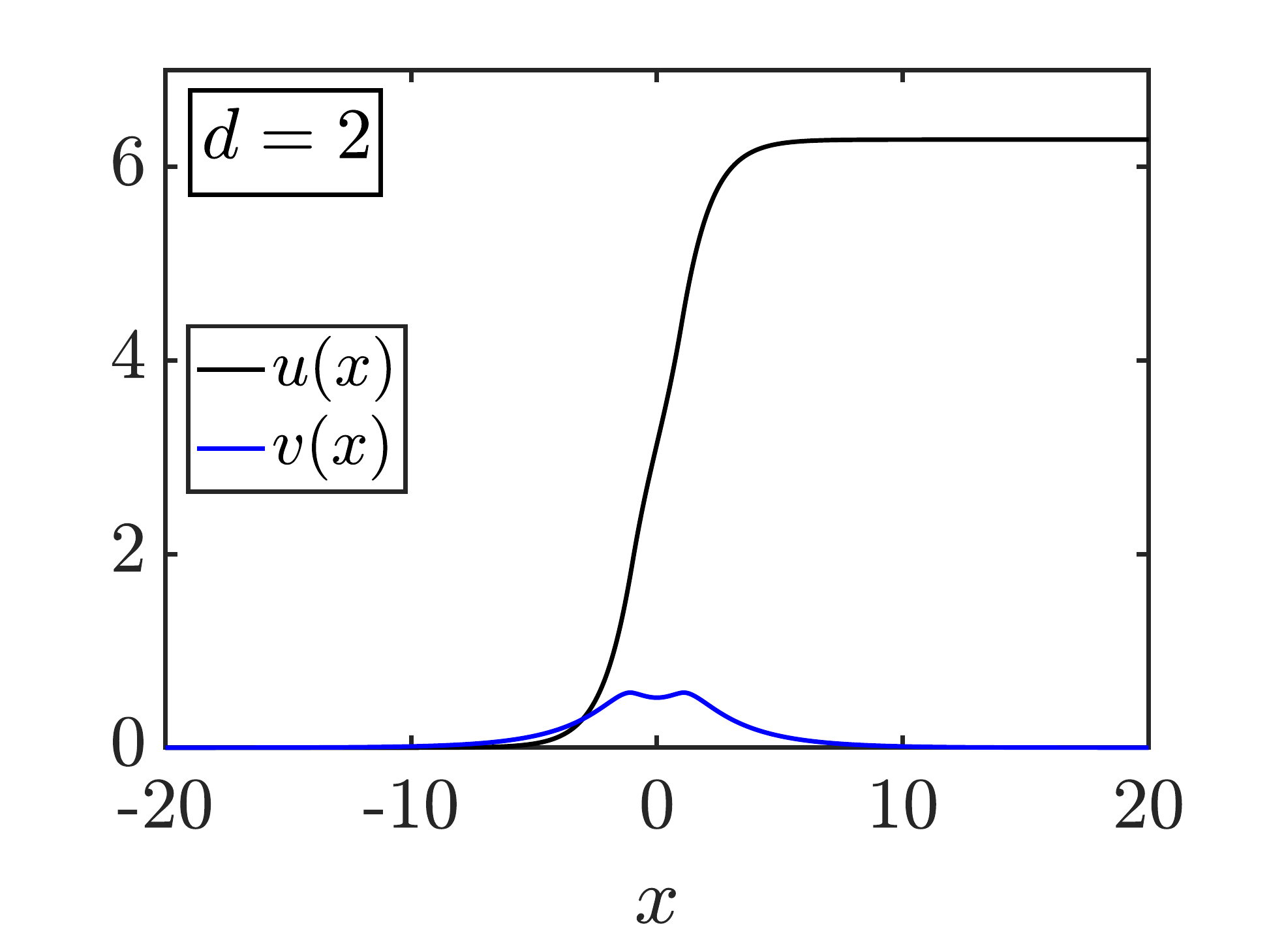}} 
	\subfloat[]{\includegraphics[scale=0.4]{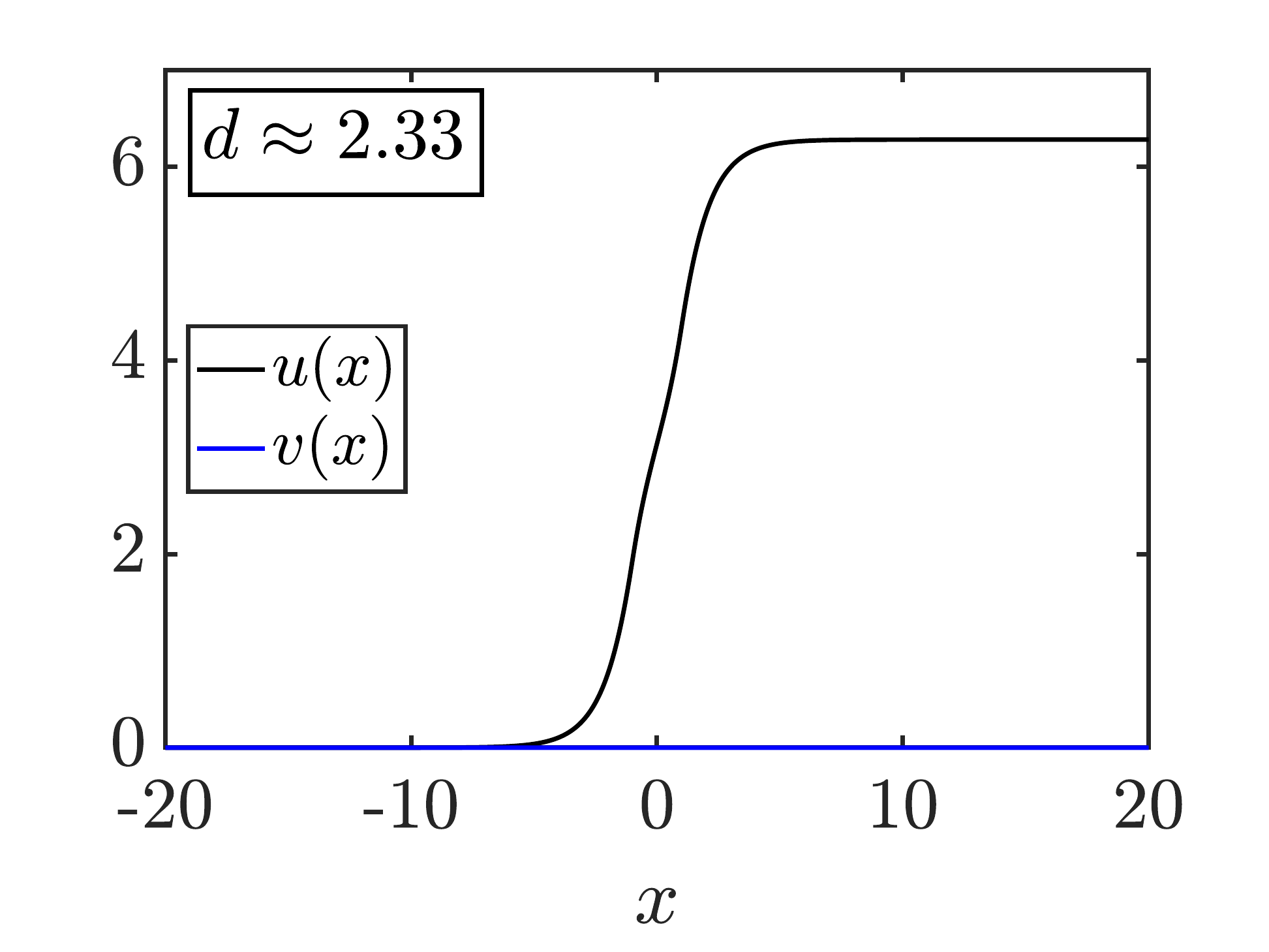}}\vspace{-14pt}\\
		\caption{The top left panel shows the evolution of the $L^2$
		norm of non-zero $v(x)$ component for $d \in [1,2.33]$ where
		$\alpha=0.4$ and $\Delta=1.$ The other panels show snapshots
		$u$ and $v$ components of the solution,
		for different values of~$d$, of the
		system~(\ref{numericsystem}). This illustrates that the~$v$
		component shrinks as~$d$ increases.}
	\label{frontdecay}
\end{figure} 

\subsection[]{The bifurcation for any $\bm{d >0}$}
We now show the pitchfork bifurcation occurs for any $d>0$. Consider
the solution with $v(x)\ne 0$ in~Figure \ref{q=1pitchfork}(c) with
fixed $\alpha=0.4$, $\Delta = 1$, and $d=1$. Increasing the parameter
$d$ results in the decay of the $v(x)$ component of the solution; see
Figure~\ref{frontdecay}. Notice that when $d\ne1$ the $v(x)$ solution
does not always have a bell shape; Figure \ref{frontdecay} shows the
$v(x)$ component developing two maximum points as~$d$
increases. Eventually, at $d\approx 2.33$ the $v(x)$ component
vanishes. This implies that, when $\Delta=1$, the bifurcation point
$\alpha^* \approx 0.185$ observed for $d=1$ increases to
$\alpha^*=0.4$ as $d$ increases to $d\approx
2.33$. Figure~\ref{qne1alphabif}(a) shows this behaviour happens for
all $\Delta$. It shows the bifurcation locus in the
$(\alpha,\Delta)$ plane for $d=2.33, d=1$ and $d=0.5.$ For parameter
choices to the left of the bifurcation loci, only solutions with $v(x)=0$ exist
whilst for parameters selected on the right of the curves there are
also solutions with a non-zero $v$ component. As $d$ increases, the
locus moves rightwards and the solution with $v(x)=0$ becomes
the dominant one since there is less parameter choice for the solution
with $v(x)\ne 0$ to exist. On the other hand, as $d$ decreases the
curve moves leftwards and more solutions with $v(x)\ne 0$ exist. In
particular, for $d<1$, the loci asymptote to $\alpha=0$ for $\Delta \to
\infty$ and are also close to $\alpha=0$ for a large range of $\Delta$ values. In
Figure~\ref{qne1alphabif}(b) we plot the bifurcation branches for the
$v(x)$ component when $d \approx 2.33$ and $\Delta=1$.

\begin{figure}
	\captionsetup[subfigure]{labelformat=empty}
	\centering
	\subfloat[]{\includegraphics[scale=0.4]{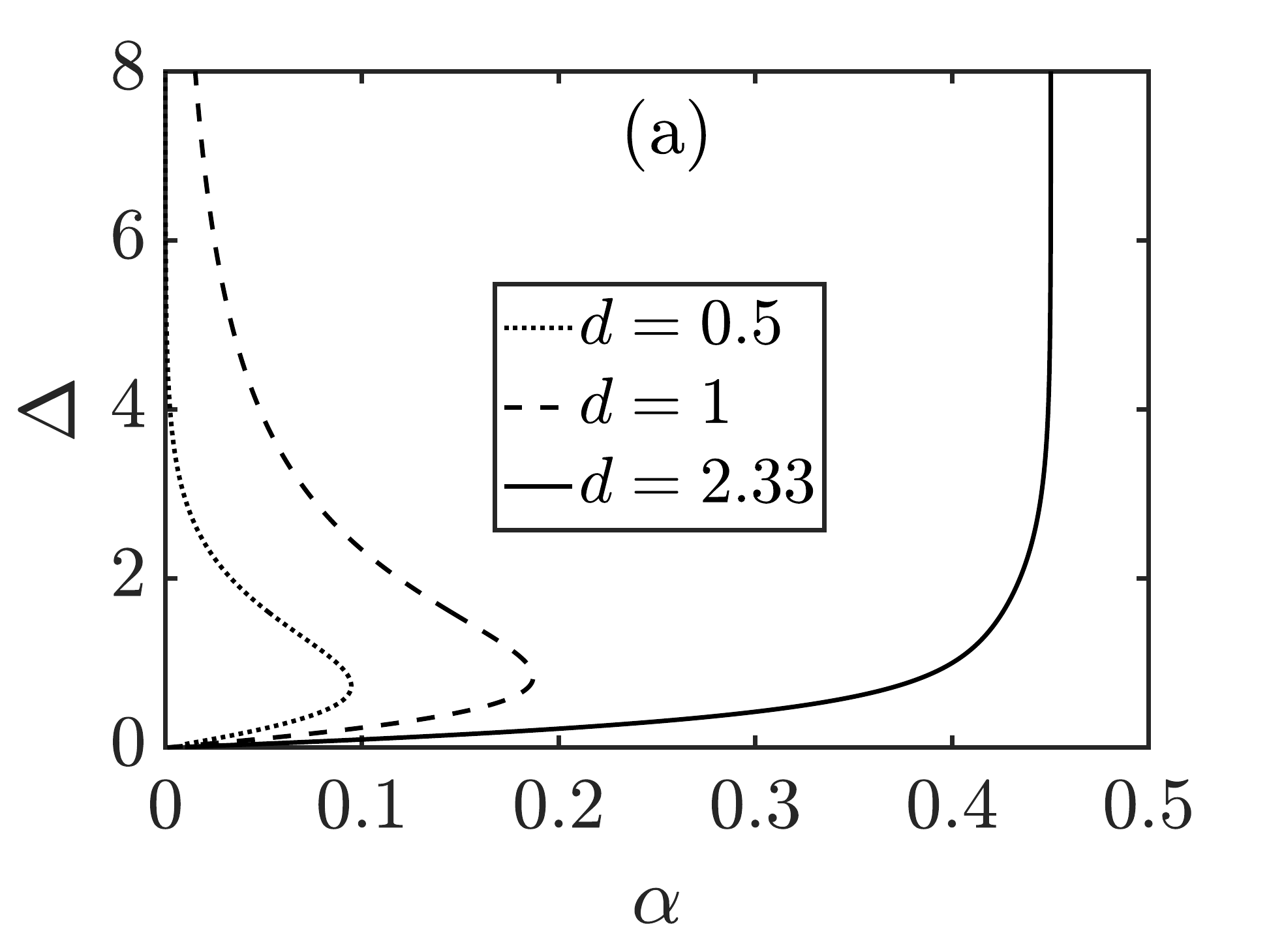}} 
	\subfloat[]{\includegraphics[scale=0.4]{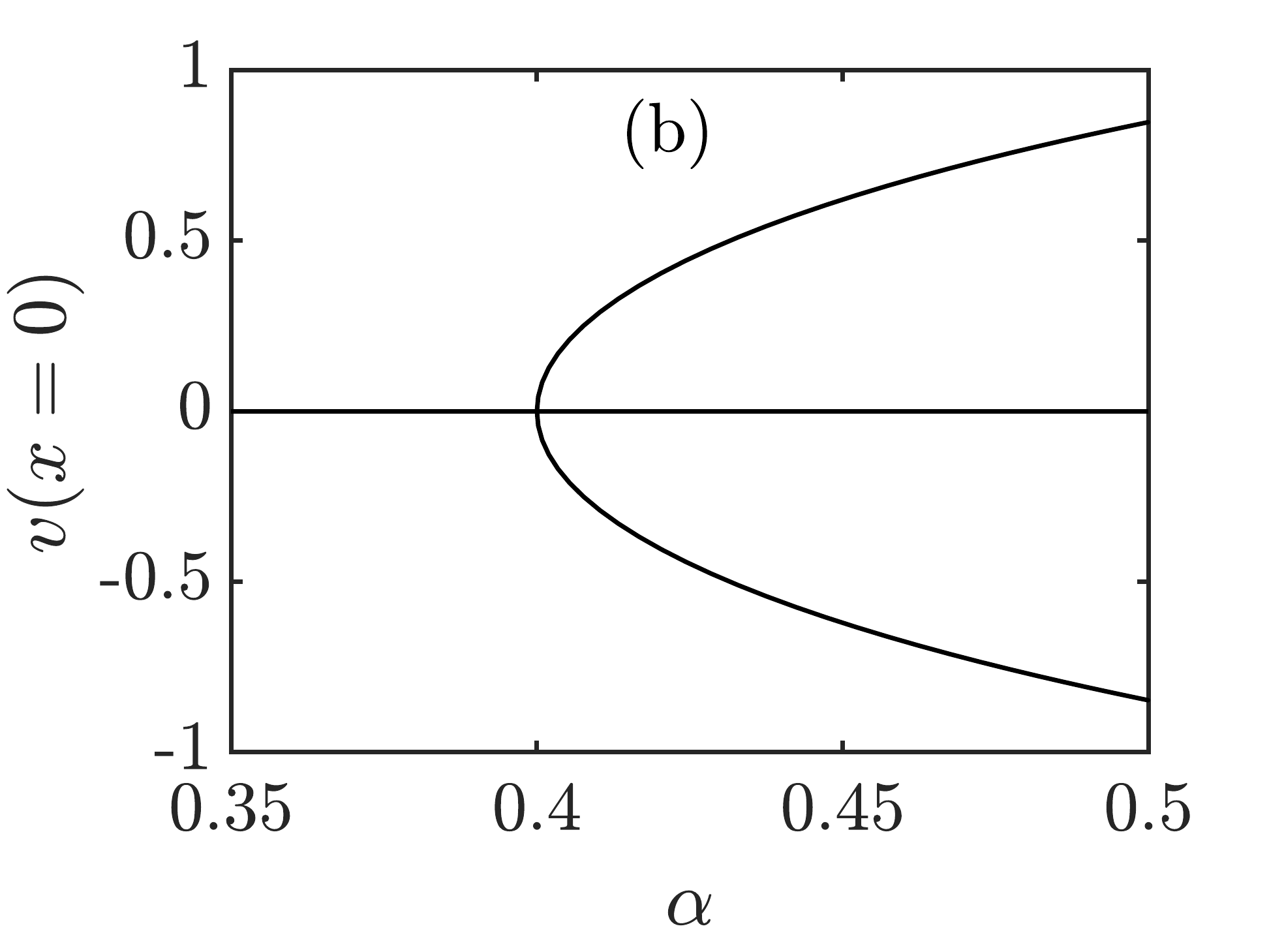}}\vspace{-14pt}\\
	\caption{Panel~(a) corresponds to the bifurcation locus in the
          $(\alpha,\Delta)$ plane when $d=0.5$, $1$, and $2.33$. To the right of each
          curve $v(x) \ne 0$ solutions exist.
          Panel~(b) shows the bifurcation branches and the evolution
          of the $v(x)$ solution when $\Delta=1$ and $d=2.33$.}
	\label{qne1alphabif}
\end{figure} 

Finally we fix $\alpha$ and consider the bifurcation locus in the
$(d,\Delta)$ plane. In the main panel of Figure~\ref{qvsdelta} we
trace the pitchfork bifurcation locus in the $(d,\Delta)$ plane for
fixed $\alpha=0.4$. On the curve and in the area to the right
$v(x)=0$.  Meanwhile on the left $v(x) \ne 0$ solutions also
exist. Figures~\ref{qvsdelta}(a-c) give details about the solution
($u(x)$ and $v(x)$ components and $(u,p)$ phase plane, where $p=u_x$,
and eigenfunction for the eigenvalue zero at selected points on the
curve.  Figure~\ref{qvsdelta}(a) corresponds to a large $\Delta$ value
at the top of the curve ($\Delta=8$). Here we observe that the $u$
front has a plateau around $\pi$ at $x=0$. As one passes to the points
(b) and (c) on the curve in Figure~\ref{qvsdelta} (hence $\Delta$
decreases and $d$ increases), we see that this plateau disappears and
$u$ tends to the unperturbed sine-Gordon front as $d$ becomes large.
%
\begin{figure}
	\captionsetup[subfigure]{labelformat=empty}
	\centering
	\subfloat[]{\includegraphics[scale=0.45]{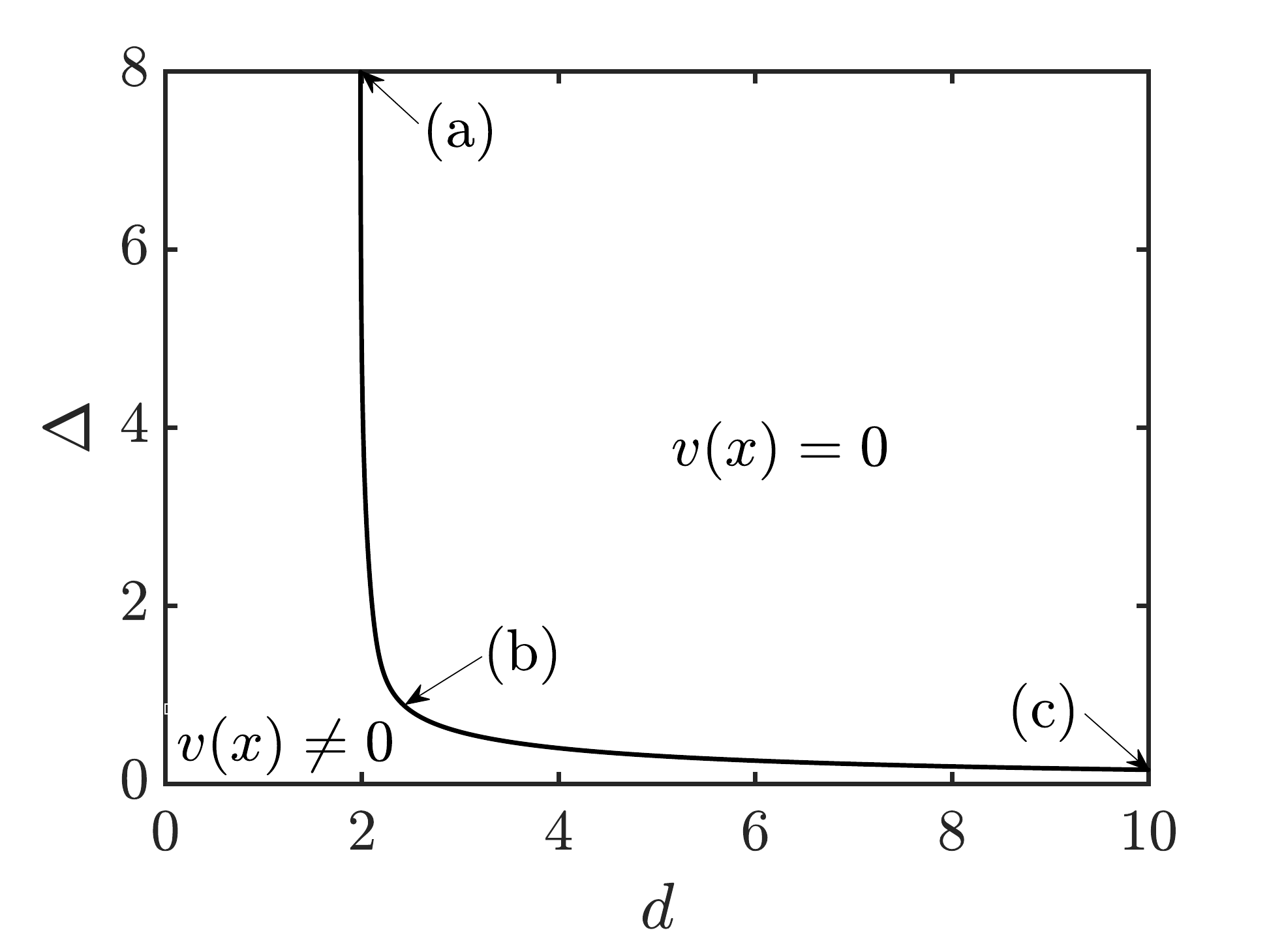}} \vspace{-30pt}\\
	\subfloat[]{\includegraphics[scale=0.26]{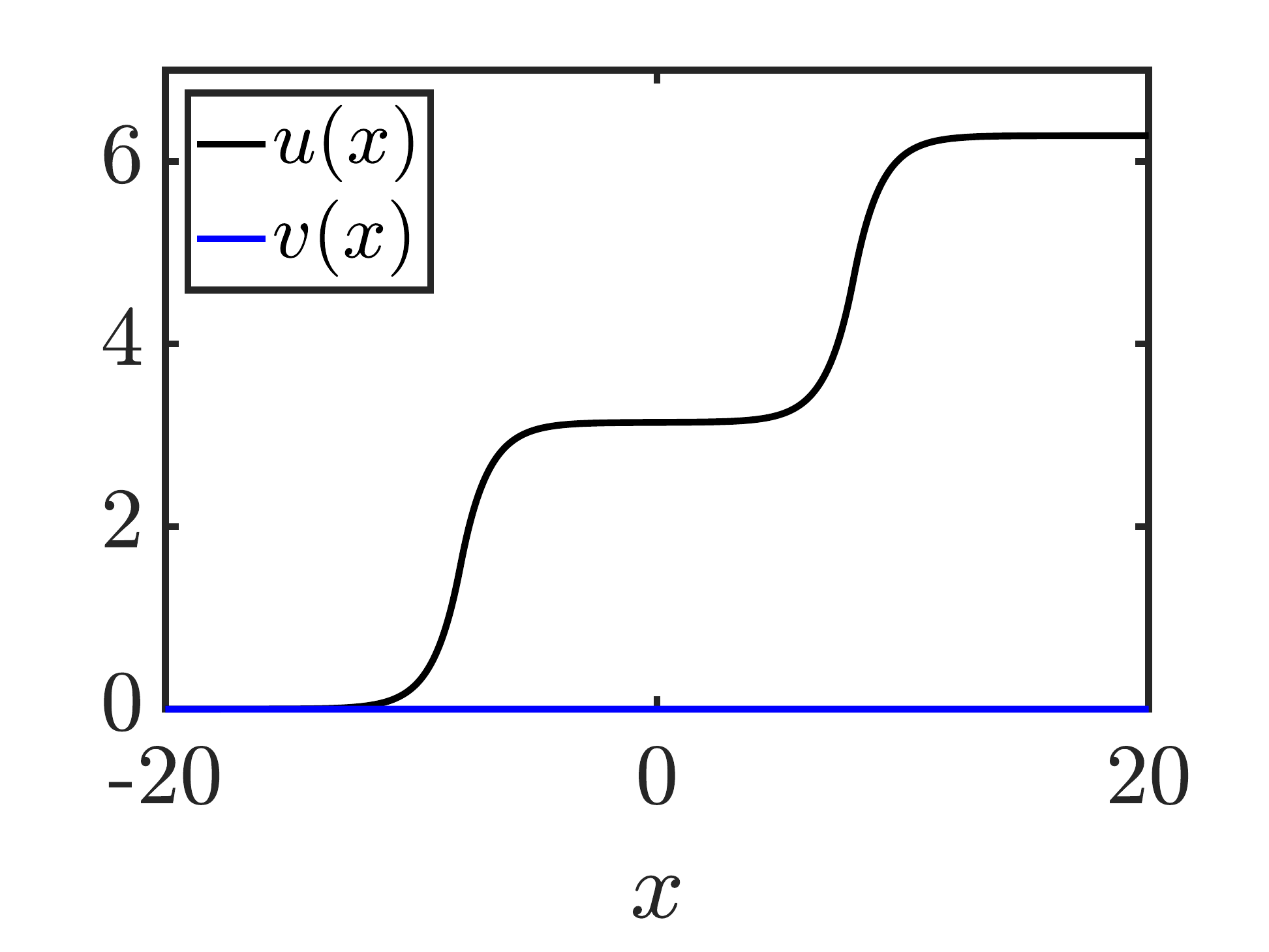}} 
	\subfloat[(a)]{\includegraphics[scale=0.26]{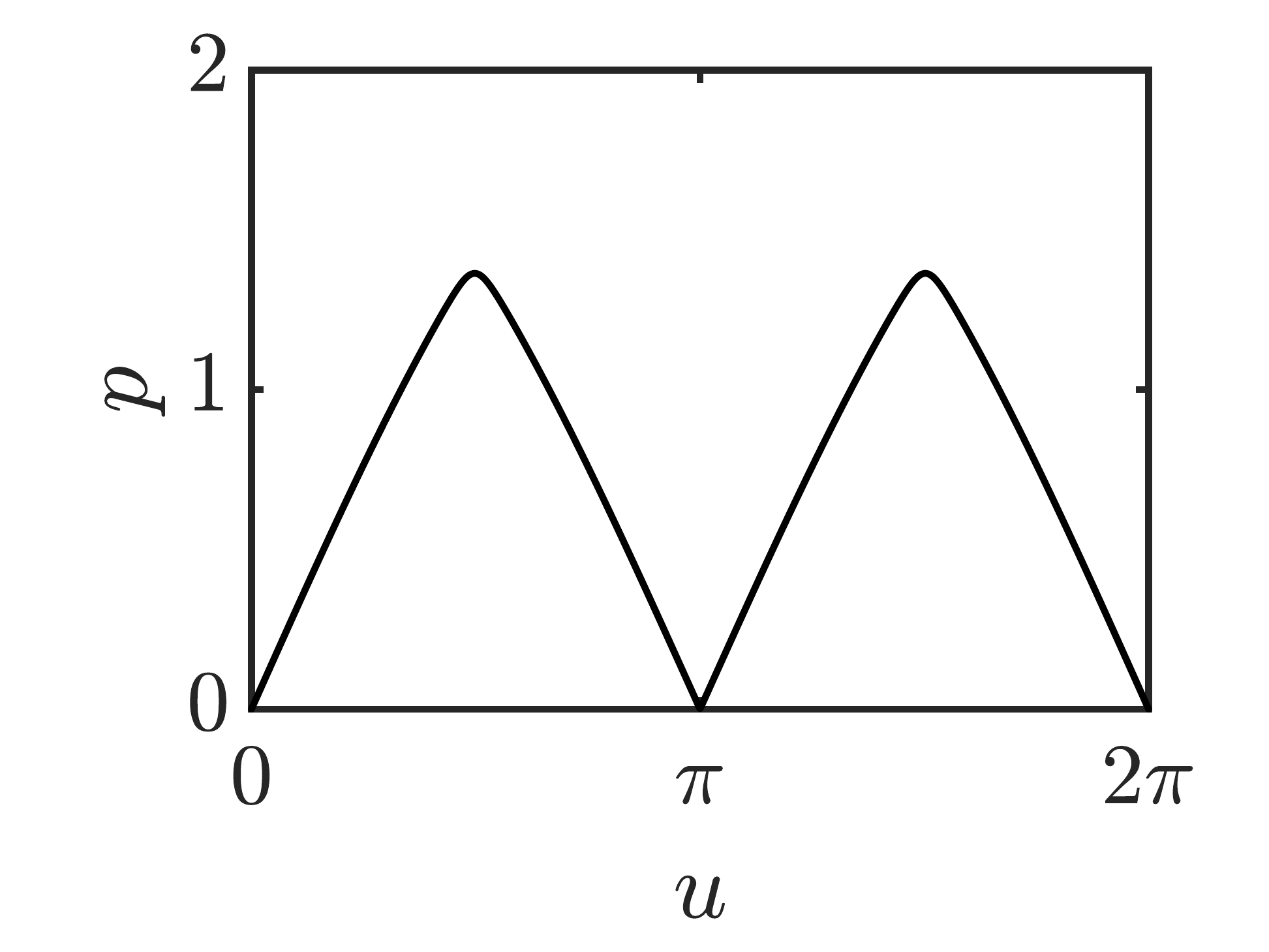}}
	\subfloat[]{\includegraphics[scale=0.26]{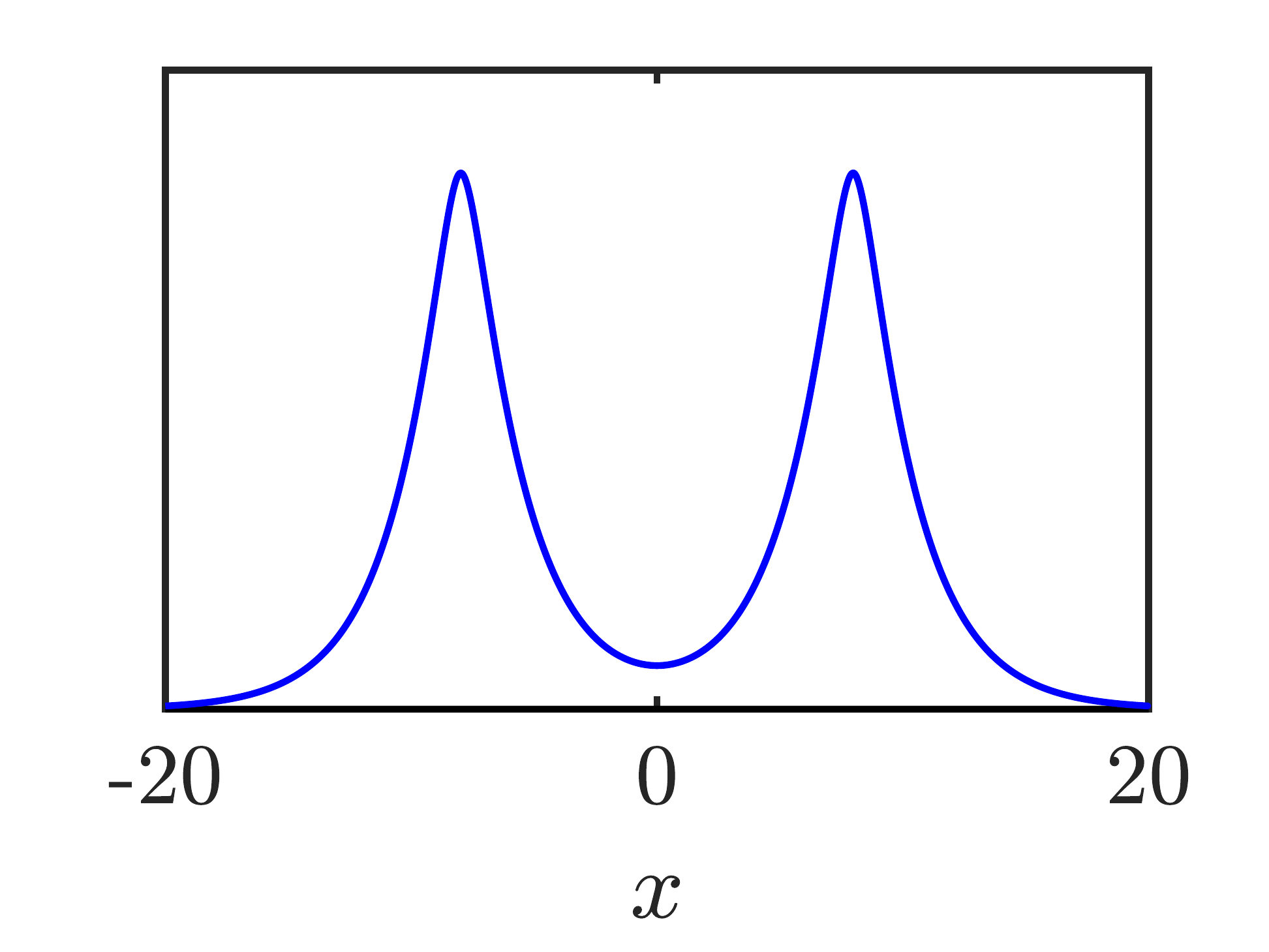}}\vspace{-12pt}\\
	\subfloat[]{\includegraphics[scale=0.26]{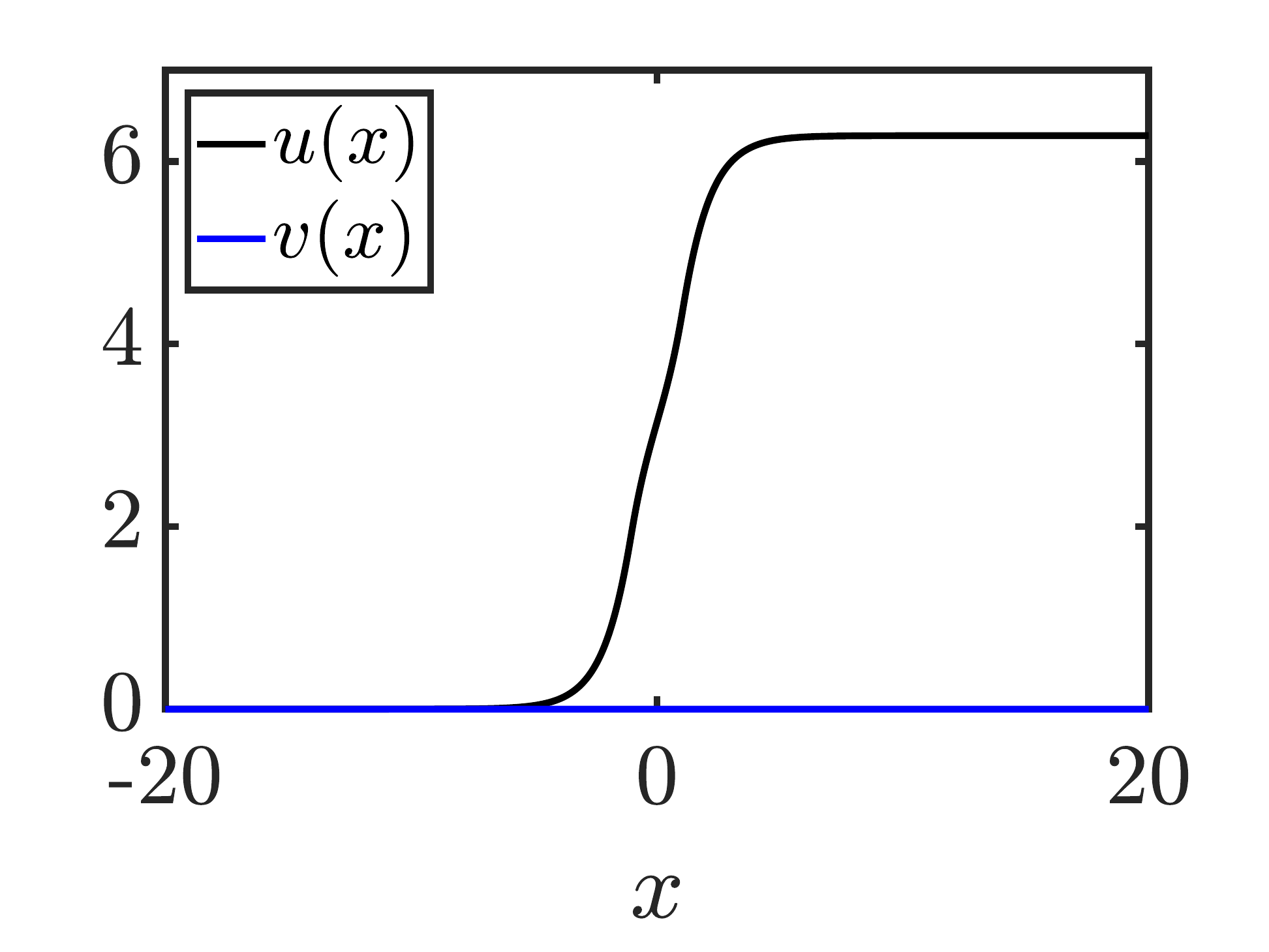}} 
	\subfloat[(b)]{\includegraphics[scale=0.26]{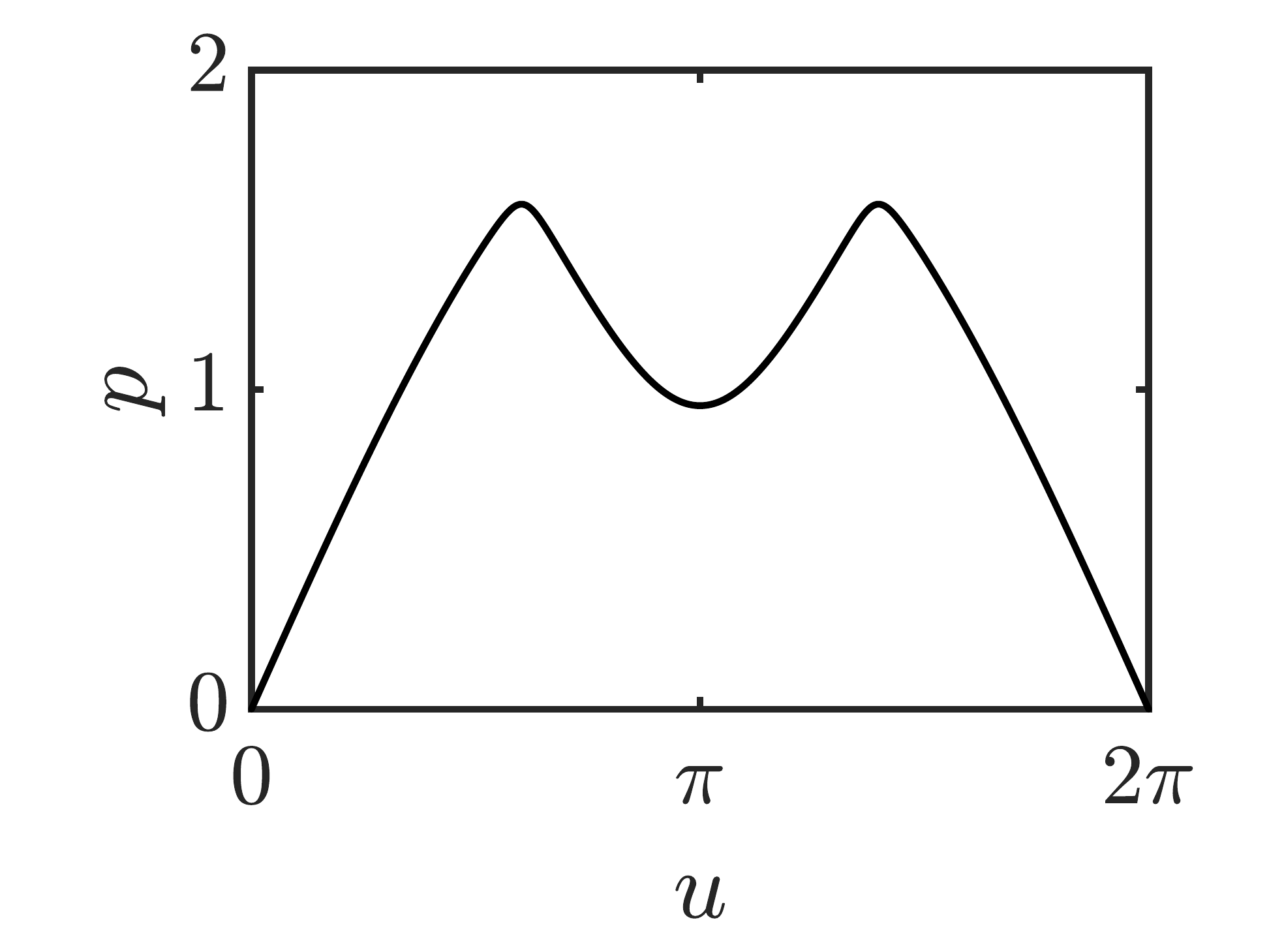}}
	\subfloat[]{\includegraphics[scale=0.26]{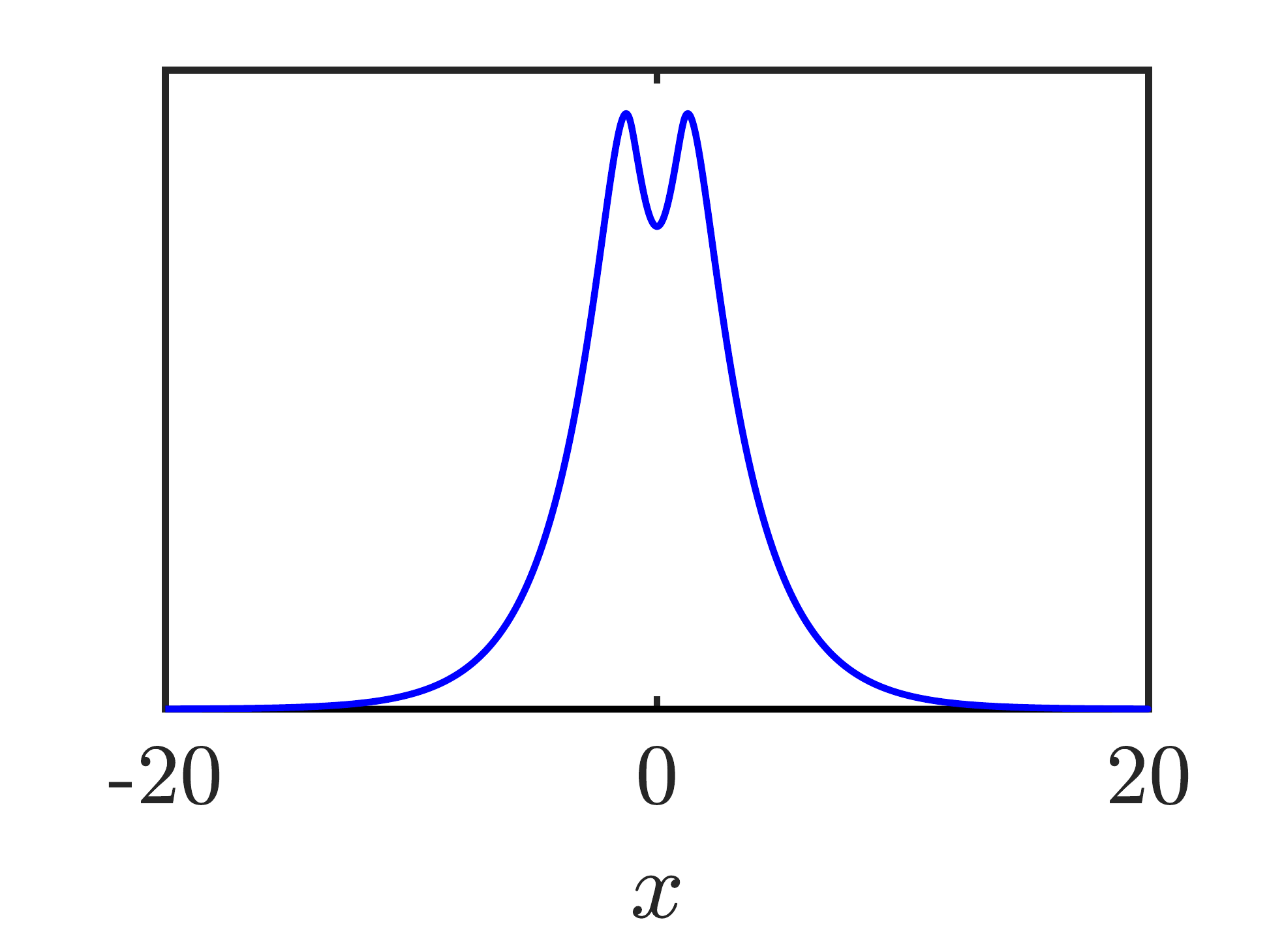}}\vspace{-12pt} \\ 
	\subfloat[]{\includegraphics[scale=0.26]{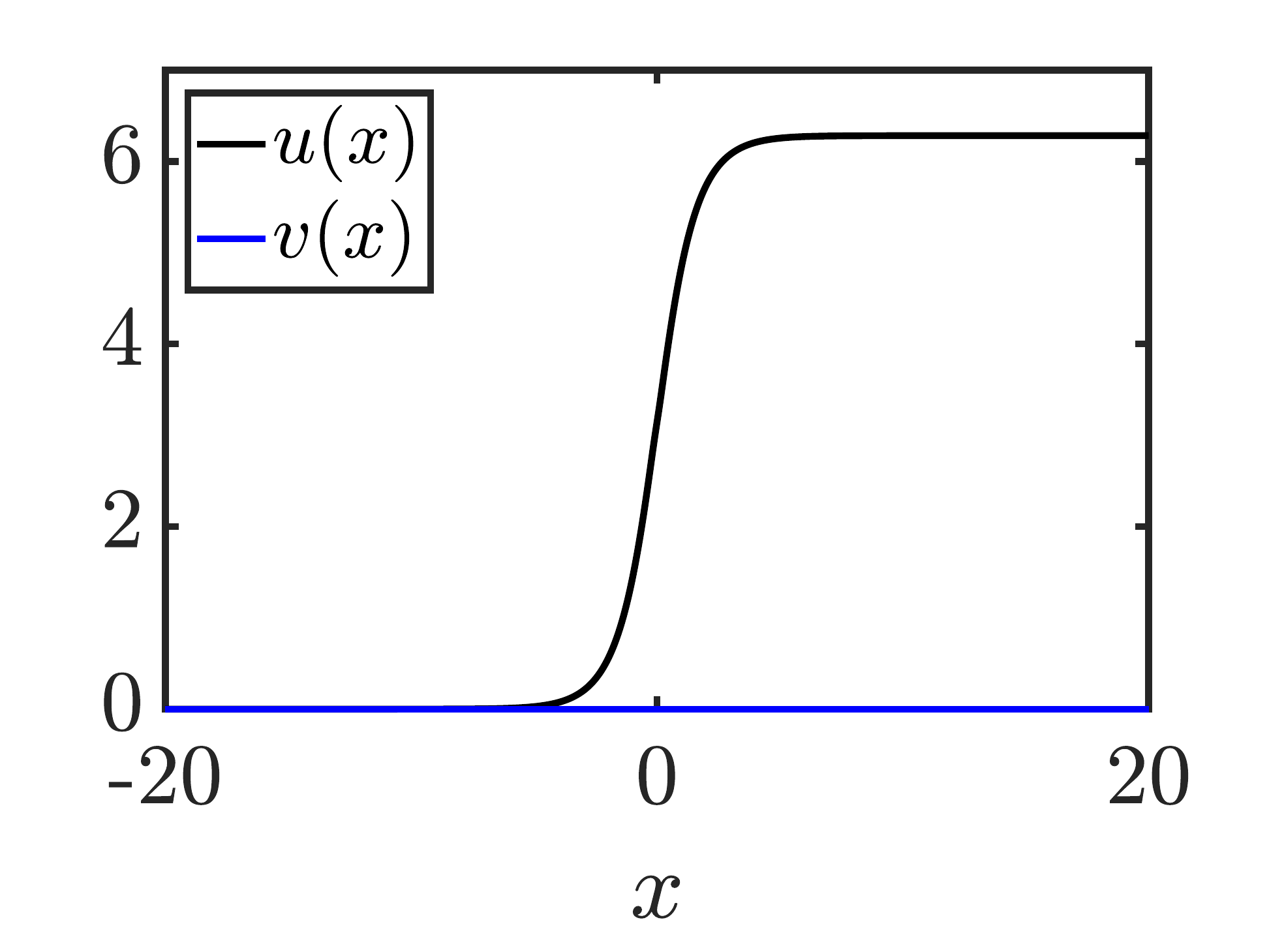}} 
	\subfloat[(c)]{\includegraphics[scale=0.26]{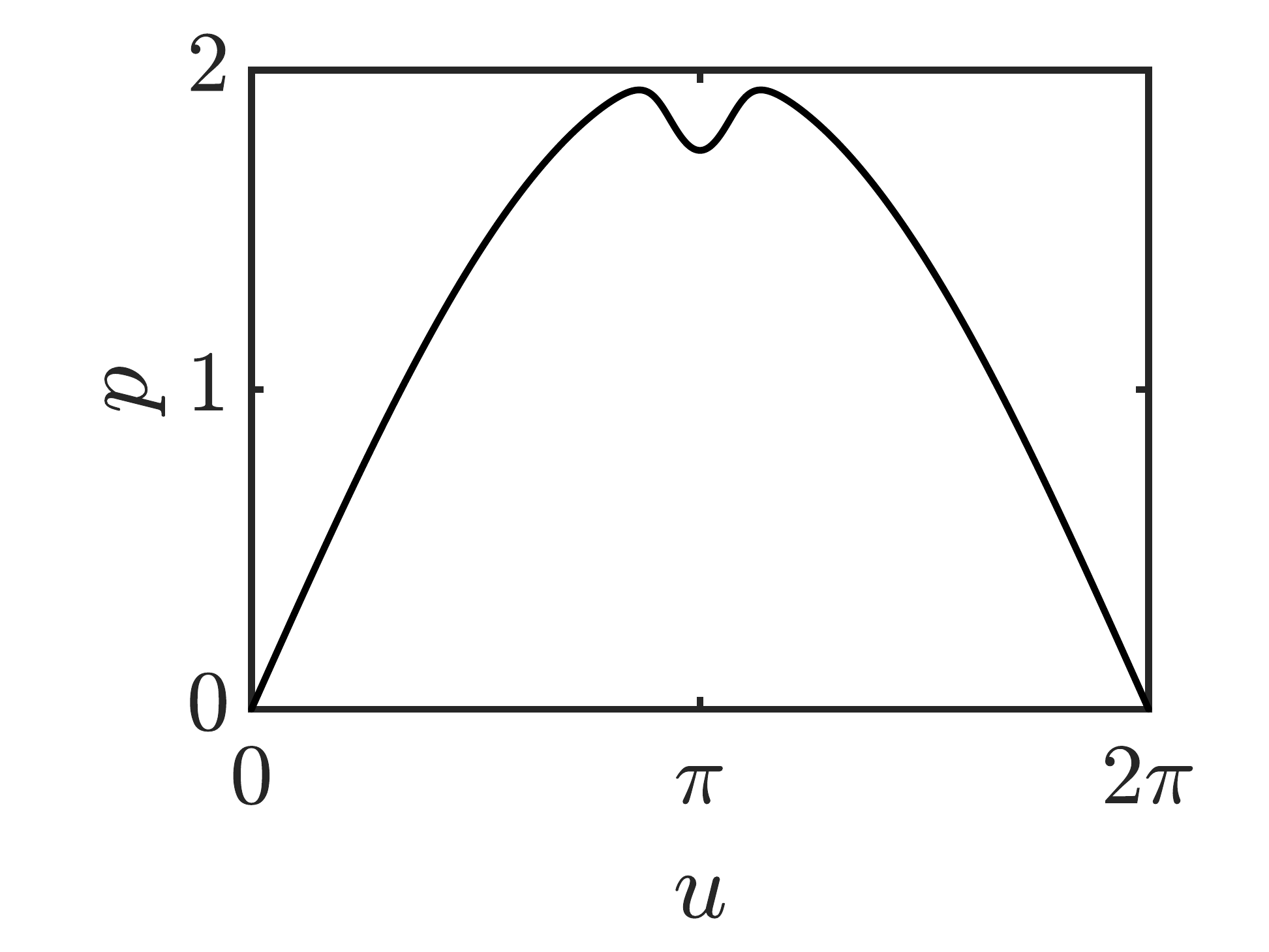}}
	\subfloat[]{\includegraphics[scale=0.26]{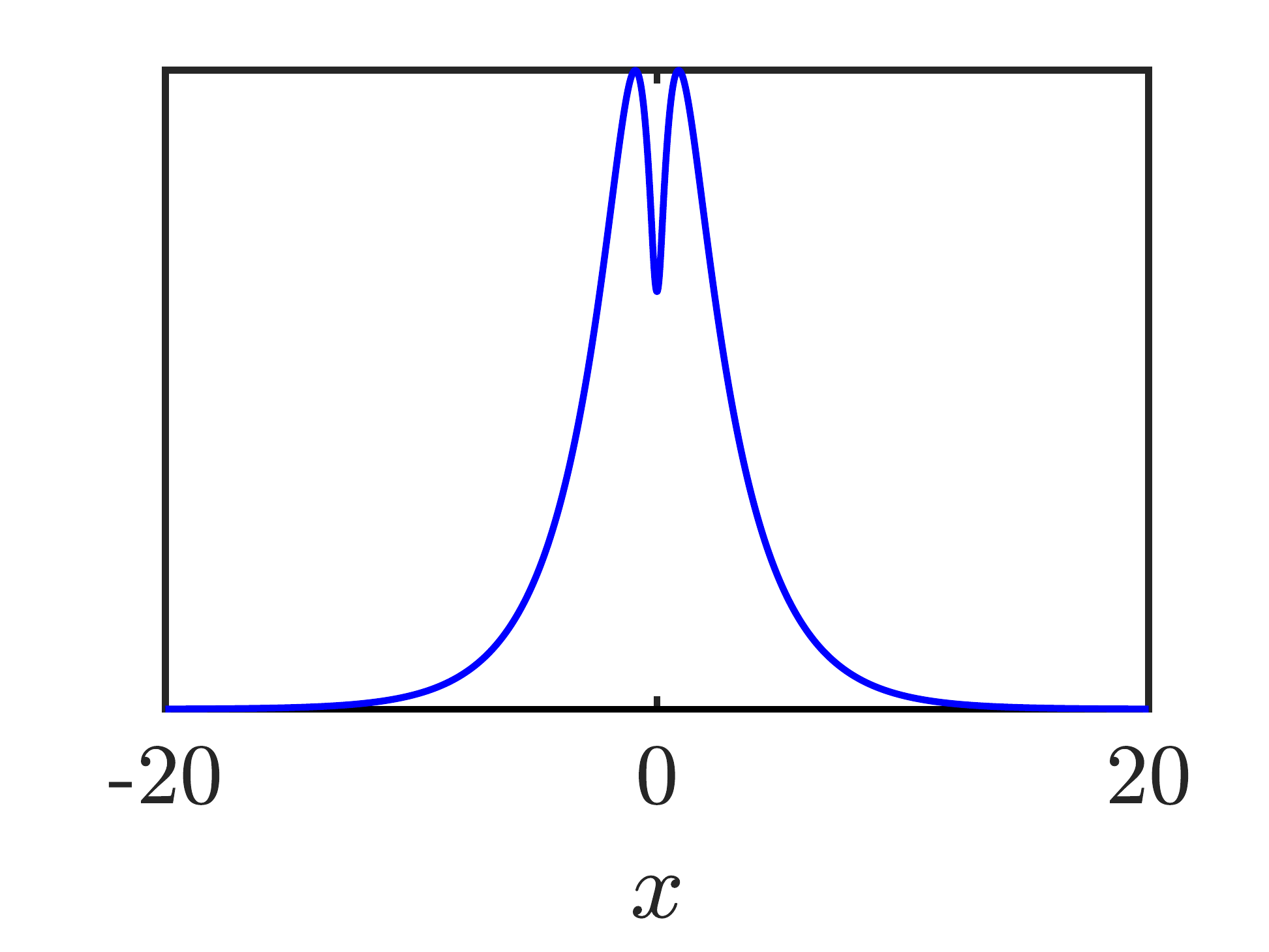}}\\
	\caption{The top panel shows the bifurcation locus for
          $\alpha=0.4$ in the $(d, \Delta)$ plane. Panels (a), (b),
          and (c) give details at the three bifurcation
          points labelled in the top panel. The left column represents the physical
          space whilst the middle is a plot of the $(u,p)$ phase
          space, where $p=u_x$. The
          final column is a plot of the $v$ component of the
          eigenfunction at the eigenvalue zero.}
	\label{qdscas}
	\label{qvsdelta}
      \end{figure}
      
      In Figure \ref{qdeltabifplane}(a), bifurcation loci in the
      $(d,\Delta)$ parameter plane have been plotted for
      $\alpha = 0.01$, $0.1$, $0.25$, $0.4$, and $0.49$. To the right
      of each curve $v(x)=0$ is the only solution whilst to the left
      non-zero solutions for $v(x)$ exist. As $\alpha \to 1/2$, the
      bifurcation locus translates rightwards and more $v(x)\ne 0$
      solutions exist. On the other hand as $\alpha\to 0$ the locus moves leftwards and becomes non-monotonic. However
      for fixed values of $\Delta$, the curves have the property that
      $d\to 0$ as $\alpha \to 0$. Also, for $\alpha$ fixed, the
      bifurcation curves asymptote to $\Delta =0$ for $d\to\infty$ and
      to some $d(\alpha)>1$ for $\Delta\to\infty$, with
      $d(\alpha)\to 1$ for $\alpha\to0$.
      This is illustrated in Figure~\ref{qdeltabifplane}(b) where the bifurcation locus in
      the $(\alpha,d)$ parameter plane is shown for fixed $\Delta=8$ and
      $0.01\leqslant \alpha \leqslant0.49$.

\begin{figure}
	\captionsetup[subfigure]{labelformat=empty}
	\centering
	\subfloat[]{\includegraphics[scale=0.4]{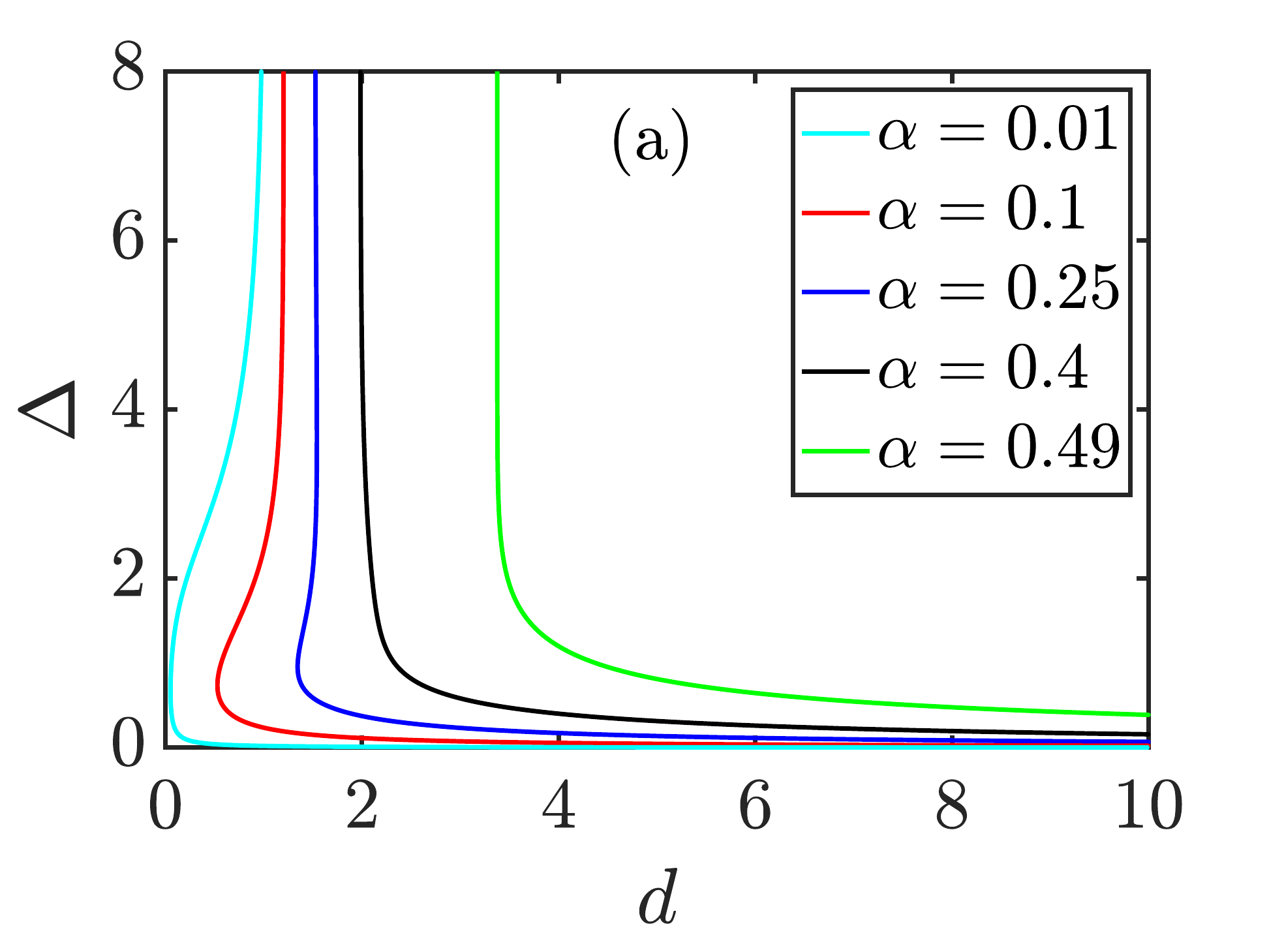}}
	\subfloat[]{\includegraphics[scale=0.4]{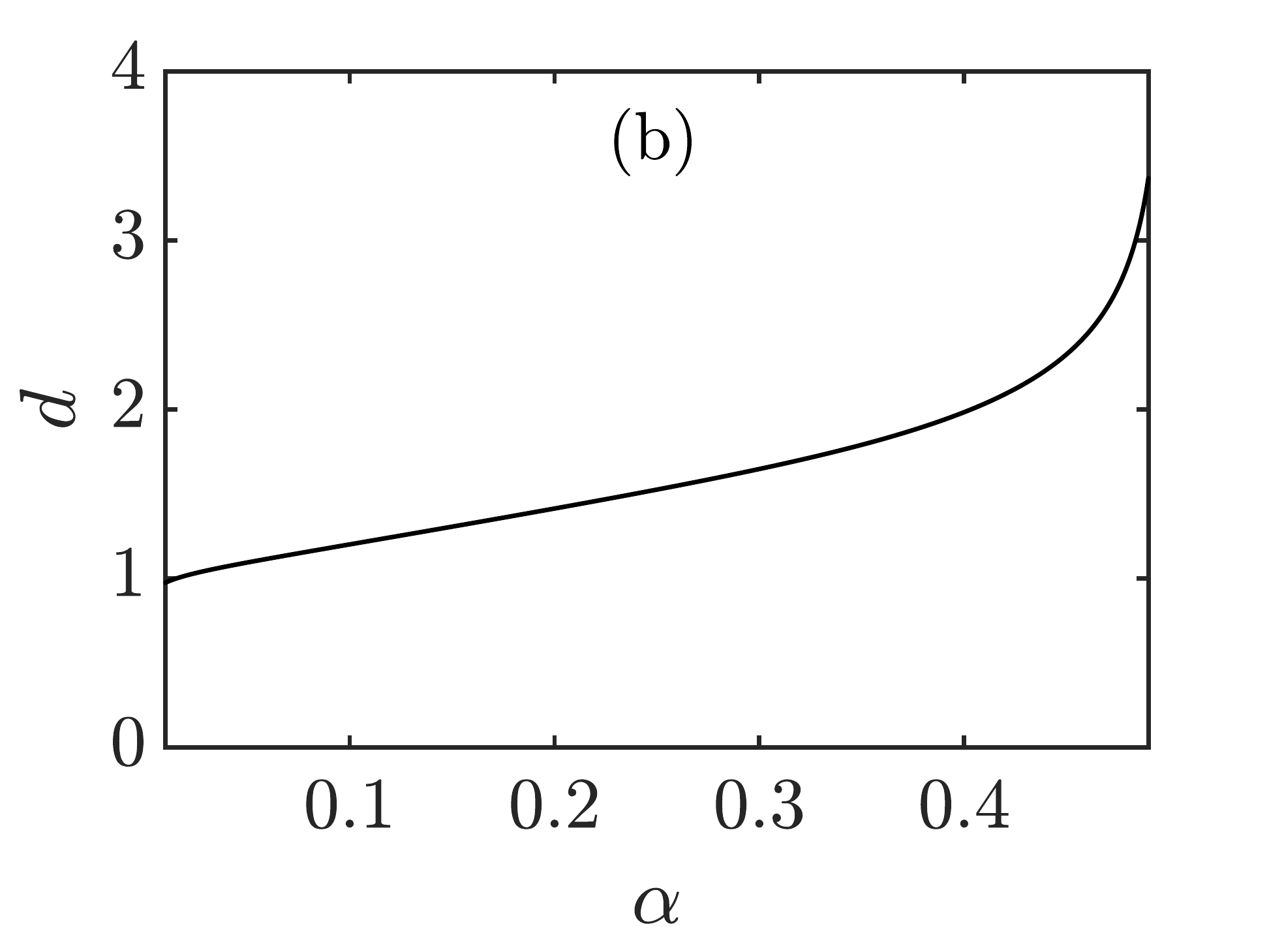}} 
	\vspace{-12pt}
	\caption{(a) shows the bifurcation locus in the
          $(d,\Delta)$ plane for several fixed values of $\alpha$. (b) gives the bifurcation locus in the $(\alpha,d)$
          plane for fixed $\Delta=8$ and $\alpha\in[0.01,0.49]$.}
	\label{qdeltabifplane}
\end{figure}

\end{section}

\newcommand{\gd}[1]{{#1}} 

\begin{section}{Bifurcation manifold analysis} \label{diagramanalysis}

  Upon fixing $0<\delta \ll 1$, the numerical investigation in the
  previous section on the BVP (\ref{alphasys}-\ref{BCS}) suggests that
  there is a single bifurcation manifold in the three parameter space
  $(\alpha, d, \Delta)$ where a pitchfork bifurcation occurs. On this
  manifold the solution state $(u(x),v(x))=(u(x),0)$ bifurcates to one
  where $v(x) \ne 0$.
  
  Using the piecewise constant inhomogeneity $\rho_0(x;\Delta)$ given
  in (\ref{step}) we have the explicit expression~(\ref{q1front}) for
  solutions $(u(x),v(x))=(u_0(x;d,\Delta),0)$ to the
  BVP~(\ref{alphasys}-\ref{BCS}) in the $d=1$ case.  Furthermore, we
  can derive approximations for the solutions $u_0(x;d,\Delta)$ in the
  $d\gg1$ and $\Delta \gg 1$ limits. So in this section we consider
  the piecewise constant inhomogeneity $\rho_0(x;\Delta)$ as an
  approximation of $\rho(x;\Delta,\delta)$. We look for the critical
  parameters in the three parameter space $(\alpha,d,\Delta)$ of the
  BVP
 
	\begin{gather}
	\begin{split}
	u_{xx} &= (1-d\rho_0(x;\Delta))\sin u \cos v, 
	\\
	v_{xx} &= (1-d\rho_0(x;\Delta)) \sin v \cos u -\alpha \sin 2v,	\label{alphasys2}
	\end{split}
	\\
	\lim_{x\to-\infty}(u(x),v(x))=(0,0) \quad \text {and} \quad \lim_{x\to+\infty}(u(x),v(x))=(2\pi,0), \label{BCS2}
	\end{gather}
at which the solution state $(u(x),v(x)) = (u_0(x;d,\Delta),0)$ can bifurcate to
one where $v(x)\ne0$. To be specific, we find the parameter values for
which the linearisation about the state $(u_0(x;d,\Delta),0)$ has an
eigenvalue zero. When $d=1$ we determine an implicit relation between
$\alpha$ and $\Delta$ that characterises the bifurcation locus. When
$d\gg1$ or $\Delta \gg 1$, we obtain approximations of the bifurcation
locus. We give these results in the following Theorem.
\begin{theorem} \label{bifanalysis2}
Consider the BVP~(\ref{alphasys2}-\ref{BCS2}). In the cases below, the
solution $(u(x),v(x))=(u_0(x;d,\Delta),0)$ can bifurcate to a solution with $v\neq0$. 

\underline{\bf{\textit{Case 1:}} $\bm{d=1}$} 

The bifurcation locus
  $\alpha(\Delta)$ is determined implicitly by
\begin{multline}
\frac{h}{2}-\sqrt{1-2\alpha}\bigg(\sqrt{1-2\alpha}+\frac{\sqrt{2(2-h)}}{2}\bigg) \\
=-\sqrt{2\alpha}\bigg(\sqrt{1-2\alpha}+\frac{\sqrt{2(2-h)}}{2}\bigg)
\tan\bigg(\sqrt{\frac{\alpha}{h}}\bigg(\arccos(h-1)\bigg)\bigg)
\label{alphabif}
\end{multline}	
where $0<h<2$ is determined from the one to one relation $\Delta = \arccos (h-1)/\sqrt{2h}$.

\underline{\bf{\textit{Case 2:}} $\bm{d\gg1}$}

For $0<\alpha<1/2$ the bifurcation locus is approximated by
\begin{equation} \label{d>>0}
\Delta(\alpha;d) =
\frac{1}{d}\bigg(\frac{2\alpha}{\sqrt{1-2\alpha}}\bigg)
\bigg(1+\mathcal{O}\bigg(\frac{1}{\sqrt{d}}\bigg)\bigg).
\end{equation}

\underline{\bf{\textit{Case 3:}} $\bm{\Delta \gg 1}$}
\begin{enumerate}[label=(\alph*)]
	
\item When $d>1$ the bifurcation
  locus is approximated by $\alpha(d;\Delta)=\alpha_0(d)
  +\mathcal{O}(e^{-\sqrt{d-1}\Delta})$ with $\alpha_0(d)$  the solution of 
  \begin{multline} \label{BifDeltaLarge}
   \left(\frac{\sqrt{1-2\alpha_0}}{\sqrt{d}}-2\alpha_0+ \frac{1}{d}
    \right)
    \left(\frac{d-1}{\sqrt{d}}+\sqrt{d-1-2\alpha_0}\right) \\
    +\left(\frac{1}{\sqrt{d}}+\sqrt{1-2\alpha_0}\right)
    \left(\frac{d-1}{\sqrt{d}}\sqrt{d-1-2\alpha_0}-{2\alpha_0}
        +\frac{(d-1)^2}{d}\right)=0.
\end{multline}
This implies that $0\leqslant \alpha_0(d)<\frac12\min(d-1,1)$ and
$\alpha_0(d)=\frac{d-1}2 +\mathcal{O}((d-1)^2)$ for
$d\downarrow 1$.

\item When $0<d<1$ the bifurcation locus satisfies
$\alpha(d;\Delta) = o(1)$ for $\Delta \to \infty$. 
\end{enumerate}
\end{theorem}

The bifurcation locus $\alpha(\Delta)$ in case 1 can not be distinguished from the numerics shown in
Figure~\ref{q=1pitchforklocus}. The approximations of the bifurcation locus in cases $2$ and $3(a)$ are plotted in Figure \ref{bifanalysis} and compared with the numerically computed ones seen in the previous section. Here we see
excellent agreement in the respective limits for $d\to \infty$ and for
$\Delta \to \infty$ with $d>1$.  For fixed $0<d<1$, case $3(b)$ clarifies the numerical results in
Figure~\ref{qne1alphabif}(a) (where $d=0.5$) and the sharp downturn in
the $d(\alpha)$ curve in Figure~\ref{qdeltabifplane}(b).

\begin{figure}
	\captionsetup[subfigure]{labelformat=empty}
	\centering
	\subfloat[]{\includegraphics[scale=0.4]{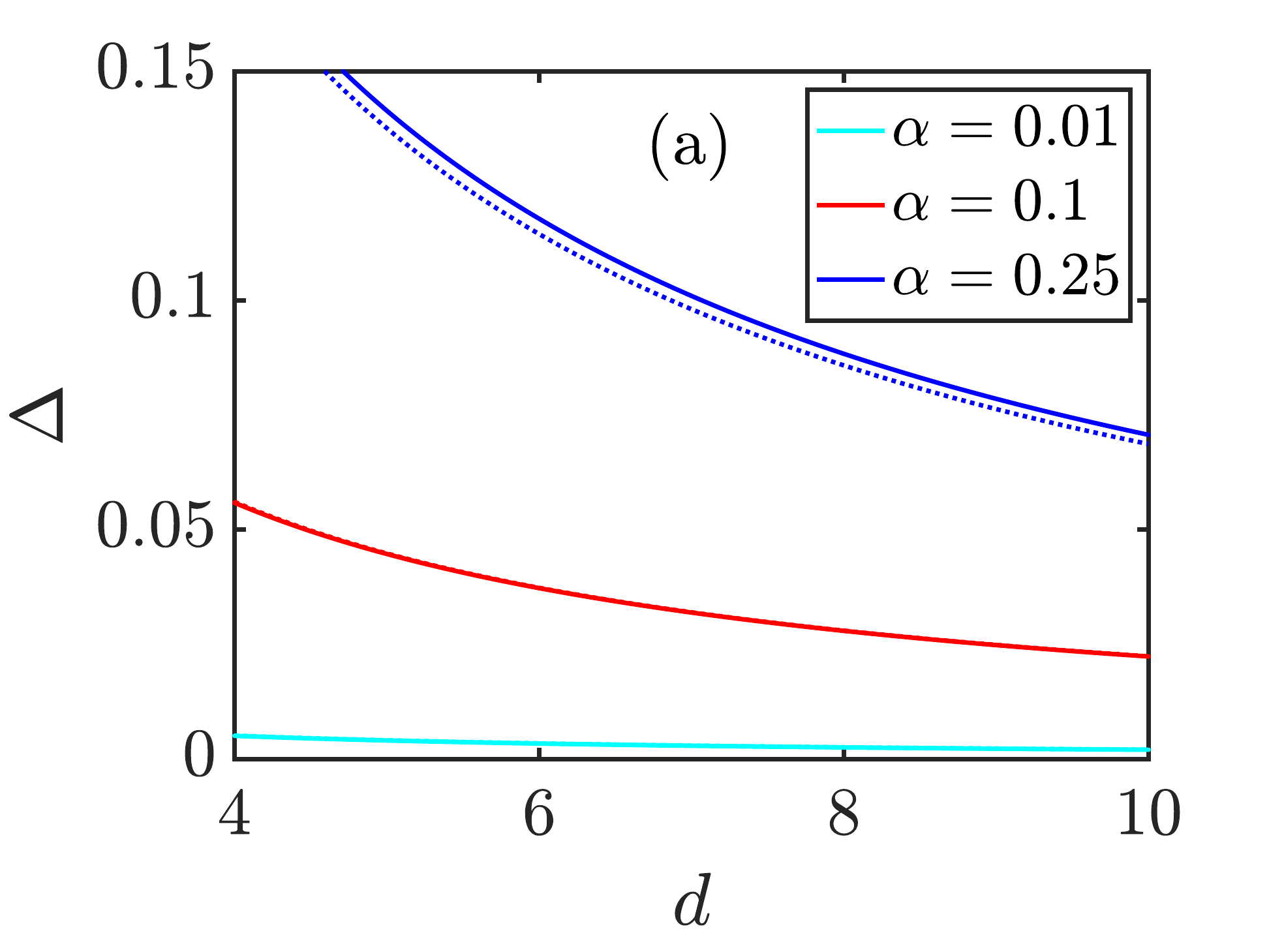}}
	\subfloat[]{\includegraphics[scale=0.4]{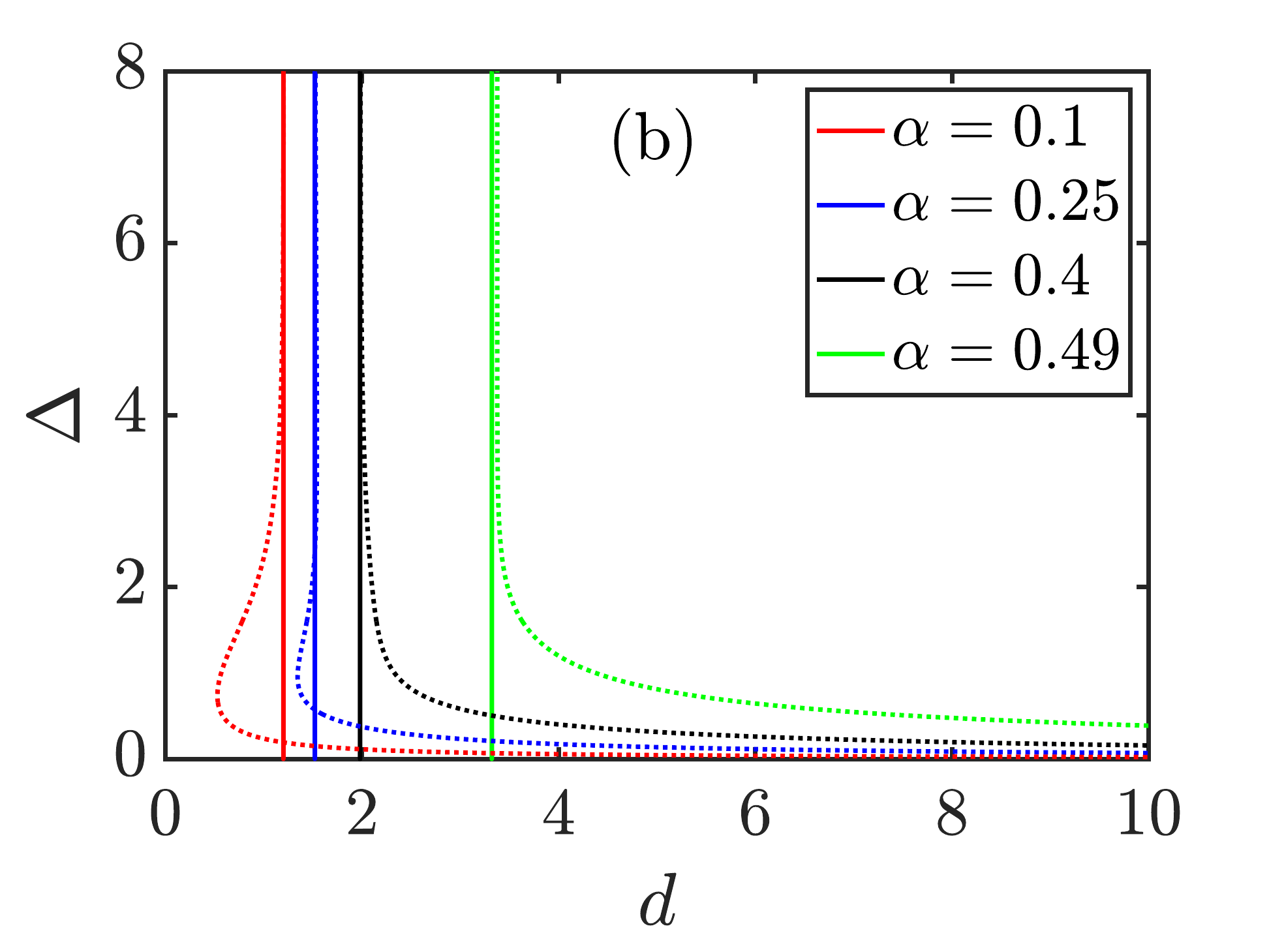}} 
	\vspace{-12pt}
	\caption{The solid curves in (a) shows the approximation (\ref{d>>0}) of the bifurcation locus in the $(d,\Delta)$ plane for fixed $\alpha$ and large $d.$ The solid lines in (b) shows the approximation (\ref{BifDeltaLarge}) of the bifurcation locus for fixed $\alpha$ and large $\Delta.$ In both (a) and (b) the dashed curves correspond to the numerics presented in the previous section.}
	\label{bifanalysis}
\end{figure}

The remainder of this section is spent proving this theorem.  First we
consider the solutions of the BVP (\ref{alphasys2}-\ref{BCS2}) with $v(x)=0$. When $v(x)=0$, the BVP reduces to,
\begin{gather}
  u_{xx} = (1-d\rho_0(x;\Delta))\sin u \label{v=0sg}
  \\
  \lim_{x\to-\infty}u(x)=0 \quad \text {and} \quad
  \lim_{x\to+\infty}u(x)=2\pi. \label{BCS3}
\end{gather}
We call (\ref{v=0sg}) the inhomogeneous sine-Gordon equation. The
existence of fronts that connect $u=0$ to $u=2\pi$ for all
$d\in \mathbb{R}$ and $\Delta>0$ is shown in \cite{Derks2012}. The
construction used in this paper is based on the following idea which is illustrated in Figure \ref{ppd}. Since
the spatial inhomogeneity is piecewise constant, we can
interpret~(\ref{v=0sg}) as a homogeneous Hamiltonian system in both
individual regions $\vert x\vert>\Delta$ and $\vert x \vert <\Delta$.

The Hamiltonian of (\ref{v=0sg}) in the region
$\vert x \vert >\Delta,$ i.e. when $\rho_0=0$, is
\begin{equation}\label{eq:H0}
H_0(u,p) = \frac{1}{2}p^2+\cos(u)-1,
\end{equation}
with $p=u_x$. The Hamiltonian is chosen such
that it vanishes on the heteroclinic connection between $(0,0)$ and
$(2\pi,0)$.  This heteroclinic connection corresponds to the stationary sine-Gordon front, i.e. $u_{ sG}^+$ in (\ref{travelling}) with $c=0$, and is explicitly described by
\begin{equation*}
\left(
\begin{array}{c}
u(x)\\
p(x)
\end{array}\right) = \left(
\begin{array}{c}
4\arctan(e^{x+x^*})\\
2\sech(x+x^*)
\end{array}\right), \quad \text{where} \quad x^* \in \mathbb{R}.
\end{equation*}
The front solution to the inhomogeneous sine-Gordon equation has
to lie on this heteroclinic connection for $|x|>\Delta$, hence the
solution has to satisfy $H_0(u(x),p(x))=0$ for $|x|>\Delta$.  

On the other hand, in the region $\vert x \vert <\Delta,$ i.e. when
$\rho_0=1,$ the Hamiltonian is
\begin{equation}\label{eq:H1}
H_1(u,p) = \frac{1}{2}p^2+(1-d)(1+\cos(u)),
\end{equation}
with $p=u_x$. This Hamiltonian is chosen such that it vanishes on the fixed point
$(u,p) = (\pi,0)$. This fixed point is a saddle point when $d>1$ and a
centre when $d<1.$ A front solution of the inhomogeneous sine-Gordon
equation can be characterised by the value of the Hamiltonian $H_1$ on
the interval $|x|<\Delta$.  Denote this value by~$h$. Then a front
solution $u(x)$ of~\eqref{v=0sg} satisfies
$H_1(u(x),p(x))=h$ with $h>0$ for $|x|<\Delta$.  The relations $H_0=0$
and $H_1=h$ at $x=-\Delta$ give the following matching coordinates
\begin{equation}
u^-(h;d) := u(-\Delta)=\arccos\left(-1+\frac{2-h}{d}\right) \quad \text{and} \quad 
p^-(h;d) := p(-\Delta) = \sqrt{\frac{4(d-1)+2h}{d}}. \label{q>1uin}
\end{equation}
Furthermore we have the following symmetry relations
$u^+ = 2\pi -u^-$ and $p^+ = p^-.$ It can be seen from (\ref{q>1uin})
that both $u^-$ and $p^-$ are increasing in $h$ and 
$(u^-,p^-) \to (\pi,2)$ as $h \to 2$.
Consequently the values of the Hamiltonian~$H_1$ relevant for the construction of a stationary
front are $0 < h \leqslant 2$. Finally, the Hamiltonian $H_1$ can be
used to derive a bijection between the length of the inhomogeneity
$\Delta>0$ and the parameter $h\in(0,2)$, thus $h$ can be considered
as a function of $\Delta$.  

When $d=1$, the non-linearity in (\ref{v=0sg}) vanishes in the region
$\vert x\vert <\Delta$ and the construction can be used to show that
the fronts are given explicitly by (\ref{q1front}). When $d\ne1$, it
is no longer possible to construct explicit fronts without employing
the Jacobi elliptic functions. However, the construction above can be
used to show that the front is close to $\pi$ for all $|x|<\Delta$ when
$d\gg1$ and also that its shape for $|x|<\Delta$ is close to the sine-Gordon front
shape~$u_{\rm sG}^+$ when $\Delta\gg1$.

Next we return to the full BVP (\ref{alphasys2}-\ref{BCS2}).  We set
$\bm{w} = (u,v)$ and hence consider boundary conditions
$\bm{w}(-\infty)=(0,0)$ and $\bm{w}(+\infty)=(2\pi,0)$.  We denote the
front solution of the inhomogeneous sine-Gordon equation as
constructed above by $u_0(x;d,\Delta)$. Then $\bm{w}_0 = (u_0,0)$
solves (\ref{alphasys2}) for all $\alpha \in \mathbb{R}.$ We wish to
determine the bifurcation points in the three parameter space at which
the second component becomes non-zero. Due to the non-zero boundary
conditions, it is convenient to set
\begin{equation*}
  \bm{\tilde{w}} = \bm{w}-\bm{w_0} \in H^2(\mathbb{R}) \times H^2(\mathbb{R}).
\end{equation*}

Now, fixing $d,\Delta>0$, we can define
$\mathbf{{F}}: H^2(\mathbb{R}) \times H^2(\mathbb{R}) \times \mathbb{R} \mapsto L^2(\mathbb{R}) \times
L^2(\mathbb{R})$ where
\begin{equation*}
  \mathbf{{F}}(\bm{\tilde{w}};\alpha) = \left(
    \begin{array}{ c }
      \tilde{u}_{xx} \\
      \tilde{v}_{xx}
    \end{array} \right)
  -
  \left(
    \begin{array}{ c }
      (1-d\rho_{0})(\sin(\tilde{u}+u_{0})\cos(\tilde{v})-\sin(u_{0}))  \\
      (1-d\rho_{0})\sin(\tilde{v})\cos(\tilde{u}+u_{0}) -\alpha \sin(2\tilde{v})
    \end{array} \right).
\end{equation*}
Note that $\mathbf{{F}}(\mathbf{0};\alpha) =\mathbf{0}$ for all
$\alpha \in \mathbb{R}$. A necessary condition for the existence of a
bifurcation locus is that the linearisation of
$\mathbf{F}$ about $\bm{\tilde{w}}= \mathbf{0}$ has an eigenvalue
zero. 

Linearising $\mathbf{{F}}(\bm{\tilde{w}};\alpha) $ about
$\bm{\tilde{w}} = \mathbf{0}$ yields the linear operator
$\tilde{\mathcal{L}}_{\alpha}: H^2(\mathbb{R}) \times H^2(\mathbb{R}) \times \mathbb{R} \to
L^2(\mathbb{R}) \times L^2(\mathbb{R})$ where
\begin{equation}
  \tilde{\mathcal{L}}_{\alpha} = \left(
    \begin{array}{c c}
      \mathcal{L} & 0\\
      0 & \mathcal{L}+2\alpha   
    \end{array}\right)  \label{fulllinear}
\end{equation}
and $\mathcal{L}: H^2(\mathbb{R}) \to L^2(\mathbb{R})$ is given by
\begin{equation} \label{reducedlinear} \mathcal{L} =
  D_{xx}-(1-d\rho_{0})\cos(u_0).
\end{equation} 
We call $\tilde{\Lambda}$ an eigenvalue of $\tilde{\mathcal{L}}_{\alpha}$ if there exists a $\bm{W} \in H^2(\mathbb{R}) \times H^2(\mathbb{R})$ such that $\tilde{\mathcal{L}}_{\alpha}\bm{W} = \tilde{\Lambda}\bm{W}$. Since $\tilde{\mathcal{L}}_{\alpha}$ is a self-adjoint operator all eigenvalues are real. Moreover, $\tilde{{\Lambda}}$ is an eigenvalue of $\tilde{\mathcal{L}}_{\alpha}$ if either:
\begin{enumerate}[label=\roman*)]
	\item $\Lambda = \tilde{{\Lambda}}$ is an eigenvalue of $\mathcal{L}$  with eigenfunction $\Psi \in H^2(\mathbb{R})$. Hence $\tilde{{\Lambda}}$ has associated eigenvector $\bm{W}=(\Psi,0)$,
	\item  $\Lambda = \tilde{{\Lambda}}-2\alpha$ is an eigenvalue of $\mathcal{L}$ with eigenfunction $\Psi \in H^2(\mathbb{R})$. Hence $\tilde{{\Lambda}}$  has associated eigenvector $\bm{W}=(0,\Psi)$.
\end{enumerate} 
The continuous spectrum of $\tilde{\mathcal{L}}_{\alpha}$ is determined by the system at $\pm \infty$ and corresponds to the interval $(-\infty, -1+2\alpha]$.

To proceed with the analysis of the existence of an eigenvalue zero of
$\tilde{\mathcal{L}}_\alpha$, we require more knowledge of $u_0$. As
indicated above, such knowledge can be obtained in the cases $d=1$,
$d\gg1$, and $\Delta \gg 1$ without use of the Jacobi elliptic functions.

\begin{figure}
	\centering
	\includegraphics[scale=0.28]{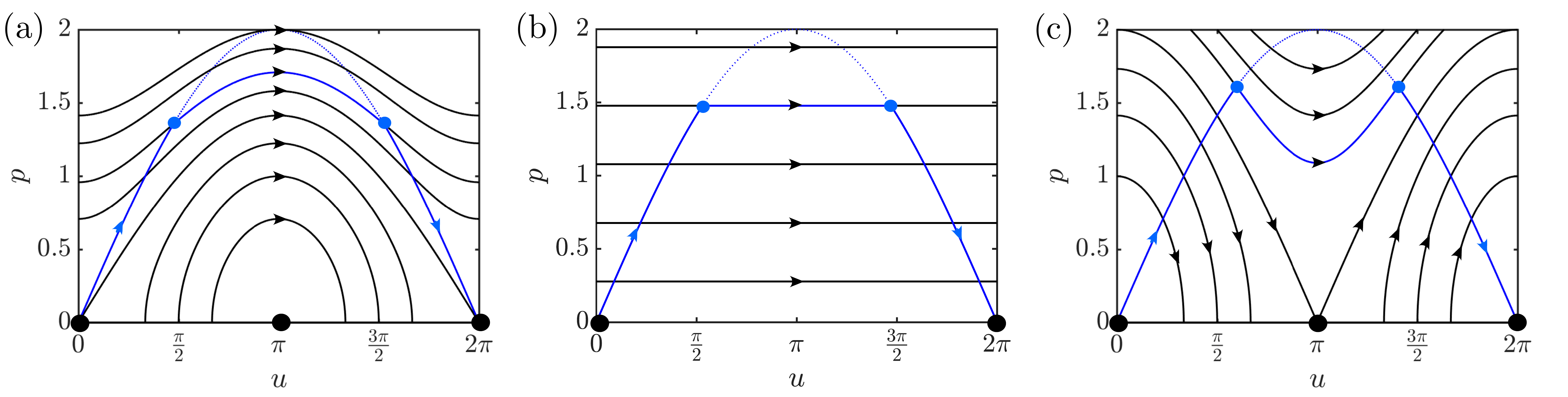}
	\caption{The black trajectories in (a), (b) and (c) correspond to solution curves in the phase plane of $H_1$ (see (\ref{eq:H1})) for $d=0.5$, $d=1$ and $d=2$ respectively. In each panel, the dashed blue curve is the heteroclinic connection in the $H_0$ dynamics (see (\ref{eq:H0})) and the bold blue curve corresponds to front solutions of (\ref{v=0sg}) with $\Delta =1$. Finally the blue points represent the matching points (\ref{q>1uin}).}
	 \label{ppd}
\end{figure}

\subsection*{Case 1: \boldmath{$d=1$}}
When $d=1,$ the BVP (\ref{v=0sg}-\ref{BCS3}) has unique solutions for
all $\Delta>0$ explicitly given by (\ref{q1front}). Thus in this case
the linear operator (\ref{reducedlinear}) becomes
\begin{equation} \label{linearisation}
\mathcal{L} = D_{xx}-(1-\rho_{0})\cos(u_0) = 
\begin{cases}
D_{xx}-\cos(4\arctan(e^{x+x^*})),&  x <-\Delta \\
D_{xx},     &  |x| < \Delta\\
D_{xx}-\cos(4\arctan(e^{x-x^*})),&  x >\Delta,
\end{cases} 
\end{equation}
where $x^*$ is given by~\eqref{eq:Delta_xstar}. This operator is
studied in~\cite{Derks2012} and the following Lemma is proved.
\begin{lemma}[\cite{Derks2012}]
  For fixed $\Delta>0,$ the linear operator (\ref{linearisation}) has
  a largest eigenvalue $\Lambda \in (-1,0)$ given implicitly by the
  largest solution of
	\begin{multline}
          \frac{h}{2}-\sqrt{1+\Lambda}\bigg(\sqrt{1+\Lambda}+\frac{\sqrt{2(2-h)}}{2}\bigg) \\
          =-\sqrt{-\Lambda}\bigg(\sqrt{1+\Lambda}+\frac{\sqrt{2(2-h)}}{2}\bigg)\tan\bigg(\sqrt{\frac{-\Lambda}{2h}}\bigg(\arccos(h-1)\bigg)\bigg) \label{eigenvalueq=1}
	\end{multline}
	where $h(\Delta)\in(0,2)$ is given by the implicit relation
        $\Delta = \arccos (h-1)/\sqrt{2h}$. The eigenvalue has an
        associated eigenfunction $\Psi \in H^2(\mathbb{R})$ given by
	\begin{equation} \label{eigenfunctionq=1}
	\frac{\Psi }{R}= 
	\begin{cases}
	\exp({\sqrt{1+\Lambda}(x+x^{*})})(\tanh(x+x^{*})-\sqrt{1+\Lambda}),&  x < -\Delta \\
	A\cos(\sqrt{-\Lambda}x) ,     &  |x| < \Delta \\
	-\exp({-\sqrt{1+\Lambda}(x-x^{*})})(\tanh(x-x^{*})+\sqrt{1+\Lambda}), & x> \Delta.
	\end{cases} 
	\end{equation}
	In the above, $A$ is a constant found by matching the above at
        either $x=\pm\Delta.$ Furthermore, $R$ is the rescaling
        constant, dependent on $\Delta,$ such that $||\Psi||_{L^2(\mathbb{R})}^2 = 1.$
        Both $A$ and $R$ are given in Appendix~\ref{appB}.
\end{lemma}
This Lemma implies that for fixed values of $\Delta>0$, the operator
$\tilde{\mathcal{L}}_{\alpha}$ has an eigenvalue zero at
$\alpha = -\Lambda/2$ with associated eigenvector $(0,\Psi) \in H^2(\mathbb{R}) \times H^2(\mathbb{R})$. Replacing $\Lambda$ by $-2\alpha$
in~(\ref{eigenvalueq=1}) yields (\ref{alphabif}) which completes the
first part of the proof of Theorem \ref{bifanalysis2}.

\subsection*{Case 2: \boldmath{$d\gg1$}}
Next we seek approximations of front solutions to (\ref{v=0sg}) when
$d\gg1$. It is apparent from~\eqref{q>1uin} that for any $\Delta>0$, i.e. any $h(\Delta) \in (0,2),$ the coordinates $(u^-,p^-) \to (\pi,2)$ as $d\to \infty$. To be more precise, by setting
$\epsilon = 1/\sqrt{d},$ \eqref{q>1uin} implies
$u^- = \pi-\sqrt{2(2-h)}\epsilon+\mathcal{O}(\epsilon^3)$. Thus, using the symmetry
$u^- = 2\pi-u^+,$ it is apparent that in the region $\vert x \vert \leq\Delta$
\begin{equation*}
\vert u_0(x)-u^- \vert \leqslant \vert u^+ - u^- \vert  = 2\vert \pi -
u^- \vert = \mathcal{O}(\epsilon). 
\end{equation*}
Consequently, $u_0(x) = \pi +\mathcal{O}(\epsilon),$ uniform in the
region $\vert x \vert \leq\Delta$. Therefore when $d \gg 1$ stationary
fronts to the system (\ref{v=0sg}) can be approximated by
\begin{equation} \label{frontapproxq>>0}
u_0(x;d,\Delta)= 
\begin{cases}
4\arctan(e^{x+x^*}),&  x \leqslant -\Delta \\
\pi +\mathcal{O}\big(1/\sqrt{d}\big) ,      & \vert x \vert \leqslant \Delta\\
4\arctan(e^{x-x^*}), & x \geqslant \Delta.
\end{cases}
\end{equation}
To determine the translation~$x^*$, we will use the expressions~(\ref{q>1uin}) for the value at the matching point.  Since $d\gg1$ these expressions imply
$$x^* = \Delta-\frac{2\sqrt{2(2-h)}}{\sqrt{d}}+\mathcal{O}\bigg(\frac{1}{d}\bigg).$$
The approximation (\ref{frontapproxq>>0}) of $u_0$ in the linear
operator $\mathcal{L}$ defined in~(\ref{reducedlinear}) gives
the following Lemma about the eigenvalues of the operator $\mathcal{L}$.
\begin{lemma}
  Consider $d \gg 1$. For any $\Lambda \in (-1,0)$, there is a
    $\Delta$ satisfying
 \begin{equation}\label{eq:Delta_d_large}
      \Delta = \frac{-\Lambda}{\sqrt{1+\Lambda}}
      \bigg(\frac{1}{d} +
      \mathcal{O}\bigg(\bigg(\frac{1}{d}\bigg)^{\frac{3}{2}}\bigg)\bigg).  
  \end{equation} 
 such that the linear operator~$\mathcal{L}$ as defined in~(\ref{reducedlinear}) has an
  eigenvalue $\Lambda \in (-1,0)$.
\end{lemma}

\begin{proof}
  We call $\Lambda$ an eigenvalue of $ \mathcal{L}$ if there exists an
  eigenfunction $\Psi \in H^2(\mathbb{R})$ such that
  $\mathcal{L}\Psi = \Lambda\Psi,$ i.e.
  \begin{subequations}
    \begin{align}
      [D_{xx}-\cos(u_{0})]\Psi &
                                 = \Lambda\Psi, \qquad \vert x \vert >\Delta \label{op11}
      \\
      [D_{xx}-(1-d)\cos(u_{0}))]\Psi &
                                       = \Lambda\Psi, \qquad \vert x \vert <\Delta, \label{op22}
    \end{align}
  \end{subequations}
  Since $\mathcal{L}$ is a Sturm-Liouville operator the eigenvalue
  $\Lambda$ has to be real. For any eigenvalue~$\Lambda$, the
  eigenfunction $\Psi \in H^2(\mathbb{R})$, hence $\Psi \to 0$ and
  $\Psi_x \to 0$ for $\vert x \vert \to \infty$. These boundary
  conditions, the fact that $u_0-\pi$ is an odd function and the equations
  (\ref{op11}-\ref{op22}) imply that $\Psi$ is an even function. Using
  the results for the sine-Gordon linearisation in~\cite{Mann1997},
  the solutions for the linear ordinary differential
  equation~(\ref{op11}) for $x<-\Delta$ are spanned by
  \begin{subequations}
	\begin{align}
	\Psi^1(x)&
	=  e^{-\sqrt{1+\Lambda}(x+ x^*)}(\tanh(x+ x^*)+\sqrt{1+\Lambda}) \label{eigen1}
	\\
	\Psi^2(x) &
	= e^{\sqrt{1+\Lambda}(x+ x^*)}(\tanh(x+ x^*)-\sqrt{1+\Lambda}). \label{eigen2}
	\end{align}
\end{subequations}
  
Since we are interested in $\Psi \in H^2(\mathbb{R})$, in the region
$x<-\Delta$ we consider the decaying solution for (\ref{op11}) 
  \begin{equation}\label{Mann}
    \Psi^{-}(x) =
    e^{\sqrt{1+\Lambda}(x+x^*)}(\tanh(x+x^*)-\sqrt{1+\Lambda})
  \end{equation}
  with derivative,
  \begin{equation} \label{MannDiff}
    \frac{d\Psi^-}{dx} =
    e^{\sqrt{1+\Lambda}(x+x^*)}(\sqrt{1+\Lambda}\tanh(x+x^*) -
    \Lambda-\tanh^2(x+x^*)). 
  \end{equation}
  On the other hand, for $d\gg1$ the function
  $u_0(x)=\pi+\mathcal{O}(1/\sqrt d)$, uniform for
  $|x|\leqslant\Delta$, hence the
  ODE~(\ref{op22}) can be written as
  \begin{equation*}
    \frac{1}{d}\frac{d^2\Psi}{dx^2} =
    \bigg(1+\mathcal{O}\bigg(\frac{1}{d}\bigg)\bigg)\Psi, \quad
    \text{uniform for } \vert x \vert \leqslant \Delta. 
  \end{equation*}
  For any fixed $\xi_0$, the even solutions of this linear ordinary
  differential equation are given by
  \begin{equation*}
    \Psi(x) = A\cosh(\sqrt{d}x)+\mathcal{O}\bigg(\frac{1}{d}\bigg)
    \quad \text{and} \quad  \frac{d\Psi}{dx} =
    A\sqrt{d}\sinh(\sqrt{d}x)+\mathcal{O}\bigg(\frac{1}{\sqrt{d}}\bigg),
    \quad |x|\leqslant \frac{\xi_0}{\sqrt d},
  \end{equation*}
  where $A$ is a matching constant. If $\Delta\leqslant
    \xi_0/\sqrt d$, setting
  $\Psi^{-}(-\Delta) =\Psi(-\Delta)$ yields
  \[
 A
 =-\frac{1}{\cosh(\sqrt{d}\Delta)}\bigg(\sqrt{1+\Lambda}-\frac{2\Lambda\sqrt{2(2-h)}}{\sqrt{d}}
 +\mathcal{O}\bigg(\frac{1}{d}\bigg)\bigg).
 \]
   Since we require a
  continuously differentiable solution in $H^2(\mathbb{R})$ we determine the eigenvalue
  $\Lambda$ by matching the derivatives at $x=-\Delta.$ Doing so one
  obtains the equality
  \begin{equation*}  \frac{-\Lambda}{\sqrt{1+\Lambda}}  =
    \sqrt{d}\tanh(\sqrt{d}\Delta)+\mathcal{O}\left(\frac{1}{\sqrt{d}}\right).
  \end{equation*}
  Since $\Delta>0$ the equality above only holds when
  $\Lambda \in (-1,0)$ and $\Delta$ satisfies
  \begin{equation} \label{eigenvalueequation}
\sqrt d\, \Delta =
\frac{-\Lambda}{\sqrt{1+\Lambda}}\bigg(\frac{1}{\sqrt
  d}+\mathcal{O}\bigg( \frac{1}{d}\bigg)\bigg).
  \end{equation} 
  Hence for every $\Lambda\in(-1,0)$, there is a $\Delta$ given by~\eqref{eigenvalueequation}, such that $\Lambda$ is an eigenvalue of the linear operator~\eqref{reducedlinear}.
\end{proof}
For fixed values of $\alpha\in\left(0,1/2\right)$
and $d \gg1$, the above implies that
$\tilde{\mathcal{L}}_{\alpha}$ has an eigenvalue zero for a 
  $\Delta$ satisfying~\eqref{eigenvalueequation} with
  $\Lambda=-2\alpha$.
Substituting this into (\ref{eigenvalueequation}) completes the second
part of the proof of Theorem \ref{bifanalysis2}.

\subsection*{Case 3(a): \boldmath{$\Delta \gg1$, $d>1$}}
Next we approximate the bifurcation locus when $\Delta \gg1$ and
$d>1$. Again, we must first seek approximations to front solutions of
the BVP (\ref{v=0sg}-\ref{BCS3}) when $\Delta \gg1$.
When $d>1$, we can apply the
coordinate transformation $\xi = \sqrt{d-1}\,x+L$ to (\ref{v=0sg}) in
the region $\vert x \vert < \Delta$. The spatial coefficient,
$\sqrt{d-1}$, represents a scaling whilst $L$ is a translation. Under
such transformation the inhomogeneous sine-Gordon
equation~(\ref{v=0sg}) for $|x|<\Delta$ can be written as the
Hamiltonian system
\begin{equation}
  (\tilde{u},\tilde{p})^\top_\xi = J\nabla H(\tilde u,\tilde p),
    \quad J = \left(
\begin{array}{ c c }
0 & 1 \\
-1 & 0
\end{array} \right),  \label{SGxi}
\end{equation}
with Hamiltonian
\begin{equation*}
H(\tilde u,\tilde p) = \frac{1}{2}\tilde{p}^2 - (1+\cos(\tilde{u})), 
\end{equation*}
for $-\Delta+L < \xi/\sqrt{d-1} < \Delta+L$. 
We have
defined the Hamiltonian such that it is zero on the saddle points
$(\tilde{u},\tilde{p})=((1+2k)\pi,0)$ where $k\in \mathbb{Z}.$
Applying the shift transformation
$(\tilde{u}(\xi),\tilde{p}(\xi)) \to (\tilde{u}(\xi)+\pi,
\tilde{p}(\xi))$ the system (\ref{SGxi}) is equivalent to the
stationary sine-Gordon equation. Hence (\ref{SGxi}) has symmetric
heteroclinic connections between saddle points
$(\tilde{u},\tilde{p})=(-\pi,0)$ and $(\tilde{u},\tilde{p})=(\pi,0)$
described by
\begin{equation}
\left(
\begin{array}{c}
\tilde{u}_{\text{het}}^{\pm}(\xi)\\
\tilde{p}_{\text{het}}^{\pm}(\xi)
\end{array}\right) = \left(
\begin{array}{c}
4\arctan(e^{\pm\xi}) -\pi\\
\pm 2\sech(\xi)
\end{array}\right). 
\label{pikxi}
\end{equation}
When $\Delta$ is large, the shape of the front solution $u_0(x;d,\Delta)$ will be
close to the this heteroclinic orbit for $-\Delta<x<0$. Following ideas
from~\cite{Derks2007} we can approximate an orbit of the
system~(\ref{SGxi}) close to the heteroclinic
connections~(\ref{pikxi}). We will focus on solutions close
to~$\tilde{u}_{\text{het}}^+(\xi)$, which pass through
$(\tilde u,\tilde p)=(\pm\pi,\epsilon)$ where $\epsilon$ is a small
parameter, see Figure \ref{ppp}. The Hamiltonian structure implies
that these solutions also pass through
$(\tilde u,\tilde p)=(0,\sqrt{4+\epsilon^2})$. After
obtaining the approximation, we will show how a large length $\Delta$
can be linked to the small parameter~$\epsilon$.
\begin{figure}
	\centering
	\includegraphics[scale=0.5]{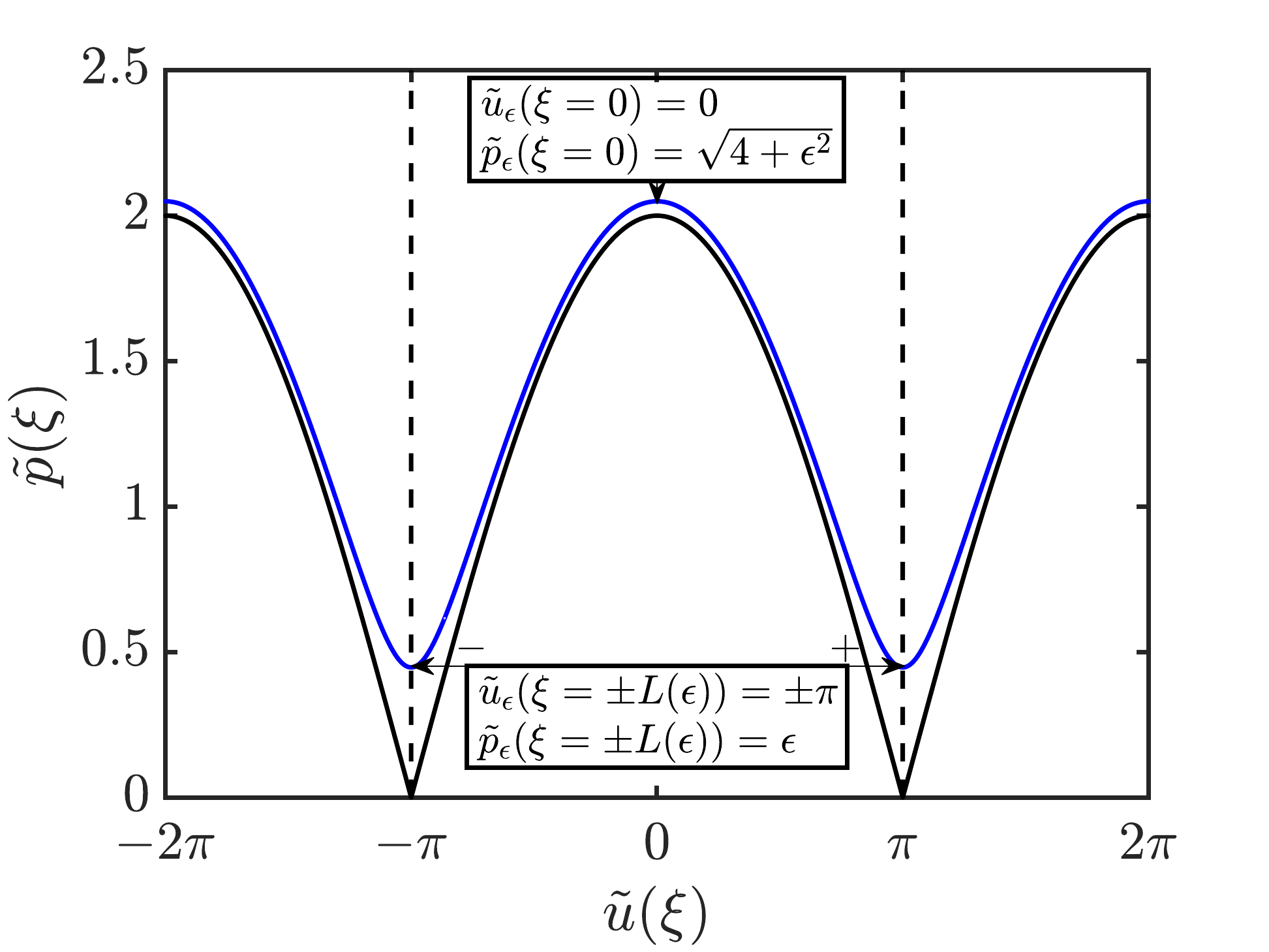}
	\caption{A sketch of the phase plane of system
          (\ref{SGxi}). The black curves are the heteroclinic
          connections. When $\Delta\gg1$, a front solution, for
          $|x|<\Delta$, will lie on the blue curve, hence is close to the
          heteroclinic connection. }
	\label{ppp}
\end{figure}

\begin{lemma} \label{theorempert} For $0<\epsilon\ll 1$, let
  $\tilde{u}_{\epsilon}(\xi)$ denote the orbit of system (\ref{SGxi})
  with initial conditions $\tilde u(\xi=0)=0$ and $\tilde
  p(\xi=0)=\sqrt{4+\epsilon^2}$.  
  This orbit can be approximated by
  \begin{equation}
    \tilde{u}_{\epsilon}(\xi) =  4\arctan(e^{\xi})-\pi +\frac{\epsilon^2}{8}\bigg(\frac{\xi}{\cosh(\xi)}+\sinh(\xi)\bigg) + \epsilon^4 R(\xi; \epsilon), \quad \vert \xi \vert \leqslant L(\epsilon),\label{uouter}
  \end{equation}
  where $L(\epsilon)$ is such that
  $\tilde{u}_{\epsilon}(\xi=L(\epsilon)) = \pi$,
  $\tilde{p}_{\epsilon}(\xi =L(\epsilon)) =\epsilon $. This implies
  \begin{equation*}
    L(\epsilon) = \ln \bigg(\frac{8}{\epsilon}\bigg) +  \mathcal{O}(\epsilon).
  \end{equation*}
  Finally the approximation is $\mathcal{O}(\epsilon)$ whilst the
  remainder term
  $\epsilon^4R(\xi;\epsilon) = \mathcal{O}(\epsilon^2),$ uniform in
  $\vert \xi \vert \leqslant L(\epsilon)$.
\end{lemma}

\begin{proof}
  Since the $(\tilde{u}_{\epsilon}, \tilde{p}_{\epsilon})$ orbit is
  unbounded in $\tilde u$, while the heteroclinic is bounded (see
  Figure \ref{ppp}), any approximation is only going to be valid for
  $\vert \xi \vert \leqslant L(\epsilon),$ where $L(\epsilon)$ is such
  that $\tilde{u}_{\epsilon}(\xi = L(\epsilon)) = \pi$ (and hence
  $\tilde{u}_{\epsilon}(\xi = -L(\epsilon)) = -\pi$). Note that the
  initial condition implies that
  $\tilde{p}_{\epsilon}(L(\epsilon)) = \epsilon$.

  We consider the perturbation series
  \begin{equation}
    \tilde{u}_{\epsilon}(\xi;\epsilon) =
    \tilde{u}_{\text{het}}^{+}(\xi) + \epsilon^2 \tilde{u}_{1}(\xi)  +
    \epsilon^4 R(\xi; \epsilon), \quad \vert \xi \vert \leqslant
    L(\epsilon),
    \label{perseries}
\end{equation}
where $R(\xi; \epsilon)$ is the remainder
term. Substituting~(\ref{perseries}) into~(\ref{SGxi}) yields at first
order
\begin{equation*}
  \frac{d^2\tilde{u}_1}{d\xi^2} + \cos(\tilde{u}_{\text{het}}^{+}(\xi))\tilde{u}_{1} = 0. 
\end{equation*}
The general solution of this second order ODE is (see e.g.~\cite{Derks2007})
\begin{equation*}
  \tilde{u}_{1}(\xi) = \frac{A}{\cosh(\xi)}+B\bigg(\sinh(\xi)+\frac{\xi}{\cosh(\xi)}\bigg),
\end{equation*} 
where $A$ and $B$ are constants, which can be found with the two
initial conditions $\tilde{u}(\xi=0) = 0$ and
$\tilde{p}(\xi=0)=\sqrt{4+\epsilon^2} =
2\bigg(1+\frac{\epsilon^2}{8}+\mathcal{O}(\epsilon^4)\bigg)$, implying
$A=0$ and $B=1/8$.
	
Next we determine the translation constant $L(\epsilon)$. Since
$L(\epsilon) \to \infty$ for $\epsilon \to 0$, we consider
(\ref{perseries}) when $\xi$ is large. As both $\tilde{u}_1(\xi)$ and
$\tilde{u}_{\text{het}}(\xi)$ have $e^{\xi}$ and $e^{-\xi}$ as
fundamental building blocks, we define
$Y(\epsilon) = e^{-L(\epsilon)},$ i.e. $Y(\epsilon) \to 0$ if
$\epsilon \to 0.$ Now we can write
$\tilde{u}_{\epsilon}(\xi = -L(\epsilon)) = -\pi$ as
\begin{equation}
  4\arctan(Y) +\frac{\epsilon^2}{8}\bigg(\frac{1}{2}\bigg(Y-\frac{1}{Y}\bigg)+\frac{2\ln(Y)}{Y^{-1}+Y}\bigg) + \epsilon^4 R(-L(\epsilon);\epsilon) =0. \label{arctanY}
\end{equation}
Thus,
\begin{equation*}
  Y = \tan \bigg(\frac{\epsilon^2}{64Y}\bigg(1-Y^2-\frac{4Y^2}{1+Y^2}\ln Y \bigg)-\epsilon^4R(-L(\epsilon), \epsilon)\bigg).
\end{equation*}
Making the assumption that $\epsilon^2/Y$ and
$\epsilon^4R(-L(\epsilon), \epsilon)$ are small then we can write the
above as
\begin{equation*}
  Y^2 = \frac{\epsilon^2}{64}\bigg(1+\mathcal{O}(Y^2\ln Y)\bigg)-\epsilon^4R(-L(\epsilon), \epsilon) +\mathcal{O}\bigg(\frac{\epsilon^6}{Y^2}\bigg).
\end{equation*}
Hence, $Y = \epsilon/8 +\mathcal{O}(\epsilon^2).$ From this it is
apparent that $\epsilon^2/Y = \epsilon + \mathcal{O}(\epsilon^2)$
which validates the assumption $\epsilon^2/Y$ is small. Hence we find
\begin{equation*}
  L(\epsilon) = -\ln Y = -\ln\bigg(\frac{\epsilon}{8}\bigg)+\ln(1+\mathcal{O}(\epsilon)) = \ln\bigg(\frac{8}{\epsilon}\bigg)+\mathcal{O}(\epsilon).
\end{equation*}
	
Finally we estimate the error term $\epsilon^4R(\xi,\epsilon)$ for
$\vert \xi \vert \leqslant L(\epsilon)$ and verify the second
assumption that $\epsilon^4R(L(\epsilon),\epsilon)$ is small. To do
this we substitute $Y(\epsilon) = \epsilon/8 +\mathcal{O}(\epsilon^2)$
into~(\ref{arctanY}), which gives
\begin{equation*}
  \frac{\epsilon}{2} + \mathcal{O}(\epsilon^2) - \frac{\epsilon}{2}+\mathcal{O}(\epsilon^3\ln \epsilon) + \epsilon^4R(L(\epsilon), \epsilon) = 0,
\end{equation*}
which implies
$\epsilon^4R(L(\epsilon); \epsilon) = \mathcal{O}(\epsilon^2)$ and
hence $\epsilon^4R(\xi; \epsilon) = \mathcal{O}(\epsilon^2)$ for
$\vert \xi \vert \leqslant L(\epsilon)$. This also verifies the
assumption $\epsilon^4R(L(\epsilon),\epsilon)$ is small. Lastly, the
perturbation term
$\epsilon^2\tilde{u}_1(L(\epsilon)) = \mathcal{O}(\epsilon)$, thus
$\epsilon^2\tilde{u}_1(\xi) = \mathcal{O}(\epsilon)$ for all
$\vert \xi \vert \leqslant L(\epsilon)$.
\end{proof}

Returning to the original spatial variable, we can now use
Lemma~\ref{theorempert} to approximate front solutions to the
inhomogeneous sine-Gordon equation~(\ref{v=0sg}) when
$\Delta \gg 1$.
\begin{lemma}\label{lem.front_d>1}
 Consider the inhomogeneous sine-Gordon BVP~(\ref{v=0sg}-\ref{BCS3})
  with fixed $d>1$. For $\Delta\gg1$, the monotonic increasing stationary
  front~$u_0$ can be approximated by
\begin{equation} \label{frontapprox}
u_0(x;d,\Delta)= 
\begin{cases}
4\arctan(e^{x+x^*}),&  x \leqslant -\Delta \\
u_\Delta(x;d) ,      & -\Delta \leqslant x \leqslant 0 \\
2\pi -u_{\Delta}(-x;d),  & 0\leqslant x \leqslant \Delta \\  
4\arctan(e^{x-x^*}), & x \geqslant \Delta.
\end{cases}
\end{equation}
Here, for $-\Delta<x<0$, 
\begin{equation*}
  u_\Delta(x;d) = 4\arctan(\exp(\tilde{x})) -\pi
                + \mathcal{O}\left(\exp\left(-\Delta\sqrt{d-1}\right)\right),
                \quad\mbox{uniform for} \quad -\Delta\leqslant x\leqslant 0.
\end{equation*}
Here $\tilde{x} = \sqrt{d-1} \,(x+\Delta)+L(\Delta,d) $ where
$L(\Delta,d) = \ln\left(\tan((u_0^-+\pi)/4)\right) +
\mathcal{O}\left(\exp(-\Delta\sqrt{d-1})\right)$ and
$u_0^- = \arccos((2-d)/d)$.
Finally, $x^{*}$ is the matching constant given by
\begin{equation} \label{xstar11}
  x^{*}(\Delta,d) = \Delta+\ln \bigg(\tan\bigg(\frac{u_0^-}{4}\bigg)\bigg) +
  \mathcal{O}(\exp(-2\Delta\sqrt{d-1})).
\end{equation}
\end{lemma}

\begin{proof}
 We use Lemma~\ref{theorempert} to approximate the front solution to the BVP~(\ref{v=0sg}-\ref{BCS3}) in the region $-\Delta <x < 0$. The solutions near the heteroclinic connection as described in
  Lemma~\ref{theorempert} go through the point $(\tilde u,\tilde
  p)=(\pi,\epsilon)$ and thus satisfy $H(\tilde u,\tilde
  p)=\epsilon^2/2$.  To link this to the Hamiltonian
  $H_1$, we note that $H_1(u,p) =
  (d-1)H\left(u,p/\sqrt{d-1}\right)$, thus in the
  $(u,p)$ coordinates
this orbit satisfies  $H_1=(d-1)\epsilon^2/2$.  Thus if the solution $\tilde
  u_\epsilon(\xi)$ corresponds to the inner part of the symmetric
  front solution $u_0$ with
  $u_0(0)=\pi$, then the matching condition~\eqref{q>1uin} gives
  \begin{equation}\label{eq:u0min}
    u_0(-\Delta) = \arccos\left(\frac{2-d}d - \frac {\epsilon^2(d-1)} {2d}    \right)
      = u_0^- +\mathcal{O}(\epsilon^2),
    \end{equation}
  where $u_0^- = \arccos((2-d)/d)$.
  Using~\eqref{uouter} with $\xi=-\Delta \sqrt{d-1}+L(\epsilon)$, this
  implies
  \[
    u_0^- +\mathcal{O}(\epsilon^2) = 
    4\arctan(\exp(-\Delta \sqrt{d-1}+L(\epsilon))) -\pi
      +\mathcal{O}(\epsilon)
    \]
Recalling that $L(\epsilon) =  \ln(
  8/\epsilon)+\mathcal{O}(\epsilon)$, this gives the
following relation between $\Delta$ and $\epsilon$
\[
\exp\left(-\Delta\sqrt{d-1}\right) = \frac\epsilon 8 \,
\tan\left(\frac{u_0^-+\pi}{4}\right) +\mathcal{O}(\epsilon^2).  
\]
In other words,
$\epsilon = 
\left(8/\tan\left(\frac{u_0^-+\pi}{4}\right)\right)\exp\left(-\Delta\sqrt{d-1}\right)
+\mathcal{O}\left(\exp\left(-2\Delta\sqrt{d-1}\right)\right)$. Substituting this
for~$\epsilon$ gives the relation for $u_\Delta(x,\Delta)$ in the Lemma. 

Since we require a continuous solution we wish to determine the unique translation constant $x^*$. This can be done by setting $x=-\Delta$ in (\ref{frontapprox}) in the region
$x\leqslant -\Delta$ and using~\eqref{eq:u0min}. Doing so one
determines (\ref{xstar11}).
\end{proof}

With the approximation (\ref{frontapprox}) of $u_0$ in the
linear operator (\ref{reducedlinear}) when $\Delta\gg1$, we can give
the following Lemma. 
\begin{lemma} \label{Delta>>0} Consider $\Delta \gg1.$ Then for fixed $d>1$ the
  linear operator (\ref{reducedlinear}) associated to the unique
  stationary front approximated by (\ref{frontapprox}) has an eigenvalue
  $\Lambda \in (-1,0)$ approximated by $\Lambda = \Lambda_0(d) + \mathcal{O}(\exp\left(-\Delta\sqrt{d-1}\right))$ where $\Lambda_0$ is determined implicitly by
  \begin{multline} \label{eigenvaluezero}
\left(\frac{\sqrt{1+\Lambda_0}}{\sqrt{d}}+\Lambda_0+\frac{1}{d}
\right)
\left(\frac{d-1}{\sqrt{d}}+\sqrt{d-1+\Lambda_0}\right)
\\ 
+\left(\frac{1}{\sqrt{d}}+\sqrt{1+\Lambda_0}\right)
\left(\frac{d-1}{\sqrt{d}}\,\sqrt{d-1+\Lambda_0}+\Lambda_0+\frac{(d-1)^2}{d}\right)=0.
 \end{multline}
\end{lemma}
\begin{proof}
  Consider $\Delta\gg 1$ and fix $d>1$. We call $\Lambda$ an eigenvalue of
  $ \mathcal{L}$ if there exist an eigenfunction $\Psi \in H^2(\mathbb{R})$ such
  that $\mathcal{L}\Psi = \Lambda\Psi,$ i.e.
  \begin{subequations}
    \begin{align}
      [D_{xx}-\cos(u_{0})]\Psi &= \Lambda\Psi, \qquad \vert x \vert >\Delta \label{op1}
      \\
      [D_{xx}-(1-d)\cos(u_{0}))]\Psi &= \Lambda\Psi, \qquad \vert x
                                       \vert <\Delta. \label{op2} 
    \end{align}
  \end{subequations}
  Recall that the decaying solution as $x\to -\infty$ and its
  derivative for (\ref{op1}) in the region $x<-\Delta$ are given by (\ref{Mann}) and
  (\ref{MannDiff}) respectively.  Similarly to $\vert x \vert>\Delta$,
  the results in \cite{Mann1997} for the region $-\Delta<x<0$
  which give the two linearly independent solutions to the
  ODE~\eqref{op2} as,
    \begin{equation*}
  	\Psi^1(x)
  	=
        e^{-\xi\sqrt{1+\tilde{\Lambda}}}(\tanh(\xi)+\sqrt{1+\tilde{\Lambda}})
        \quad \text{and} \quad
  	\Psi^2(x) 
  	= e^{\xi\sqrt{1+\tilde{\Lambda}}}(\tanh(\xi)-\sqrt{1+\tilde{\Lambda}})
  \end{equation*}
  where $\xi = \sqrt{d-1}(x+\Delta)+L(\Delta,d)$ and
  $\tilde{\Lambda}=\Lambda/(d-1)$, for the leading order
  problem. Recall $\Psi(x)$ is an even function hence its derivative
  at $x=0$ must be zero. Since $\Delta\gg1$, $\Psi^1(x)$ vanishes as
  $x \to 0$, whilst $\Psi^2(x)$ grows exponentially. Thus we consider
  $\Psi(x) = A(\Psi^1(x)+\mathcal{O}(\exp(-\sqrt{d-1}\Delta))$ in the
  region $-\Delta<x<0.$ To find the matching constant $A,$ we set
  $\Psi^-(x=-\Delta) = \Psi(x=-\Delta)$ which yields
   \begin{equation*}
   A = -\frac{\left(\tan\left(\frac{u_0^-}{4}\right)\right)^{\sqrt{1+\Lambda}}\left(\frac{1}{\sqrt{d}}+\sqrt{1+\Lambda}\right)}{\left(\tan\left(\frac{u_0^-+\pi}{4}\right)\right)^{-\sqrt{1+\tilde{\Lambda}}}\left(\sqrt{\frac{d-1}{d}}+\sqrt{1+\tilde{\Lambda}}\right)} +\mathcal{O}\left(e^{-\sqrt{d-1}\Delta}\right).
   \end{equation*}
   Since we require a continuously differentiable $\Psi \in H^2(\mathbb{R}),$ we determine the eigenvalue $\Lambda$ by matching the derivatives at $x=-\Delta.$ This yields,
   \begin{multline*}
   \left(\frac{\sqrt{1+\Lambda}}{\sqrt{d}}+\Lambda+\frac{1}{d} \right) \left(\frac{1}{\sqrt{d}}+\sqrt{\frac{1}{d-1}\left(1+\frac{\Lambda}{d-1}\right)}\right) \\
   +\left(\frac{1}{\sqrt{d}}+\sqrt{1+\Lambda}\right)\left(\sqrt{1+\frac{\Lambda}{d-1}}\sqrt{\frac{d-1}{d}}+\frac{\Lambda}{d-1}+\frac{d-1}{d}\right)=\mathcal{O}\left(e^{-\sqrt{d-1}\Delta}\right).
   \end{multline*}
Multiplying both sides through by $(d-1)$ gives the expression in the Lemma.
\end{proof}
For fixed values of $d>1$ and $\Delta\gg1$, Lemma \ref{Delta>>0}
implies that $\tilde{\mathcal{L}}_{\alpha}$ has an eigenvalue zero at
$\alpha = -\Lambda/2$. Therefore substituting $\Lambda_0=-2\alpha$ into
(\ref{eigenvaluezero}) yields (\ref{BifDeltaLarge}) which completes
the third part of the proof.

\subsection*{Case 3(b): \boldmath{$\Delta \gg1$, $0<d<1$}}
When $\Delta\gg1$ and $0<d<1$,
front solution will be close to the heteroclinic solution of the stationary
sine-Gordon equation connecting $(0,0)$ with $(2\pi,0)$.
\begin{lemma}\label{lem.front_d<1}
  Consider the inhomogeneous sine-Gordon BVP~(\ref{v=0sg}-\ref{BCS3})
  with fixed $0<d<1$. For $\Delta\gg1$, the monotonic increasing stationary
  front~$u_0$ can be approximated by
\begin{equation} \label{frontapprox2}
u_0(x;d,\Delta)= 
\begin{cases}
4\arctan(e^{x+x^*}),&  x \leqslant -\Delta \\
u_\Delta(x;d) ,      & -\Delta \leqslant x \leqslant \Delta \\ 
4\arctan(e^{x-x^*}), & x \geqslant \Delta.
\end{cases}
\end{equation}
Here, with $\xi = x\,\sqrt{1-d}$, the function
$u_\Delta(x;d)$ satisfies the following estimate, uniform for $|x|\leq\Delta$,
\begin{equation*}
  u_\Delta(x;d) = 4\arctan(e^{\xi})
  + \epsilon ^2
  \left(\frac{\xi}{\cosh \xi}+\sinh \xi\right)
                + \mathcal{O}\left(e^{-2\Delta\sqrt{1-d}}\right),
              \end{equation*}
              where
              $\epsilon = \frac{8\left(1-\sqrt{1-d}\right)}{\sqrt
                d}\,e^{-\Delta\sqrt{1-d}} +
              \mathcal{O}\left(e^{-2\Delta\sqrt{1-d}}\right)$. The
              matching constant $x^{*}$ is given by
\[
x^* =   \Delta\left(1-\sqrt{1-d}\right) +
\ln\left(\frac{2\sqrt{1-d}\,\left(1-\sqrt{1-d}\right)}{d}\right) 
  + \mathcal{O}\left(e^{-\Delta\sqrt{1-d}}\right).
\]
\end{lemma}
\begin{proof}
  The proof uses the similar ideas as in the proof of
  Lemma~\ref{lem.front_d>1}. Similarly we wish to approximate the front solution in the region $\vert x \vert <\Delta$. First we note that the scaling
  $\xi=x\,\sqrt{1-d}$, $\tilde u(\xi) = u(\xi/\sqrt{1-d})$
  in~\eqref{v=0sg} in the region $\vert x \vert < \Delta$, then applying the shift transformation $\tilde u(\xi) \to \tilde{u}(\xi)+\pi$
  leads to the wave equation considered in
  Lemma~\ref{theorempert}. Hence this Lemma gives an estimate for the
  solution $\tilde u_\epsilon$ which satisfies the initial condition
  $(\tilde{u}_{\epsilon}, \tilde{p}_{\epsilon})=(\pi,\sqrt{4+\epsilon^2})$.  Therefore in the original coordinates, this Lemma gives an estimate for a
  solution $(u_\epsilon(x), p_{\epsilon}(x))$, going through the
  point~$(u_\epsilon, p_{\epsilon}) = (\pi,\sqrt{4+\epsilon^2}\sqrt{1-d})$. This implies that
  the value of $H_1$ is $h=(1-d)(4+\epsilon^2)/2$, thus at the
  matching point $x=-\Delta$, we have
  \begin{equation}\label{eq:boundary_less1}
    u_0(-\Delta;d,\Delta) = \arccos\left(1-\frac{\epsilon^2(1-d)}{2d}\right)
    \quad \text{and} \quad
    p_0(-\Delta;d, \Delta) = \epsilon \,\sqrt{\frac{1-d}{d}}.
  \end{equation}
  First we use these equalities to determine the relation between
  $\epsilon$ and $\Delta$. Approaching the boundary $x=-\Delta$ from
  the right, we get that
  $u_\Delta(-\Delta;d)=\arccos\left(1-\frac{\epsilon^2(1-d)}{2d}\right)$.
  For $\Delta$ large, $\sinh(-\Delta\,\sqrt{1-d})$ is unbounded, hence
  by taking the cosine of both sides, we obtain that
  $\epsilon^2\sinh(-\Delta\,\sqrt{1-d}) = \mathcal{O}(\epsilon)$,
  i.e. $\epsilon =\mathcal{O}(e^{-\Delta\sqrt{1-d}})$. An
  expansion in $\eta = e^{-\Delta\sqrt{1-d}}$ gives the relation for
  $\epsilon$ in the Lemma.

  Next we use the equations~\eqref{eq:boundary_less1} to determine the
  shift $x^*$. Approaching the boundary at $x=-\Delta$ from the left, we get that
  $4\arctan\left(e^{-\Delta+x^*}\right) =
    \arccos\left(1-\frac{\epsilon^2(1-d)}{2d}\right)$. Expanding about
    $\epsilon =0$ gives $e^{-\Delta+x^*} = \frac\epsilon 4
  \,\sqrt{\frac{1-d}{d}} (1+\mathcal{O}(\epsilon^2))$, which together
  with the expression above for $\epsilon$ leads to
  the expression for $x^*$  in the Lemma.  
\end{proof}

With this approximation for $u_0$ in the linear
operator~(\ref{reducedlinear}) when $\Delta\gg1$, we can give the
following Lemma.
\begin{lemma}\label{lem.Lambda_d<1}
  Consider $\Delta \gg1$. Then for fixed $0<d<1$ the linear operator
  $\mathcal{L}$, defined in~(\ref{reducedlinear}) and associated to the
  stationary front $u_0(x;d,\Delta)$ approximated by (\ref{frontapprox2}), has an
  eigenvalue $\Lambda \in (-1,0)$ {which satisfies $\Lambda = o(1)$,
    for $\Delta \to \infty$.}
\end{lemma}

\begin{proof}
  Consider $\Delta\gg 1$ and fix $0<d<1$. For $\Lambda$ to be an eigenvalue of
  $ \mathcal{L}$ there has to exist an eigenfunction $\Psi \in H^2(\mathbb{R})$ such
  that $\mathcal{L}\Psi = \Lambda\Psi,$ i.e.
  \begin{subequations}
    \begin{align}
      [D_{xx}-\cos(u_{0})]\Psi &= \Lambda\Psi, \qquad \vert x \vert >\Delta \label{op1a}
      \\
      [D_{xx}-(1-d)\cos(u_{0}))]\Psi &= \Lambda\Psi, \qquad \vert x
                                       \vert <\Delta. \label{op2a} 
    \end{align}
  \end{subequations}
  The estimate for $x^*$ in Lemma~\ref{lem.front_d<1} gives that for $|x|>\Delta$
  \[
    u_0(x;d,\Delta) \leq
    4\arctan\left(e^{-\Delta\sqrt{1-d}}\,\frac{2\sqrt{1-d}(1-\sqrt{1-d})}{d}+
       \mathcal{O}\left(e^{-2\Delta\sqrt{1-d}}\right)    \right) =
     \mathcal{O}\left( e^{- \Delta\sqrt{1-d}} \right) .
  \]
  Thus for all $x\in\mathbb{R}$, the front is approximated by 
  $u_0(x;\Delta,d) = 4\arctan\left(e^{x\sqrt{1-d}}\right) +
  \mathcal{O}\left(e^{- \Delta\sqrt{1-d}}\right)$ and we can conclude
  that the leading order eigenvalue problem
  is~\eqref{op1a}--\eqref{op2a} with $u_0$ replaced by
  $4\arctan\left(e^{x\sqrt{1-d}}\right)$. The solutions of the inner
  second order ODE~\eqref{op2a} with the $u_0$ replacement are spanned
  by (see~\cite{Mann1997})
  \begin{equation*}
    \Psi^1(x)
    =
    e^{-\xi\sqrt{1+\tilde{\Lambda}}}(\tanh(\xi)+\sqrt{1+\tilde{\Lambda}})
    \quad \text{and} \quad 
    \Psi^2(x) 
    = e^{\xi\sqrt{1+\tilde{\Lambda}}}(\tanh(\xi)-\sqrt{1+\tilde{\Lambda}})
  \end{equation*}
  where $\xi = x\,\sqrt{1-d}$ and
  $\tilde\Lambda = \Lambda/(1-d)$. 
As before, the eigenfunction $\Psi(x)$ is an even function , hence
$\Psi(x) = A(\Psi^1(x)-\Psi^2(x))$ for $|x|<\Delta$.

The solutions of the outer  second order ODE~\eqref{op1a} with the
$u_0$ replacement  can be expressed in generalised hypergeometric
functions. For this proof, it is sufficient to note that the
exponentially decaying solution can be written as
\[
  \Psi^-(x) = e^{\sqrt{\Lambda+1}x}\, \phi(x),
\]
where $\phi(x)$ and its derivative $\phi'(x)$ are uniformly bounded
functions for $x\in(-\infty,0]$.

Matching the values of $\Psi(x)$ and $\Psi^-(x)$ at $x=-\Delta$ gives
$A=\frac{\Psi^-(-\Delta)}{\Psi^1(-\Delta)-\Psi^2(-\Delta)}$.  Matching
the values of the derivatives of $\Psi(x)$ and $\Psi^-(x)$ at $x=-\Delta$
gives 
 \[
   (\Psi^1(-\Delta)-\Psi^2(-\Delta)) \, \left.\frac{d}{dx}\Psi^-(x)\right|_{x=-\Delta}=
\Psi^-(-\Delta) \,\left.\frac{d}{dx}\left(\Psi^1(x)-\Psi^2(x)\right)\right|_{x=-\Delta}.
\]
Using the explicit expressions, this gives
\[
  -e^{ \Delta\sqrt{1-d+\Lambda}} \, \Lambda\,
  \frac{\left(\left(1+\sqrt{1-d}\right)\phi(x)+\phi'(x)\right)}{2 (1-d)}
  +\mathcal{O}\left(e^{- \Delta(2\sqrt{1-d}-\sqrt{1-d+\Lambda})}+e^{-\Delta\sqrt{1-d+\Lambda}}\right) = 0.
\]
Thus for the eigenvalue problem~\eqref{op1a}--\eqref{op2a} with $u_0$
replaced by $4\arctan\left(e^{x\sqrt{1-d}}\right)$, we obtain an
eigenvalue
$\Lambda = \mathcal{O}\left(\Delta\,e^{-2\Delta\sqrt{1-d}}\right)$.
As
$u_0=4\arctan\left(e^{x\sqrt{1-d}}\right) +
\mathcal{O}\left(e^{-\Delta\sqrt{1-d}}\right)$, this implies that in
the original eigenvalue problem the eigenvalue $\Lambda$ is small for
$\Delta$ large. To get the approximation for $\Lambda$ in the original
problem, the correction terms to the approximation
$u_0(x;\Delta,d) = 4\arctan\left(e^{x\sqrt{1-d}}\right)$ 
have to be included. Details of this go beyond what is needed for
the proof of the Lemma.
%
%
\end{proof}

For fixed values of $0<d<1$ and $\Delta\gg1$, Lemma \ref{lem.Lambda_d<1}
implies that $\tilde{\mathcal{L}}_{\alpha}$ has an eigenvalue zero at
$\alpha = -\Lambda/2\gd{=o(1)}$, which completes the final part of the proof of Theorem \ref{bifanalysis2}.
\end{section}


\begin{section}{Bifurcation curve analysis when \boldmath{$d=1$}}
  \label{coupledsys}
The existence of an eigenvalue zero in the linearisation about a
solution is a necessary, but not sufficient, condition for the existence of a bifurcation. In this section, we will take $d=1$ and prove analytically the existence of a
pitchfork bifurcation in the system
(\ref{alphasys2}-\ref{BCS2}). Further we derive
approximations for the bifurcation branches and emerging solutions. 
\begin{theorem} \label{existence} Fix $d=1$ and $\Delta>0$. Then at
  $\alpha=\alpha^*$, as given in~(\ref{alphabif}), the
  system~(\ref{alphasys2}-\ref{BCS2}) undergoes a pitchfork
  bifurcation from the solution $(u,v)=(u_0,0)$ where $u_0$ is given
  by (\ref{q1front}). To be explicit, writing
  $\alpha = \alpha^*+\epsilon^2,$ there is an $\epsilon_0 > 0$ such
  that for all $\vert \epsilon \vert < \epsilon_0,$ there exists a
  unique branch $(u(\epsilon),v(\epsilon))$ such that
  $(u(\epsilon),v(\epsilon))$ are stationary solutions
  of~(\ref{alphasys2}-\ref{BCS2}) and
  \begin{align*}
    \begin{split}
      u(\epsilon) &= u_0 +\mathcal{O}(\epsilon^2),
      \\
      v(\epsilon) &= \epsilon c \Psi+\mathcal{O}(\epsilon^2) .
    \end{split}
  \end{align*}
Here the constant
\begin{equation} \label{c}
		c = \pm\bigg(\frac{4}{3}\alpha^*\int_{0}^{\infty} \Psi^4 dx - \int_{\Delta}^{\infty}V_{21}\Psi^2\sin(u_0)+\frac{\Psi^4}{6}\cos(u_0)dx\bigg)^{-\frac{1}{2}}
\end{equation}	
with $V_{21}$ given in the appendix~\ref{appA}. Finally, $\Psi$ is given by
\begin{equation} \ \frac{\Psi }{R}=
  \begin{cases}
    \exp({\sqrt{1-2\alpha^*}(x+x^{*})})(\tanh(x+x^{*})-\sqrt{1-2\alpha^*}),&  x < -\Delta \\
    A\cos(\sqrt{2\alpha^*}x) ,     &  |x| < \Delta \\
    -\exp({-\sqrt{1-2\alpha^*}(x-x^{*})})(\tanh(x-x^{*})+\sqrt{1-2\alpha^*}),
    & x> \Delta
  \end{cases}
\end{equation}
where $A$ is a matching constant and $R$ is a rescaling constant such
that $||\Psi||_{L^2}^{2} = 1$. Both $A$ and $R$ are given in
Appendix~\ref{appB}.
\end{theorem}
The remainder of this section is dedicated to proving this theorem. We
employ Lyapunov-Schmidt reduction to show existence of a pitchfork
bifurcation whereby $v(x)$ becomes non-zero at the bifurcation point
determined in the previous section.

First, 
we set $\tilde{\alpha} = \alpha- \alpha^*$ with $\alpha^*$ given
by~(\ref{alphabif}) . Now define
$\mathbf{\tilde{F}}: H^2 \times H^2 \times \mathbb{R} \to L^2 \times
L^2$ as
$\mathbf{\tilde{F}}(\mathbf{\tilde{w}};\tilde{\alpha}) =
\mathbf{F}(\mathbf{\tilde{w}},\alpha^* +\tilde{\alpha})$, i.e.
\begin{equation*}
\mathbf{\tilde{F}}(\bm{\tilde{w}};\tilde{\alpha}) =\left(
\begin{array}{ c }
\tilde{u}_{xx} \\
\tilde{v}_{xx}
\end{array} \right)
-
\left(
\begin{array}{ c }
(1-\rho_{0})(\sin(\tilde{u}+u_0)\cos(\tilde{v})-\sin(u_0))  \\
(1-\rho_{0})\sin(\tilde{v})\cos(\tilde{u}+u_0) -(\alpha^*+\tilde{\alpha}) \sin(2\tilde{v})
\end{array} \right).
\end{equation*}
Note the nonlinear operator $\mathbf{\tilde{F}}(\bm{\tilde{w}};\tilde{\alpha})$ is smooth in both $\bm{\tilde{w}}$ and $\tilde{\alpha}$. Furthermore,  $\tilde{\mathbf{F}}(\bm{\tilde{w}} = \mathbf{0};\tilde{\alpha})=\bm{0}$ for
	all $\tilde{\alpha} \in \mathbb{R}.$ For $\tilde{\alpha}$ small, we
	wish to show the existence of a non-trivial solution
	$\bm{\tilde{w}}(\tilde{\alpha})$ in $H^2 \times H^2 \times \mathbb{R}$
	with $\bm{\tilde{w}}(0) = \mathbf{0}$.

 Linearising
$\mathbf{\tilde{F}}(\bm{\tilde{w}};\tilde{\alpha})$ about
$\bm{\tilde{w}}=\mathbf{0}$ yields  (see~\eqref{fulllinear})
\begin{equation}\label{L0L1}
  D\tilde{\mathbf{F}}(\mathbf{0};\tilde{\alpha})
  = \tilde{\mathcal{L}}_{\tilde{\alpha}}
  = \mathcal{L}_{0}
  +\tilde{\alpha}\mathcal{L}_1,
  \quad\mbox{where}\quad
  \mathcal{L}_{0} = \tilde{\mathcal{L}}_{\alpha^*},
  \quad \mathcal{L}_{1} = \left( \begin{array}{c c}
                                                                                     0 & 0\\
                                                                                     0 & 2  
\end{array}\right). 
\end{equation}
Both operators $\mathcal{L}_{0}$ and $\mathcal{L}_{1}$ map
$H^2\times H^2$ into $L^2\times L^2$. We denote the kernel and range of $\mathcal{L}_{0}$ by
$\text{ker}(\mathcal{L}_0)$ and $\text{ran}(\mathcal{L}_0)$,
respectively. The kernel of $\mathcal{L}_{0}$ is given by
$\text{ker}(\mathcal{L}_0) = \text{span}\{\bm{W}\}$ where
$\bm{W}=(0,\Psi)$ and $\Psi \in H^2$ is given by
(\ref{eigenfunctionq=1}) upon substituting $\Lambda=-2\alpha^*$. 
 Since $\text{dim}(\text{ker}(\mathcal{L}_0)) =1$, the range of $\mathcal{L}_0$ is a closed subspace of $L^2\times L^2$ and the orthogonal complement is $\text{ker}(\mathcal{L}_0)$,  $\mathcal{L}_0$ is a Fredholm operator of index zero; see \cite{Golubitsky1985}. Since
$\text{ker}(\mathcal{L}_{0})\ne \{0\}$, the operator $\mathcal{L}_0$ is
not invertible. Thus we employ Lyapunov-Schmidt reduction.

First we rewrite
\begin{equation}
\mathbf{\tilde{F}}(\bm{\tilde{w}};\tilde{\alpha}) = \mathcal{L}_0{\bm{\tilde{w}}} +\tilde{\alpha}\mathcal{L}_1\bm{\tilde{w}} +(1-\rho_{0})N_0(\bm{\tilde{w}}) + (\alpha^*+\tilde{\alpha})N_{1}(\bm{\tilde{w}})
\label{Fnew}
\end{equation}
where $\mathcal{L}_0$ and $\mathcal{L}_1$ are given by (\ref{L0L1}) and
\begin{equation*}
N_{0}(\bm{\tilde{w}}) = \left(
\begin{array}{ c }
\cos(u_0)\tilde{u}+\sin(u_0)-\sin(u_0+\tilde{u})\cos(\tilde{v})\\
\cos(u_0)\tilde{v}-\cos(u_0+\tilde{u})\sin(\tilde{v})
\end{array} \right), \quad N_{1}(\bm{\tilde{w}}) = 
\left(
\begin{array}{ c }
0\\
\sin(2\tilde{v}) - 2\tilde{v}
\end{array} \right).
\end{equation*}
Since $\mathcal{L}_0$ is an elliptic operator we can use the following decompositions
\begin{equation*}
  H^2 \times H^2 = \text{ker}(\mathcal{L}_0) \oplus \text{ker}(\mathcal{L}_0)^{\perp} \quad \text{and} \quad
  L^2 \times L^2 = \text{span}\{\bm{W}\} \oplus \text{span}\{\bm{W}\}^{\perp}, 
\end{equation*}
where $\perp$ denotes the orthogonal complement in $H^2\times H^2$ and
$L^2\times L^2$ respectively. The related projections are
\begin{equation*}
P:H^2 \times H^2 \to H^2 \times H^2 \quad \text{and} \quad Q: L^2 \times L^2 \to L^2 \times L^2,
\end{equation*}
where
\begin{equation*}
  P\bm{\tilde{w}} = \langle \bm{\tilde{w}}, \bm{W} \rangle_{H^2 \times H^2} \bm{W} \quad \text{and} \quad Q\bm{\tilde{w}} = \langle \bm{\tilde{w}}, \bm{W} \rangle_{L^2 \times L^2}\bm{W},
\end{equation*}
respectively. Hence,
$\text{ran}(P) = \text{ker}(\mathcal{L}_0)$ and
$\text{ran}(Q) = \text{span}\{\bm{W}\}$. Since $\mathcal{L}_0$ is a
Sturm-Liouville operator, its eigenvalue zero is simple and all other
eigenvalues and the continuous spectrum are away from zero. Therefore
$\mathcal{L}_0$ is an invertible operator from
$\text{ran}(I-P) = \text{ker}(\mathcal{L}_0)^{\perp} \to \text{ran}(I-Q) =
\text{span}(\bm{W})^{\perp}.$ We now decompose
$\bm{\tilde{w}} \in H^2 \times H^2$ into
$\bm{\tilde{w}} = B\bm{W}+\bm{V}$ where
\begin{equation*}
B\bm{W}=P(\bm{\tilde{w}}) \in \text{ker}(\mathcal{L}_0)
\end{equation*}
\begin{equation*}
\bm{V} = (I-P)(\bm{\tilde{w}}) \in \text{ker}(\mathcal{L}_0)^{\perp}.
\end{equation*}
Furthermore,
$\mathbf{\tilde{F}}(\bm{\tilde{w}};\tilde{\alpha})=\bm{0}$ can be
written as
\begin{subequations}
\begin{align}
  Q\mathbf{\tilde{F}}(B\bm{W}+\mathbf{V};\tilde{\alpha})&= \mathbf{0}, \label{Q} \\
  (I-Q)\mathbf{\tilde{F}}(B\bm{W}+\bm{V};\tilde{\alpha})&= \mathbf{0}. \label{I-Q}
\end{align}
\end{subequations}
Equation (\ref{I-Q}) motivates the definition of
$\mathbf{\tilde{F}}_{2}: \text{ker}(\mathcal{L}_0)^{\perp}\times
\mathbb{R} \times \mathbb{R} \to \text{span}(\bm{W})^{\perp}$ with
\begin{equation*}
  \mathbf{\tilde{F}}_2(\bm{V}; B, \tilde{\alpha}) = (I-Q)\mathbf{\tilde{F}}(B\bm{W}+\bm{V};\tilde{\alpha}).
\end{equation*}
Note that $\mathbf{\tilde{F}}_2(\mathbf{0}; 0, 0)=0$ and
$D_{\bm{V}}\mathbf{\tilde{F}}_2(\mathbf{0};0,0) =
(I-Q)D_{\bm{\tilde{w}}}\mathbf{\tilde{F}}(\mathbf{0};0) = \hat{\mathcal{L}}_0$
where
\begin{equation*}
\hat{\mathcal{L}}_0 : \text{ker}(\mathcal{L}_0)^{\perp} \to
\text{span}(\bm{W})^{\perp} \quad \text{with} \quad
\hat{\mathcal{L}}_0\bm{V} = \mathcal{L}_0\bm{V} \quad \mbox{for}\quad
\bm{V} \in \text{ker}(\mathcal{L})^\perp.
\end{equation*}
The eigenvalues of $\hat{\mathcal{L}}_0$ are bounded away from zero
and
$\text{ran}(\hat{\mathcal{L}}_0) = \text{ran}(\mathcal{L}_0) =
\text{span}(\bm{W})^{\perp},$ hence $\hat{\mathcal{L}}_0$ is
invertible. Thus, the Implicit Function Theorem gives that there
exists an $a_0>0$ such that for $\vert (B,\tilde{\alpha}) \vert < a_0$
there exists a unique $C^\infty$-smooth function
$\bm{V}(B,\tilde{\alpha}) \in \text{ker}(\mathcal{L}_0)^{\perp}$ with
$\mathbf{\tilde{F}}_2(\bm{V}(B,\tilde{\alpha});B,\tilde\alpha)=0$ and
$\bm{V}(0,0) = \mathbf{0}$. Since $\mathbf{\tilde{F}}_2$ is smooth in $B$
and $\tilde\alpha$, the function $\bm{V}(B,\tilde{\alpha})$ will also
depend smoothly on $(B,\tilde{\alpha})$.

So we can expand $\bm{V}(B,\tilde{\alpha})$ in a
Taylor series with respect to $B.$ Note that
$\bm{V}(0,\tilde{\alpha}) = \mathbf{0}$ as
$\mathbf{\tilde{F}}_2(\mathbf{0};0,\tilde{\alpha}) =
(I-Q)\mathbf{\tilde{F}}(\mathbf{0};\tilde{\alpha}) = \mathbf{0}$,
hence
\begin{equation*}
\bm{V}(B,\tilde{\alpha}) = B\bm{V}_1(\tilde{\alpha}) + B^2\bm{V}_2(\tilde{\alpha})  + B^3\bm{V}_3(\tilde{\alpha})+ \mathcal{O}(B^4), 
\end{equation*}
where each component
$\bm{V}_{i} = (V_{i1},V_{i2}) \in H^2 \times H^2.$ Substituting the
above expansion into (\ref{I-Q}) and equating the coefficients yields
the following equations,
\begin{subequations}
  \begin{equation} \label{linear1} (I-Q)(\mathcal{L}_0\bm{V_1}
    +\tilde{\alpha}\mathcal{L}_1(\bm{W}+\bm{V_1}))= \mathbf{0},
\end{equation}
\begin{equation} \label{linear2}
	(I-Q)\bigg(\mathcal{L}_0\bm{V_2} +\tilde{\alpha}\mathcal{L}_1\bm{V_2}
	+(1-\rho_{0})\left(
	\begin{array}{c}
	(V_{11}^2+(\Psi+V_{12})^2)\sin(u_0)/2\\
	V_{11}(\Psi + V_{12})\sin(u_0)
	\end{array} \right)\bigg) = \mathbf{0} ,
	\end{equation}
\end{subequations}
at $\mathcal{O}(B)$ and $\mathcal{O}(B^2)$ respectively. We wish to
solve (\ref{linear1}) for $\bm{V}_1(\tilde{\alpha}).$ Since
$\mathcal{L}_1\bm{W}=2\bm{W},$
$(I-Q)\mathcal{L}_1\bm{W} = \mathbf{0},$ hence
$\bm{V_1}(\tilde{\alpha})=\mathbf{0}$ is a solution of
(\ref{linear1}). The Implicit Function Theorem gives uniqueness of
solutions hence $\bm{V_{1}}(\tilde{\alpha})=\mathbf{0}$ is the only
solution. Consequently,
\begin{equation*}
  \bm{V}(B,\tilde{\alpha}) =  B^2\bm{V}_2(\tilde{\alpha})  + B^3\bm{V}_3(\tilde{\alpha})+ \mathcal{O}(B^4).
\end{equation*}
Now substituting the above into (\ref{Q}) yields
\begin{multline*}
  \langle \Psi, B^2(\mathcal{L}+2\alpha^*)V_{22}+B^3(\mathcal{L}+2\alpha^*)V_{32}+2B\tilde{\alpha}\Psi +2B^2\tilde{\alpha}V_{22}+2B^3\tilde{\alpha}V_{32} \rangle_{L^2} \\
  +B^3\bigg\langle \Psi, (1-\rho_{0})\bigg(
  \frac{1}{6}\cos(u_0)\Psi^3+\sin(u_0)V_{21}\Psi\bigg)
  \bigg\rangle_{L^2} - (\alpha^*+\tilde{\alpha})B^3 \bigg\langle \Psi,
  \frac{4}{3}\Psi^3 \bigg\rangle_{L^2} + \mathcal{O}(B^4) = 0,
\end{multline*}
Since $\mathcal{L}+2\alpha^*$ is a self-adjoint operator with
$\text{ker}(\mathcal{L}+2\alpha^*) = \text{span}\{\Psi\}$, the above can be written
as
\begin{equation*}
\tilde{\alpha} = \frac{1}{2}\bigg(\frac{4}{3}\alpha^*\langle \Psi, \Psi^3 \rangle_{L^2}- \left\langle \Psi, (1-\rho_{0})\left(\frac{1}{6}\cos(u_0)\Psi^3 + \sin(u_0)V_{21}\Psi\right)\right\rangle_{L^2} \bigg)(B^2+\mathcal{O}(B^3)).
\end{equation*}
In Appendix~\ref{appA}, we determine $V_{21}$ from (\ref{linear2}) and show
that it is an odd function. Therefore, using the properties of even
and odd functions the above relation can be written as
\begin{equation*}
  \tilde{\alpha} = \bigg(\frac{4}{3}\alpha^*\int_{0}^{\infty}\Psi^4dx - \int_{\Delta}^{\infty}V_{21}\Psi^2\sin(u_0)+\frac{\Psi^4}{6}\cos(u_0)dx\bigg)B^2 +\mathcal{O}(B^3).
\end{equation*}
Collecting all the results in this section yields Theorem \ref{existence}.
\end{section}


\begin{section}{Persistence for a smooth steep inhomogeneity} \label{hamslowfast}
The main objective in this paper is to show the existence of solutions in the BVP

\begin{gather}
	\begin{split}
	u_{xx} &= (1-d\rho(x))\sin u \cos v, 
	\\
	v_{xx} &= (1-d\rho(x)) \sin v \cos u -\alpha \sin 2v,	\label{alphasysfinal}
	\end{split}
	\\
	\lim_{x\to-\infty}(u(x),v(x))=(0,0) \quad \text {and} \quad
        \lim_{x\to+\infty}(u(x),v(x))=(2\pi,0), \label{BCSfsthm}
\end{gather}
with smooth spatial inhomogeneity given by 
\begin{equation}\label{tanhinhom}
\rho(x;\Delta,\delta) =
\frac{\tanh((x+\Delta)/\delta)+\tanh((-x+\Delta)/\delta)}{2}. 
\end{equation}
However in the previous two sections we studied the existence of stationary
solutions in the BVP with piecewise constant spatial inhomogeneity
\begin{equation*}
\rho_0(x;\Delta)= 
\begin{cases}
0,&  \vert x \vert > \Delta,  \\
1,     &  \vert x \vert < \Delta. 
\end{cases} 
\end{equation*}

In this section, using ideas from both Goh and Scheel~\cite{Goh2016} and Doelman \textit{et al.}  \cite{Doelman2009}, we show that if coupled non-zero solutions exist in the
BVP (\ref{alphasysfinal}-\ref{BCSfsthm}) with the piecewise constant
inhomogeneity $\rho_0(x;\Delta)$ then they persist for the smooth
steep inhomogeneity $\rho(x; \Delta, \delta)$ when $0<\delta \ll 1$. We
summarise the results in the following persistence theorem.

\begin{theorem} \label{smoothinhom}
  Fix $\alpha, d, \Delta>0.$ Assume there exists a transverse
  stationary solution 
  $$\bm{F_0}(x):=(u_0(x),v_0(x)),$$
to the BVP~(\ref{alphasysfinal}-\ref{BCSfsthm}) with piecewise constant spatial inhomogeneity $ \rho_0(x;\Delta)$, satisfying $\partial_x u_0(\pm \Delta)\ne0$ and $ \partial_x v_0(\pm \Delta) \ne 0 $. Define
  \begin{equation}\rho_{\delta}(x;\Delta)=
    \frac{12(1+\epsilon(\Delta/\delta))}{(12+8\epsilon(\Delta/\delta))\left(1+\sqrt{1-144\frac{(1+\epsilon(\Delta/\delta))}{(12+8
          \epsilon(\Delta/\delta))^2}}\cosh(\sqrt{2(1+\epsilon(\Delta/\delta))}x/\delta)\right)}, \label{homoclinicfinal}
  \end{equation}
  where $\delta>0$ is a small parameter and $\epsilon(\Delta/\delta)$
  is such that $\rho_\delta(\pm\Delta;\Delta)=1/2$. This implies
  \begin{equation*}
 \epsilon(\Delta/\delta) = 12\exp\left(\frac{-2\sqrt{2}\Delta}{\delta}\right)\left(1+\mathcal{O}\left(\frac{\Delta}{\delta}\exp\left(\frac{-2\sqrt{2}\Delta}{\delta}\right)\right)\right).
  \end{equation*}
Then there is a $\delta_0>0$ such that for all $ \delta <\delta_0$
  there exists a locally unique stationary solution
  $\bm{F_\delta}(x)=(u_{\delta}(x),v_{\delta}(x))$ of the BVP
  (\ref{alphasysfinal}-\ref{BCSfsthm}) with smooth spatial
  inhomogeneity $ \rho_{\delta}(x;\Delta).$
\end{theorem}

\begin{remark}
	Upon setting $\delta = \sqrt{2}\tilde{\delta}$ in (\ref{tanhinhom}), $\rho(x;\Delta, \sqrt{2}\tilde{\delta}) = \rho_{\delta}(x;\Delta) +\mathcal{O}(\exp(-\Delta/\delta))$.
\end{remark}

The transversality of the solution $\bm{F_0}$ refers to the fact that
it is assumed that $\bm{F_0}$ is the locally unique solution of the dynamical
system associated to~\eqref{alphasysfinal} with $\rho=\rho_0$, which
lies in the transverse intersection of the unstable manifold of the
fixed point corresponding to $(u,v) =(0,0)$ and the stable manifold of
the fixed point corresponding to $(u,v) =(2\pi,0)$. Furthermore, the assumption $\partial_x u_0(\pm \Delta)\ne0$ and $ \partial_x v_0(\pm \Delta) \ne 0 $ implies both components of $\bm{F_0}$ must be non-zero and not have turning points at $x=\pm \Delta$. This ensures that the stable and unstable manifolds, introduced later, can be parametrised by $u$ and $v$.

In the previous section we showed the existence of such solutions $\bm{F_0}$ when $d=1$ through a pitchfork bifurcation. When $v(x)=0$, away from the bifurcation point, it is possible to use the implicit function theorem to show persistence of solutions for $\rho_\delta$, since the linearisation is invertible. However when $v(x) \ne0$, the linearisation is coupled and it becomes a challenge to study its invertibility. Here we use geometric singular perturbation theory to overcome the problem and provide an intuitive geometric description of the persistence of $\bm{F_0}$ to the BVP~\eqref{alphasysfinal} with $\rho=\rho_\delta$ and any fixed $d, \Delta, \alpha>0$.

The remainder of this section is spent proving this persistence
theorem as follows. Firstly, in Section~\ref{inhomsys} we show that
$\rho_\delta(x;\Delta)$ as given by~(\ref{homoclinicfinal}) satisfies
a second order singular ODE. In Section~\ref{extended} this equation
is coupled into the system~(\ref{alphasysfinal}), thus creating a
slow-fast system. We study the singular behaviour of this system in
Section~\ref{delta=0}. Finally in Section~\ref{deltasmall} we employ
geometric singular perturbation theory to show solutions persist and
thus completing the proof.

\subsection{A dynamical system to describe the spatial
  inhomogeneity} \label{inhomsys}
Many dynamical systems could be used to describe a continuous hat-like
spatial inhomogeneity. Here we use the second order ODE
\begin{equation}
\delta^2\rho_{xx} = 4\rho^3-(6+4\epsilon)\rho^2+2(1+\epsilon)\rho, \label{Gardner}
\end{equation} 
which has explicit table top pulse solutions \cite{Hamdi2011}. We
are interested in the case $0 \leqslant\epsilon \ll 1.$ One can
interpret (\ref{Gardner}) as the following first order system,
\begin{align} \label{inhomequation}
\begin{split}
\tilde{\rho}_{\xi} &= \tilde{s} \\
\tilde{s}_{\xi} &= 4\tilde{\rho}^3-(6+4\epsilon)\tilde{\rho}^2+2(1+\epsilon)\tilde{\rho},
\end{split}
\end{align}
where $x=\delta\xi.$ The system (\ref{inhomequation}) has three fixed
points
$(\tilde{\rho}_1^*,\tilde{s}_1^*) = (0,0),
(\tilde{\rho}_2^*,\tilde{s}_2^*)=(1/2,0)$ and
$(\tilde{\rho}_3^*,\tilde{s}_3^*)=(1+\epsilon,0).$ It can be shown
that $(\tilde{\rho}_1^*,\tilde{s}_1^*)$ and
$(\tilde{\rho}_3^*,\tilde{s}_3^*)$ correspond to saddle points whilst
$(\tilde{\rho}_2^*,\tilde{s}_2^*)=(1/2,0)$ is a centre
point. Furthermore, (\ref{inhomequation}) is a Hamiltonian system with
Hamiltonian given by
\begin{equation*}
  H(\tilde{\rho},\tilde{s}) = \frac{1}{2}\tilde{s}^2-\tilde{\rho}^4
  +(2+\frac{4}{3}\epsilon)\tilde{\rho}^3 - (1+\epsilon)\tilde{\rho}^2 ,
\end{equation*}
where the Hamiltonian  is defined such that it vanishes at the saddle
$(\tilde{\rho}_1^*,\tilde{s}_1^*)=(0,0)$. Since we restrict ourselves to
$0\leqslant \epsilon \ll 1$ the system~(\ref{inhomequation}) has two
special cases. They are $\epsilon = 0$ and $\epsilon>0$; see Figure
\ref{phaseplanesfs}.

When $\epsilon = 0$, the system (\ref{inhomequation}) reduces to
\begin{align}
\begin{split}
\tilde{\rho}_{\xi} &= \tilde{s}, 
\\
\tilde{s}_{\xi} &= 4\tilde{\rho}^3-6\tilde{\rho}^2+2\tilde{\rho}.
\end{split}
\label{reducedfastslowfast}
\end{align}
This system possesses heteroclinic orbits which
connect the saddle points $(\tilde{\rho}_1^*,\tilde{s}_1^*)$ and
$(\tilde{\rho}_3^*,\tilde{s}_3^*)$ and are explicitly described by,
\begin{equation}
 ( \tilde{\rho}_0^{\pm}(\xi),\tilde{s}_0^{\pm}(\xi)) =
  \left(\frac{\tanh\left(\pm \xi/\sqrt{2}\right)+1}{2}, \pm\frac{\sqrt{2}\sech^2(\xi/\sqrt{2})}{4}\right).
  \label{hetroclinic1}
\end{equation}
The fronts $\tilde{\rho}_0^{\pm}(\xi)$ are centred at $\xi=0$ in physical space. Here $\tilde{\rho}_0^{+}(\xi)$
corresponds to the monotonic increasing front, whilst
$\tilde{\rho}_0^{-}(\xi)$ the monotonic decreasing front.

\begin{figure}
	\centering
	\includegraphics[scale=0.37]{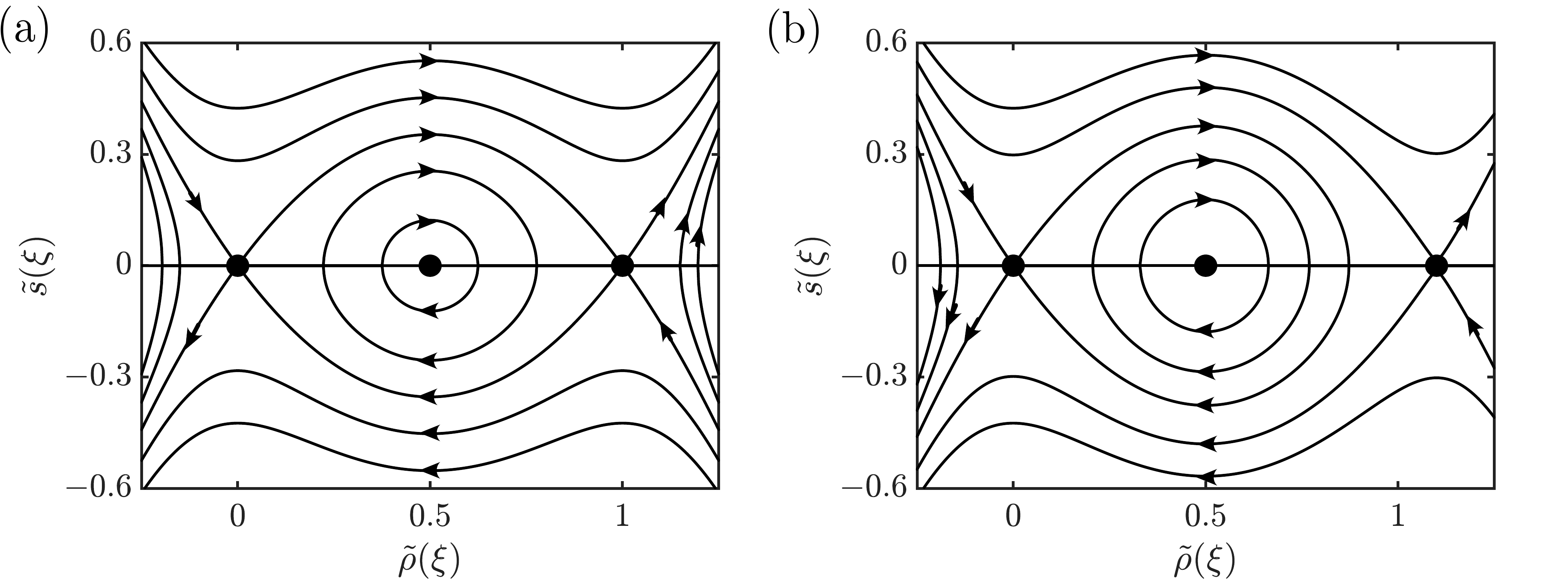}
	\caption{The left panel is a plot of the phase plane of the system~(\ref{inhomequation}) when $\epsilon=0$. The right panel when $\epsilon=0.1$.}
	\label{phaseplanesfs}
\end{figure}

All three of the fixed points persist for $0<\epsilon\ll1.$ The fixed
points $(\tilde{\rho}_1^*,\tilde{s}_1^*)$ and
$(\tilde{\rho}_2^*,\tilde{s}_2^*)$ are invariant, whilst
$(\tilde{\rho}_3^*,\tilde{s}_3^*)$ translates rightwards in phase
space. Furthermore, when $\epsilon >0$ the system
(\ref{inhomequation}) no longer possesses heteroclinic
orbits. Instead, the saddle point $(\tilde{\rho}_1^*,\tilde{s}_1^*)$
has a homoclinic orbit which is order $\sqrt{\epsilon}$ close to the
heteroclinic connections, see Figure \ref{phaseplanesfs}. Using
results in \cite{Hamdi2011}, the homoclinic orbit is explicitly
described by
\begin{equation}
(\tilde{\rho}_\epsilon(\xi),\tilde{s}_\epsilon(\xi)) =
\left(\frac{a}{1+b\cosh(c\xi)},-\frac{abc\sinh(c\xi)}{(1+b\cosh(c\xi))^2}\right)
\label{homoclinic1} 
\end{equation}
where
\begin{equation*}
a=\frac{12(1+\epsilon)}{12+8\epsilon}, \quad b= \sqrt{1-144\frac{1+\epsilon}{(12+8\epsilon)^2}} \quad \text{and} \quad c=\sqrt{2(1+\epsilon)}.
\end{equation*}
The pulse $\tilde{\rho}_\epsilon(\xi)$ is centred at $\xi=0$ in physical space. Setting
$\xi = 0$ in (\ref{homoclinic1}) one obtains the maximum height of the
pulse. As $\epsilon \to 0,$ the maximum height of the pulse tends to
one and the pulse widens. Note that substituting $\epsilon=0$ into
(\ref{homoclinic1}) yields the saddle point
$(\tilde{\rho}_3^*,\tilde{s}_3^*)$. Finally we obtain a
relationship between the perturbation parameter $\epsilon$ and the
width of the pulse. Defining $L>0$ to be such that
$\tilde{\rho}_\epsilon(\pm L)=1/2$, one obtains
\begin{equation*}
  L =
      \frac{1}{c} \arccosh\left(\frac{3+4\epsilon}{b(3+2\epsilon)}\right) =  \frac{1}{\sqrt{2}}\log\bigg(\frac{2\sqrt{3}}
      {\sqrt{\epsilon}}\bigg)(1+\mathcal{O}(\epsilon)). 
\end{equation*}
This  is an invertible relation between $L$ and the small
parameter~$\epsilon$ and leads to a function $\epsilon(L)$ satisfying 
$\epsilon(L) =
12\exp(-2\sqrt{2}L)(1+\mathcal{O}(L\exp(-2\sqrt{2}L)))$.

Returning to the original spatial variable $x$, the above results to give the following Lemma.
\begin{lemma} \label{fastslowlem}
	Consider $0<\epsilon, \delta \ll 1.$ Then the ODE (\ref{Gardner}) with boundary conditions $\rho(\pm\infty)=0$ has solution
	\begin{equation} \label{homocliniclemma}
	\rho_{\delta}(x;\epsilon) = \frac{12(1+\epsilon)}{(12+8\epsilon)\left(1+\sqrt{1-144\frac{(1+\epsilon)}{(12+8
				\epsilon)^2}}\cosh\big(\sqrt{2(1+\epsilon)}x/\delta\big)\right)}.
	\end{equation} 
Define $\Delta>0$ to be such that $\rho_{\delta}(\pm \Delta; \epsilon)=1/2$, then 
\begin{equation*}
\frac\Delta{\delta} = \frac{1}{c}
\arccosh\left(\frac{3+4\epsilon}{b(3+2\epsilon)}\right) =
\frac{1}{\sqrt{2}}\log\bigg(\frac{2\sqrt{3}}
{\sqrt{\epsilon}}\bigg)(1+\mathcal{O}(\epsilon)).
\end{equation*}
This is an invertible relation between $\Delta/\delta$ and $\epsilon$
and leads to a function $\epsilon(\Delta/\delta)$ satisfying
 \begin{equation} \label{epsilon}
\epsilon(\Delta/\delta) = 12\exp\left(\frac{-2\sqrt{2}\Delta}{\delta}\right)\left(1+\mathcal{O}\left(\frac{\Delta}{\delta}\exp\left(\frac{-2\sqrt{2}\Delta}{\delta}\right)\right)\right).
 \end{equation}
Furthermore, the pulse $\rho_\delta$ approximates $\rho_0$. To be
specific, for fixed $\Delta>0$ in the three regions
\begin{equation*}
R^{-} := (-\infty,-\Delta-\delta^{a}],\quad
R^{0} := [-\Delta+\delta^{a}, \Delta - \delta^{a}], 
R^{+} := [\Delta+\delta^{a-1}, \infty),
\end{equation*}
the pulses satisfy $\vert (\rho_{\delta}(x;\epsilon(\Delta/\delta))-\rho_{0}(x;\Delta), s_{\delta}(x;\epsilon(\Delta/\delta)))\vert = \mathcal{O}\left(\exp\left(-\sqrt{2}\delta^{a-1}\right)\right)$ for any $0<a<1.$
\end{lemma}

\subsection{The extended slow-fast system} \label{extended}

Here we extend the inhomogeneous system~(\ref{alphasysfinal}) with the
equation~(\ref{Gardner}) to describe a smooth steep spatial
inhomogeneity~$\rho$. To be specific, we consider the following system
\begin{align}
\begin{split}
u_{xx} &= (1-d\rho)\sin u \cos v, 
\\
v_{xx} &= (1-d\rho)\sin v \cos u -\alpha\sin 2v , 
\\
\delta^2\rho_{xx} &=
4\rho^3-(6+4\epsilon(\Delta/\delta))\rho^2+2(1+\epsilon(\Delta/\delta))\rho,  
\end{split}
\label{fastslow}
\end{align}
where $u$, $v$, and $\rho$ are the dependent variables whilst $d$,
$\Delta$, and $\alpha$ are constants. The perturbation parameter
$0<\delta \ll 1$ corresponds to the steepness of the smooth spatial
inhomogeneity. Furthermore, $0<\epsilon(\Delta/\delta) \ll 1$ is
given by the condition $\rho(\pm \Delta; \epsilon)=1/2$ and satisfies (\ref{epsilon}). Thus $\epsilon(\Delta/\delta)$ is determined by both
the length and steepness of the spatial inhomogeneity. For
fixed $\Delta>0,$ $\epsilon(\Delta/\delta) \to 0$ as $\delta \to 0$ and
$0<\epsilon \ll \delta \ll1$.

Notice that the third equation of
(\ref{fastslow}) is independent of the first two. On the other hand,
the first two equations in (\ref{fastslow}) are coupled to each other
and the last equation in the system. We can rewrite (\ref{fastslow})
as the following six dimensional first order dynamical system,
\begin{align}
\begin{split}
u_{x} &= p, 
\\
p_{x} &= (1-d\rho)\sin u \cos v, 
\\
v_{x} &= q,
\\
q_{x} &= (1-d\rho)\sin v \cos u -\alpha\sin 2v,
\\
\delta\rho_{x} &= s, 
\\
\delta s_{x} &= 4\rho^3-(6+4\epsilon(\Delta/\delta))\rho^2+2(1+\epsilon(\Delta/\delta))\rho. 
\end{split}\label{fastslowslow}
\end{align}
We call this the `slow' system.
This system has saddle points at
\[
  \mathcal{B}_{-\infty}=(0,0,0,0,0,0)\quad \text{and} \quad
  \mathcal{B}_{+\infty}=(2\pi,0,0,0,0,0) .
\]
We are interested in solutions of the
`slow' system (\ref{fastslowslow}) with the boundary conditions
\[
  \lim_{x\to-\infty}(u,p,v,q,\rho,s)(x) =\mathcal{B}_{-\infty} \quad \text{and} \quad
  \lim_{x\to +\infty}(u,p,v,q,\rho,s)(x) = \mathcal{B}_{+\infty}.
\]
Upon making the change of variable $x=\delta \xi$ we obtain the `fast'
system,
\begin{align}
\begin{split}
\tilde{u}_{\xi} &= \delta \tilde{p}, 
\\
\tilde{p}_{\xi} &= \delta (1-d\tilde{\rho})\sin \tilde{u} \cos \tilde{v}, 
\\
\tilde{v}_{\xi} &= \delta \tilde{q}, 
\\
\tilde{q}_{\xi} &= \delta((1-d\tilde{\rho})\sin \tilde{v} \cos \tilde{u} -\alpha\sin 2\tilde{v}), 
\\
\tilde{\rho}_{\xi} &= \tilde{s}, 
\\
\tilde{s}_{\xi} &= 4\tilde{\rho}^3-(6+4\epsilon(\Delta/\delta))\tilde{\rho}^2+2(1+\epsilon(\Delta/\delta))\tilde{\rho}.
\end{split} \label{fastslowfast}
\end{align}
Note that the last two equations in the `fast' system
(\ref{fastslowfast}) are exactly the system (\ref{inhomequation}).

The slow and fast systems~(\ref{fastslowslow})
and~(\ref{fastslowfast}) are equivalent when $\delta \ne0$. However
they are not equivalent in the limit $\delta \to 0$. Furthermore, the
system~(\ref{fastslowslow}) is singularly perturbed in the limit
$\delta \to 0$. We first study both systems when
$\delta \to 0$. Then we use regular and singular perturbation theory
to analyse the systems when $0<\delta\ll1$.

\subsection[A]{Dynamics of the extended system in the limit ${\bm{\delta \to 0}}$} \label{delta=0}
\subsubsection{Fast dynamics}
When $\delta \to 0,$ the `fast' system (\ref{fastslowfast}) reduces to
the planar system (\ref{reducedfastslowfast}) and
$(\tilde{u},\tilde{p},\tilde{v},\tilde{q})$ staying constant. We call
this the `fast' reduced system (FRS). Recall that
(\ref{reducedfastslowfast}) has saddle points
$(\tilde{\rho},\tilde{s})=(0,0)$ and $(\tilde{\rho},\tilde{s})=(1,0).$
Thus, we can define the following 4-dimensional normally hyperbolic
invariant slow manifolds
\begin{equation*}
\begin{split}
\mathcal{M}_{0}^{l}= \{(\tilde{u}, \tilde{p}, \tilde{v}, \tilde{q}, \tilde{\rho}, \tilde{s}) = (\tilde{u}, \tilde{p}, \tilde{v}, \tilde{q}, 0, 0) \hspace{6pt} \vert \hspace{6pt} (\tilde{u}, \tilde{p}, \tilde{v}, \tilde{q}) \in \mathbb{R}^4 \},\\
\mathcal{M}_{0}^{r}= \{(\tilde{u}, \tilde{p}, \tilde{v}, \tilde{q}, \tilde{\rho}, \tilde{s}) = (\tilde{u}, \tilde{p}, \tilde{v}, \tilde{q}, 1, 0) \hspace{6pt} \vert \hspace{6pt} (\tilde{u}, \tilde{p}, \tilde{v}, \tilde{q}) \in \mathbb{R}^4 \}.
\end{split}
\end{equation*}
The manifolds $\mathcal{M}_{0}^{l/r}$ have 5-dimensional stable and
unstable manifolds $\mathcal{W}^{s/u}(\mathcal{M}_{0}^{l/r})$. We are interested in the connections between the manifolds
$\mathcal{W}^{s/u}(\mathcal{M}_{0}^{l})$ and
$\mathcal{W}^{s/u}(\mathcal{M}_{0}^{r})$ which correspond to the
one dimensional heteroclinic connections~(\ref{hetroclinic1}) of the saddle points
$(\tilde{\rho},\tilde{s})=(0,0)$ and $(\tilde{\rho},\tilde{s})=(1,0)$
in the 4-parameter family
$(\tilde{u}, \tilde{p}, \tilde{v}, \tilde{q})\in \mathbb{R}^4$.

\subsubsection{Slow dynamics}
On the other hand taking $\delta \to 0$ in the `slow' system
(\ref{fastslowslow}) yields the following differential-algebraic
system,
\begin{align}
\begin{split}
u_{x} &= p, 
\\
p_{x} &= (1-d\rho)\sin u \cos v, 
\\
v_{x} &= q, 
\\
q_{x} &= (1-d\rho)\sin v \cos u -\alpha\sin 2v , 
\\
0 &= s, 
\\
0 &= 4\rho^3-6\rho^2+2\rho.
\end{split}\label{reducedfastslowslow}
\end{align}
We call this the `slow' reduced system (SRS). Solving the last two
algebraic equations in this system yields the solution set
$(\rho,s) = \{(0,0),(1/2,0),(1,0)\}$. These are exactly the fixed
points of (\ref{reducedfastslowfast}).  The point $(\rho, s) =(0,0)$
corresponds to the slow manifold~$\mathcal{M}_{0}^{l}$
and the slow  dynamics on this manifold is 
\begin{align*}
u_x&= p, \quad p_{x} = \sin u \cos v; \nonumber
\\
v_x&=q,  \quad q_{x} = \sin v \cos u -\alpha\sin 2v. 
\end{align*}
Similarly the point $(\rho,s) =(1,0)$ corresponds
to~$\mathcal{M}_{0}^{r}$ and the slow dynamics on this manifold is 
\begin{align*}
u_x&= p, \quad p_{x} = (1-d)\sin u \cos v; \nonumber
\\
v_x&=q,  \quad q_{x} = (1-d)\sin v \cos u -\alpha\sin 2v.  
\end{align*}

The dynamics of the system (\ref{alphasysfinal}) with the piecewise
constant spatial inhomogeneity $\rho(x;\Delta) = \rho_0(x;\Delta)$ can
be described as the slow dynamics on $\mathcal{M}_{0}^{l}$ for $|x|>\Delta$
and on $\mathcal{M}_{0}^{r}$ for $|x|<\Delta$ with matching conditions
at $|x|=\Delta$. The slow dynamics on the four dimensional manifold
$\mathcal{M}_{0}^{l}$ have a fixed point $B_{-\infty}$ at
$(u,p,v,q)=(0,0,0,0)$ and $B_{+\infty}$ at
$(u,p,v,q)=(2\pi,0,0,0)$. These fixed points are saddles with two
dimensional stable and two dimensional unstable manifolds. Denote the
unstable manifold to $B_{-\infty}$ by
$\mathcal{W}_{0,\text{slow}}^u(B_{-\infty})$ and the stable manifold
to $B_{+\infty}$ by $\mathcal{W}_{0,\text{slow}}^s(B_{+\infty})$. 

 If $\bm{F_0}=(u_{0}(x),v_{0}(x))$ is a stationary solution to
(\ref{alphasysfinal}), with piecewise constant spatial inhomogeneity
given by $\rho_{0}(x;\Delta)$ and satisfies the boundary conditions
(\ref{BCSfsthm}), then
$(u_{0}, p_{0},v_{0}, q_{0}) = (u_{0}, \partial_x u_{0},v_{0}, \partial_x v_{0})  $ lies on the
unstable manifold $\mathcal{W}_{0,\text{slow}}^u(B_{-\infty})$ for
$x<-\Delta$ and on the stable manifold
$\mathcal{W}_{0,\text{slow}}^s(B_{+\infty})$ for $x>\Delta$. Then the
existence of a unique transverse solution $\bm{F_0}$ is equivalent to
continuing the solutions on the two dimensional unstable manifold
$\mathcal{W}_{0,\text{slow}}^u(B_{-\infty})$ at $x=-\Delta$ with the
flow of the system on $\mathcal{M}_0^r$ up to $x=\Delta$ and matching
one of them to a unique solution on the stable manifold
$\mathcal{W}_{0,\text{slow}}^s(B_{+\infty})$. Hence the transversality
of $\bm{F_0}$ refers to a transverse intersection of the continuation
of the two-dimensional unstable manifold  with
the two-dimensional stable manifold in the four dimensional
$(u,p,v,q)$ space. Since the system is Hamiltonian, the transverse intersection is one dimensional, despite the fact that a dimension count suggests that the intersection is zero dimensional.

\subsubsection{Combining the geometry of the slow and fast dynamics}
Both the two dimensional stable and unstable manifolds
$\mathcal{W}_{0,\text{slow}}^s(B_{+\infty})$ and
$\mathcal{W}_{0,\text{slow}}^u(B_{-\infty})$ are subsets of the four dimensional slow
manifold $\mathcal{M}_{0}^{l}$. Recall that $\bm{F_0}(x)$ lies on $\mathcal{W}_{0,\text{slow}}^u(B_{-\infty})$ for $x<-\Delta$ and on $\mathcal{W}_{0,\text{slow}}^s(B_{+\infty})$ for $x>\Delta$. Define $\bm{F_0}(\pm \Delta) = (u_{0}^{\pm}, p_{0}^{\pm},v_{0}^{\pm}, q_{0}^{\pm}) \in \mathbb{R}^4$  where by assumption $p_0^{\pm} = \partial_x u_0(\pm \Delta) \ne0$ and $ q_0^{\pm} = \partial_x v_0(\pm \Delta) \ne 0 $. This ensures that $\mathcal{W}_{0,\text{slow}}^s(B_{+\infty})$ and
$\mathcal{W}_{0,\text{slow}}^u(B_{-\infty})$ can be written as graphs over $u$ and $v$. Hence nearby $\bm{F_0}(\pm\Delta)$, the manifolds
$\mathcal{W}_{0,\text{slow}}^s(B_{+\infty})$ and
$\mathcal{W}_{0,\text{slow}}^u(B_{-\infty})$ can be characterised by
their $(u,v)$ values at $x=\pm\Delta$ and give two dimensional
sub-manifolds in $\mathcal{M}_{0}^{l}:$
\begin{eqnarray*}
	\overline{\mathcal{W}}_{0,\text{slow}}^u&:=&\{(u,p_0^u(u,v),v,q_0^u(u,v))
	\mid (u,v) \mbox{ near } (u_0^-,v_0^-) \},\\
	\overline{\mathcal{W}}_{0,\text{slow}}^s&:=&\{(u,p_0^s(u,v),v,q_0^s(u,v))
	\mid (u,v) \mbox{ near } (u_0^+,v_0^+) \}.
\end{eqnarray*}
Flowing the unstable sub-manifold
$\overline{\mathcal{W}}_{0,\text{slow}}^u$ forward with the flow of
the slow system on $\mathcal{M}_0^r$ for a $\Delta$ length gives a
three dimensional manifold containing
$(u_{0}, p_{0}, v_{0}, q_{0})([-\Delta,0])$. Similarly, flowing the stable sub-manifold
$\overline{\mathcal{W}}_{0,\text{slow}}^s$ backwards with the flow of
the slow system on $\mathcal{M}_0^r$ for length $\Delta$ gives a three
dimensional manifold containing
$(u_{0}, p_{0},v_{0}, q_{0})([0,\Delta])$.  We will denote these manifolds as
\begin{eqnarray*}
	\overline{\mathcal{W}}_{0,\text{slow}}^{u,r} &=&
	\{\Phi_{x,-\Delta}^{0,r,\text slow} (u,p,v,q) \mid (u,p,v,q) \in
	\overline{\mathcal{W}}_{0,\text{slow}}^u ,\, x\in[-\Delta,0]  \},\\
	\overline{\mathcal{W}}_{0,\text{slow}}^{s,r} &=&
	\{\Phi_{x,\Delta}^{0,r,\text slow} (u,p,v,q) \mid (u,p,v,q) \in
	\overline{\mathcal{W}}_{0,\text{slow}}^s,\, x\in[0,\Delta]  \},
\end{eqnarray*}
where $\Phi_{x,x_0}^{0,r,\text slow}$ denotes the flow of the slow system
on the right slow manifold~$\mathcal{M}_0^r$. The transversality
assumption on $\bm{F_0}$ implies that the boundaries of these two
manifolds at $x=0$ intersect transversely.

On the other hand flowing the unstable sub-manifold
$\overline{\mathcal{W}}_{0,\text{slow}}^u$ backwards with the flow of
the slow system on $\mathcal{M}_0^l$ gives a three dimensional
manifold containing
$(u_{0}, p_{0},v_{0}, q_{0})((-\infty,-\Delta])$. Similarly, flowing the stable
sub-manifold $\overline{\mathcal{W}}_{0,\text{slow}}^s$ forward with
the flow of the slow system on $\mathcal{M}_0^l$ gives a three
dimensional manifold containing
$(u_{0}, p_{0},v_{0},
q_{0})([\Delta,+\infty))$.  We will denote these manifolds as
\begin{eqnarray*}
\overline{\mathcal{W}}_{0,\text{slow}}^{u,l}
  &=&\{\Phi^{0,l,slow}_{x,-\Delta}(u,p,v,q)\mid 
(u,p,v,q) \in \overline{\mathcal{W}}_{0,\text{slow}}^u, \,
x\in(-\infty,-\Delta]\}, \\
  \overline{\mathcal{W}}_{0,\text{slow}}^{s,l}
  &=& \{\Phi^{0,l,slow}_{x,\Delta}(u,p,v,q)\mid 
(u,p,v,q) \in \overline{\mathcal{W}}_{0,\text{slow}}^s, \,
x\in[\Delta,\infty)\}
\end{eqnarray*}
where $\Phi_{x,x_0}^{0,l,\text slow}$ denotes the flow of the slow system on the left slow manifold~$\mathcal{M}_0^l$.

Extending these sets with either the unstable respectively stable manifolds
of the fast dynamics or the fixed points gives the following three
dimensional manifolds in $\mathbb{R}^6$
\begin{eqnarray*}
		{\mathcal{S}}_0^{u,l} &=&
		\{(u,p,v,q,0,0) \mid (u,p,v,q) \in
		\overline{\mathcal{W}}_{0,\text{slow}}^{u,l}\};\\
		{\mathcal{S}}_0^{s,l} &=&
		\{(u,p,v,q,0,0) \mid (u,p,v,q) \in
		\overline{\mathcal{W}}_{0,\text{slow}}^{s,l}\};\\
	\mathcal{S}_0^u &=&
	\{(u,p,v,q,\tilde\rho_0^{+}(\xi),\tilde s_0^{+}(\xi))\mid
	(u,p,v,q) \in
	\overline{\mathcal{W}}_{0,\text{slow}}^u, \,
	\xi \in \mathbb{R}\}; \\
	\mathcal{S}_0^s &=&
	\{(u,p,v,q,\tilde\rho_0^{-}(\xi),\tilde s_0^{-}(\xi))\mid
	(u,p,v,q) \in
	\overline{\mathcal{W}}_{0,\text{slow}}^s, \,
	\xi \in \mathbb{R} \};\\
	{\mathcal{S}}_0^{u,r} &=&
	\{(u,p,v,q,1,0) \mid (u,p,v,q) \in
	\overline{\mathcal{W}}_{0,\text{slow}}^{u,r}\};\\
	{\mathcal{S}}_0^{s,r} &=&
	\{(u,p,v,q,1,0) \mid (u,p,v,q) \in
	\overline{\mathcal{W}}_{0,\text{slow}}^{s,r}\}.
\end{eqnarray*}
Here $(\tilde\rho_0^{\pm},\tilde s_0^{\pm})(\xi)$ are the heteroclinic
connections (\ref{hetroclinic1}) between saddle points $(0,0)$ to
$(1,0)$ in the FRS.
Note $\mathcal{S}_0^{u} $ is a three dimensional submanifold of the
five dimensional unstable manifold of $\mathcal{M}_{0}^{l}$ and
$\mathcal{S}_0^{s} $ is a three dimensional submanifold of the five
dimensional stable manifold of $\mathcal{M}_{0}^{l}$. To capture a
full neighbourhood of the solution associated with $\bm{F_0}(x)$, we
combine these manifolds as follows (see Figure~\ref{fig.approx_by_smooth}):
\begin{eqnarray*}
	\mathcal{T}_0^u &=& \mathcal{S}_0^{u,l} \cup \mathcal{S}_0^u \cup \mathcal{S}_0^{u,r};\\
	\mathcal{T}_0^s &=& \mathcal{S}_0^{s,l} \cup \mathcal{S}_0^s \cup \mathcal{S}_0^{s,r}.
\end{eqnarray*}

\begin{figure}[h]
	\captionsetup[subfigure]{labelformat=empty}
	\centering
\includegraphics[scale=0.58]{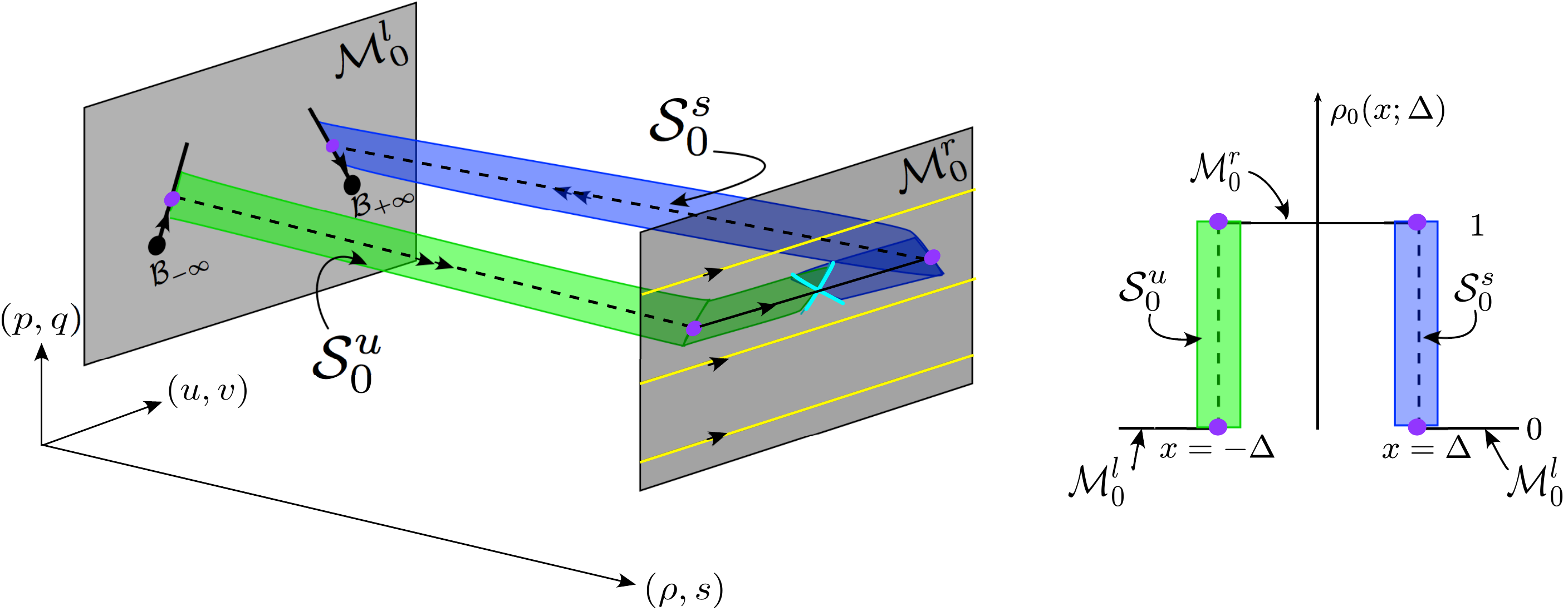}
	\caption{\label{fig.approx_by_smooth}%
          The figure depicts a sketch of the manifolds of the system (\ref{fastslowslow}) in the limit $\delta \to 0.$ The left
          grey manifold represents~$\mathcal{M}_0^l$ and the black curves on it
          represent the unstable and stable manifolds of $\mathcal{B}_{\pm \infty}$ on which $\bm{F_0}$ lies for $\vert x \vert > \Delta$. The
          right grey manifold represents~$\mathcal{M}_0^r$ and the
          yellow curves depict the flow on this manifold with the
          black curve corresponding to $\bm{F_0}$ for $|x|<\Delta$. The purple dots correspond to $(u_0^{\pm}, p_0^{\pm}, v_0^{\pm}, q_0^{\pm})$. The
          green manifold in~$\mathcal{M}_0^r$
          represents~$\mathcal{S}_0^{u,r}$ and the blue manifold
          in~$\mathcal{M}_0^r$ represents~$\mathcal{S}_0^{s,r}$. The
          light blue curves are the edges of~$\mathcal{S}_0^{u,r}$
          and~$\mathcal{S}_0^{s,r}$ when $x=0$. This illustrates the
          transverse intersection of these edges.  }
\end{figure}

\begin{figure}[h]
	\captionsetup[subfigure]{labelformat=empty}
	\centering
	\includegraphics[scale=0.56]{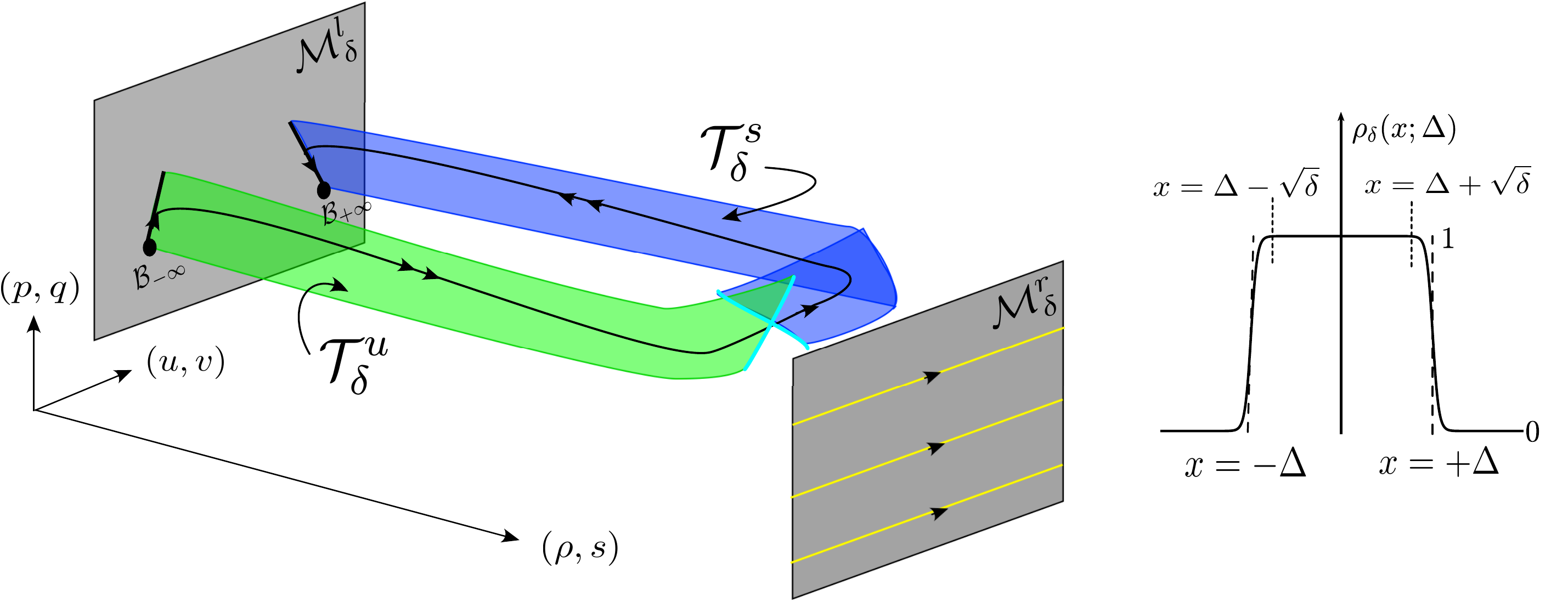}
	\caption{%
		The figure depicts a sketch of the manifolds of the system (\ref{fastslowslow}) when $0<\delta \ll 1$.  The left
		grey manifold is the persisting manifold
		$\mathcal{M}_\delta^l=\mathcal{M}_0^l$ and the right grey manifold represents $\mathcal{M}_\delta^r.$ The bold black curves on $\mathcal{M}_\delta^l$ depict the relevant unstable and stable manifolds of $\mathcal{B}_{\pm \infty}$. The yellow curves on $\mathcal{M}_\delta^r$ represent the flow on this manifold. The front green
		manifold is the part of the unstable manifold of the origin, $\mathcal{T}_\delta^u$, while the back blue manifold
		is the part of the stable manifold of $(2\pi,0,0,0,0,0)$,  $\mathcal{T}_\delta^s$.  The black curve is the
		persisting solution~$\bm{F_\delta}$. The double arrows represent fast transition whilst the single represent slow dynamics. To the right is a sketch of $\rho_{\delta}(x;\Delta)$ and the regions we split $\mathcal{T}_\delta^{u/s}$ into. }
	\label{smoothmanifold}
\end{figure}

\subsection[A1]{Dynamics of the extended system when
  $\bm{0<\delta\ll 1}$} \label{deltasmall}
When $0<\delta \ll 1$, the dynamics of the fast variables $(\rho,s)$
remain uncoupled from the slow $(u,p,v,q)$ variables. This two
dimensional fast sub-system is given by (\ref{inhomequation}) with
$\epsilon(\Delta/\delta)$ approximated by (\ref{epsilon}). Its limit for
$\delta\to 0$ (i.e.  also $\epsilon\to 0$) is
(\ref{reducedfastslowfast}). The fixed points of
\eqref{reducedfastslowfast} persist for $\delta>0$, see
section~\ref{inhomsys}.  However, the heteroclinic connections in
\eqref{reducedfastslowfast} do not persist. Instead a homoclinic
connection to the origin is created. This connection is denoted by
$(\rho_\delta,s_\delta)$ and is explicitly described in the slow
variables by~(\ref{homocliniclemma}). Recall that
$\epsilon(\Delta/\delta)$ is such that $\rho_\delta(\pm\Delta) =1/2$.

The persistence of the fixed points in the two dimensional fast
sub-system implies that the full system (\ref{fastslowfast}) has
$4$-dimensional normally hyperbolic invariant slow manifolds,
explicitly given by
\begin{equation*}
\mathcal{M}_{\delta}^{l} = \{(\tilde{u}, \tilde{p}, \tilde{v}, \tilde{q}, 0, 0)\}, \quad
\mathcal{M}_{\delta}^{r} = \{(\tilde{u}, \tilde{p}, \tilde{v}, \tilde{q}, 1 +\epsilon(\Delta/\delta), 0)\}.
\end{equation*}
Notice
that $\mathcal{M}_{\delta}^{l}=\mathcal{M}_{0}^{l}$ and in the slow
variable the flow on $\mathcal{M}_{\delta}^{l}$ is also governed by
\begin{align*}
u_x&= p, \quad p_{x} = \sin u\cos v;\\ 
v_x&= q, \quad q_{x} = \sin v\cos u-\alpha \sin 2v.
\end{align*}
Whilst on $\mathcal{M}_{\delta}^{r}$ it is governed by
\begin{align*}
u_x&= p, \quad p_{x} = (1-d(1+\epsilon(\Delta/\delta)))\sin u\cos v;\\ 
v_x&= q, \quad q_{x} =  (1-d(1+\epsilon(\Delta/\delta)))\sin v\cos u-\alpha \sin2 v.
\end{align*}
Note that the flow on $\mathcal{M}_{0}^{r}$ is given by the above with
$\epsilon=0$.

Recall that $\mathcal{B}_{-\infty}=(0,0,0,0,0,0)$ and
$\mathcal{B}_{+\infty}=(2\pi,0,0,0,0,0)$ are saddle points. 
On the three dimensional stable manifold
$\mathcal{W}_{\delta}^{s}(\mathcal{B}_{+\infty})$ 
we can define the solution
\[
\bu_\delta^{s}(x;u,v):= (u^{s}_\delta(x;u,v),p_\delta^{s}(x;u,v),v^{s}_\delta(x;u,v),q_\delta^{s}(x;u,v),\rho_\delta(x),s_\delta(x))
\]
for $(u,v)$ near $(u_0^+,v_0^+)$. Whilst on the three dimensional unstable manifold  $\mathcal{W}_{\delta}^{u}(\mathcal{B}_{-\infty})$ we can define the solution
\[
\bu_\delta^{u}(x;u,v):= (u^{u}_\delta(x;u,v),p_\delta^{u}(x;u,v),v^{u}_\delta(x;u,v),q_\delta^{u}(x;u,v),\rho_\delta(x),s_\delta(x))
\]
for $(u,v)$ near $(u_0^-,v_0^-)$. Next we form manifolds by flowing $\bu_\delta^{s}(x;u,v)$ and $\bu_\delta^{u}(x;u,v)$ forwards respectively backwards:
\begin{eqnarray*}
  \mathcal{T}_\delta^s &=&
                           \{
		\bu^s_\delta(x;u,v)\mid (u,v) \mbox{ near } (u_0^+,v_0^+),\,
		x\in [0,\infty) \};\\
	\mathcal{T}_\delta^u &=&
	\{
	\bu^u_\delta(x;u,v)\mid (u,v) \mbox{ near }(u_0^-,v_0^-),\,
	x\in (-\infty,0] \}.
\end{eqnarray*}
Hence these manifolds are part of
$\mathcal{W}_{\delta}^{s}(\mathcal{B}_{+\infty})$ and
$\mathcal{W}_{\delta}^{u}(\mathcal{B}_{-\infty})$ respectively. See Figure~\ref{smoothmanifold} for a sketch of the manifolds $\mathcal{T}_\delta^{u/s}$.

We want to show that $\mathcal{T}_\delta^u$ converges to
$\mathcal{T}_0^u$ and $\mathcal{T}_\delta^s$ to $\mathcal{T}_0^s$ as
$\delta$ goes to~0.  Since $\rho_0^{\pm}$ are heteroclinic connections
and $\rho_\delta$ a homoclinic one, the parametrisation of these
orbits is different. Consequently, the $\xi$ parametrisation of
$\mathcal{T}_0^{u/s}$ is different from the $\xi=x/\delta$
parametrisation in $\mathcal{T}_\delta^{u/s}$.  But this is not
important as we want to compare them as manifolds. To show the
convergence, we split both manifolds $\mathcal{T}_\delta^{u/s}$ in
three parts. We have seen in Lemma \ref{fastslowlem} that
\[
|(\rho_\delta(x)-1,s_\delta(x))| =\mathcal{O}(e^{-\sqrt{2/\delta}}) \text{, for }
x\in(-\Delta+\sqrt{\delta},\Delta-\sqrt{\delta}).
\]
Hence we split the manifold $\mathcal{T}_\delta^{u}$ as follows
\begin{eqnarray*}
	\mathcal{T}_{\delta,\text{I}}^{u} &=&
	\{(\bu^u_\delta(x;u,v)\mid (u,v) \mbox{ near }(u_0^-,v_0^-),\,
	 x \leqslant -\Delta \};\\
	\mathcal{T}_{\delta,\text{II}}^{u} &=&
	\{(\bu^u_\delta(x;u,v)\mid (u,v) \mbox{ near }(u_0^-,v_0^-),\,
	-\Delta \leqslant x\leqslant -\Delta+\sqrt{\delta} \};\\
	\mathcal{T}_{\delta,\text{III}}^{u} &=&
	\{(\bu^u_\delta(x;u,v)\mid (u,v) \mbox{ near }(u_0^-,v_0^-),\,
	-\Delta+\sqrt{\delta} \leqslant x \leqslant 0 \}.
\end{eqnarray*}
Similarly, we split the manifold $\mathcal{T}_\delta^{s}$ as follows
\begin{eqnarray*}
	\mathcal{T}_{\delta,\text{I}}^{s} &=&
	\{(\bu^s_\delta(x;u,v)\mid (u,v) \mbox{ near }(u_0^+,v_0^+),\,
	x \geqslant \Delta \};\\
	\mathcal{T}_{\delta,\text{II}}^{s} &=&
	\{(\bu^s_\delta(x;u,v)\mid (u,v) \mbox{ near }(u_0^+,v_0^+),\,
	\Delta-\sqrt{\delta} \leqslant x\leqslant \Delta \};\\
	\mathcal{T}_{\delta,\text{III}}^{s} &=&
	\{(\bu^s_\delta(x;u,v)\mid (u,v) \mbox{ near }(u_0^+,v_0^+),\,
	0 \leqslant x \leqslant \Delta - \sqrt{\delta} \}.
\end{eqnarray*}
We have sketched all relevant manifolds in Figure~\ref{smoothmanifold}.
Next we will prove the convergence.
\begin{itemize}
\item Fenichel's second persistence theorem states that the stable and
  unstable manifolds $W^{s/u}(\mathcal{M}_{\delta}^{l})$ lie locally
  within $\mathcal{O}(\delta)$ of $W^{s/u}(\mathcal{M}_{0}^{l})$ and
  are invariant under the dynamics of the full fast system
  (\ref{fastslowfast}); see \cite{Fenichel1979, Jones1995}. The manifold
  $\mathcal{T}_{\delta,\text{I}}^{s}$ is a subset of
  $\mathcal{W}_{\delta, slow}^{s}(\mathcal{B}_{+\infty})\subset
  W^{s}(\mathcal{M}_{\delta}^{l})$ and 
  $\mathcal{T}_{\delta,\text{I}}^{u}$ is a subset of
  $\mathcal{W}_{\delta, slow}^{u}(\mathcal{B}_{-\infty})\subset
  W^{u}(\mathcal{M}_{\delta}^{l})$ for 
  $\delta > 0$. This implies that the manifolds
  $\mathcal{T}_{\delta,\text{I}}^{u/s}$ are order $\delta$ close to
  $\mathcal{T}_{0}^{u/s}$ for $\rho\in\left[0,1/2\right]$, i.e. $\vert x \vert \geqslant \Delta$ .
\item For $-\Delta \leq x \leq -\Delta + \sqrt{\delta}$, the slow
  variables can change at most order $\sqrt\delta$. Furthermore, the
  set
  $\{(\rho_\delta(x),s_\delta(x))\mid -\Delta \leqslant x \leqslant
  -\Delta +\sqrt{\delta}\}$ is order
  $\sqrt{\delta}$ close to
  $\{(\rho_0^+(\xi),s_0^+(\xi))\mid \xi\in[0,\infty)\}.$ Hence,
  we see that the manifolds $\mathcal{T}_{\delta,\text{II}}^{u}$
	are order
	$\sqrt{\delta}$ close to $\mathcal{T}_0^{u}$. A similar
        argument gives that the manifolds $\mathcal{T}_{\delta,\text{II}}^{s}$
	are order
	$\sqrt{\delta}$ close to $\mathcal{T}_0^{s}$.
        
      \item The previous observation also implies that the set
        $\{\bu^u_{\delta}(-\Delta+\sqrt{\delta};u,v)\mid (u,v) \mbox{
          near } (u_0^-,v_0^-)\}$ is order $\sqrt\delta$ close to the
        intersection
        $\mathcal{S}_0^u \cap \mathcal{S}_0^{u,r} \subset
        \mathcal{T}_0^u$. And the set
        $\{\bu_{\delta}^s(\Delta-\sqrt{\delta};u,v)\mid (u,v) \mbox{
          near } (u_0^+,v_0^+)\}$ is order $\sqrt\delta$ close to the
        intersection
        $\mathcal{S}_0^s \cap \mathcal{S}_0^{s,r} \subset
        \mathcal{T}_0^s$.  For
        $x\in[-\Delta+\sqrt\delta,\Delta-\sqrt\delta]$,
        $(\rho_\delta(x),s_\delta(x))$ is at least order
        $\exp(-\sqrt{2/\delta})$ close to $(1,0)$ (see
        Lemma~\ref{fastslowlem}), so on this finite spatial interval, the
        flow of the perturbed system is order
        $\exp(-\sqrt{2/\delta}) $ close to the flow on
        $\mathcal{M}_0^r$.  Hence the manifolds
        $\mathcal{T}_{\delta,\text{III}}^{u/s}$
	are order $\sqrt{\delta}$ close to $\mathcal{T}_0^{u/s}$.
\end{itemize}
So we can conclude that $\mathcal{T}_\delta^{u/s}$ are order
$\sqrt\delta$ close to $\mathcal{T}_0^{u/s}$ and hence
$\mathcal{T}_\delta^{u/s}$ converges to $\mathcal{T}_0^{u/s}$ when
$\delta$ goes to~0. Since $\mathcal{T}_0^{u}\vert_{x=0}$ and
$\mathcal{T}_0^{s}\vert_{x=0}$ intersect transversely,
$\mathcal{T}_\delta^{u}\vert_{x=0}$ and
$\mathcal{T}_\delta^{s}\vert_{x=0}$ will intersect transversely. Thus
we can conclude that the heteroclinic connection $\bm{F_0}$ persists for
$\delta$ small.
\end{section}


\begin{section}{Discussion} \label{discussion}
  In this paper, we have presented both an analytic and numerical
  investigation into the existence of stationary fronts in the system
  (\ref{coupledsystem}) for $0<\alpha <1/2.$ In Section
  \ref{numericalinvestigation} we showed numerically the existence of
  a pitchfork bifurcation at some $\alpha \in (0,1/2)$ for all
  $\Delta>0$ and $d> 0$. Then in Section \ref{diagramanalysis}, using
  the piecewise constant approximation $\rho_0(x;\Delta)$ given by
  (\ref{step}), we gave an implicit expression for the bifurcation
  locus when $d=1$ and approximations for the two cases $\Delta\gg1$
  and $d\gg1$. In Section \ref{coupledsys} we used the implicit
  expression for $d=1$ from Section \ref{diagramanalysis} to
  rigorously show existence of a non-zero $v(x)$ component using
  Lyapunov-Schmidt reduction. Finally, in Section~\ref{hamslowfast}
  using geometric singular perturbation theory, we showed that if
  fronts exist for the piecewise constant inhomogeneity
  $\rho_0(x;\Delta)$, they persist for a smooth sharp inhomogeneity
  $\rho_{\delta}(x;\Delta)$.
  
  The work in this paper provides a broader understanding of what
  happens to the destabilised front looked at by Braun {\it et
    al.\/}~\cite{Braun1988} in the coupled homogeneous system ($d=0$)
  and how the front can be stabilised in the coupled system using a
  spatial inhomogeneity. The effect of the spatial inhomogeneity is to
  stabilise the sine-Gordon front by a small perturbation where the
  $\psi$-component dips and the $\theta$-component rises around the
  centre of the transition. Exploring the effects of hat-like
  (smooth and non-smooth) inhomogeneities on fronts and their
  interaction is of practical interest, but also yields some
  interesting mathematical results, especially the rigorous proof of
  the persistence of fronts when going from non-smooth to smooth
  inhomogeneities. 

  Several interesting avenues for future work are possible. Since
  existence has been established, a natural question would be to
  consider the stability of the stationary fronts in the system
  (\ref{coupledsystem}). Time simulations shown in
  Figure~\ref{spacetimeplots} indicate that the pitchfork bifurcation
  is supercritical. It should be possible to verify this analytically
  using Lyapunov-Schmidt reduction, for example as developed for the
  stability of rolls in the Swift-Hohenberg
  equation~\cite{Mielke1995,Mielke1997}. Computing the stability with
  respect to forcing or damping terms would also be of practical
  interest. Alternatively, one could also consider the existence and
  stability of stationary solutions where both $u$ and $v$ connect to
  zero as $|x|\rightarrow\infty$ building on the work of of Derks {\it
    et al.}~\cite{Derks2011}.

\begin{figure}
	\captionsetup[subfigure]{labelformat=empty}
	\centering
	\subfloat[(a)]{\includegraphics[scale=0.4]{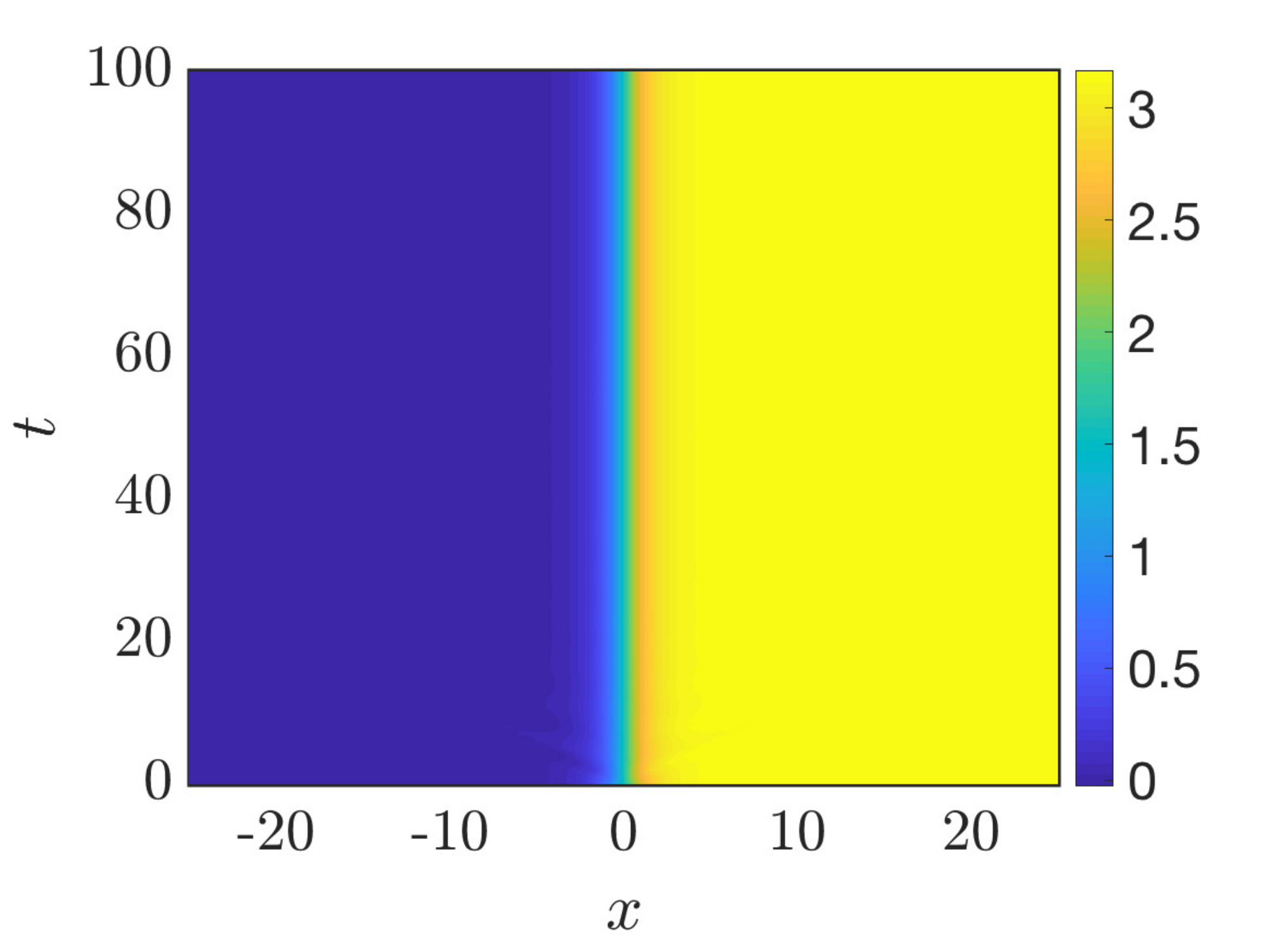}} 
	\subfloat[(b)]{\includegraphics[scale=0.4]{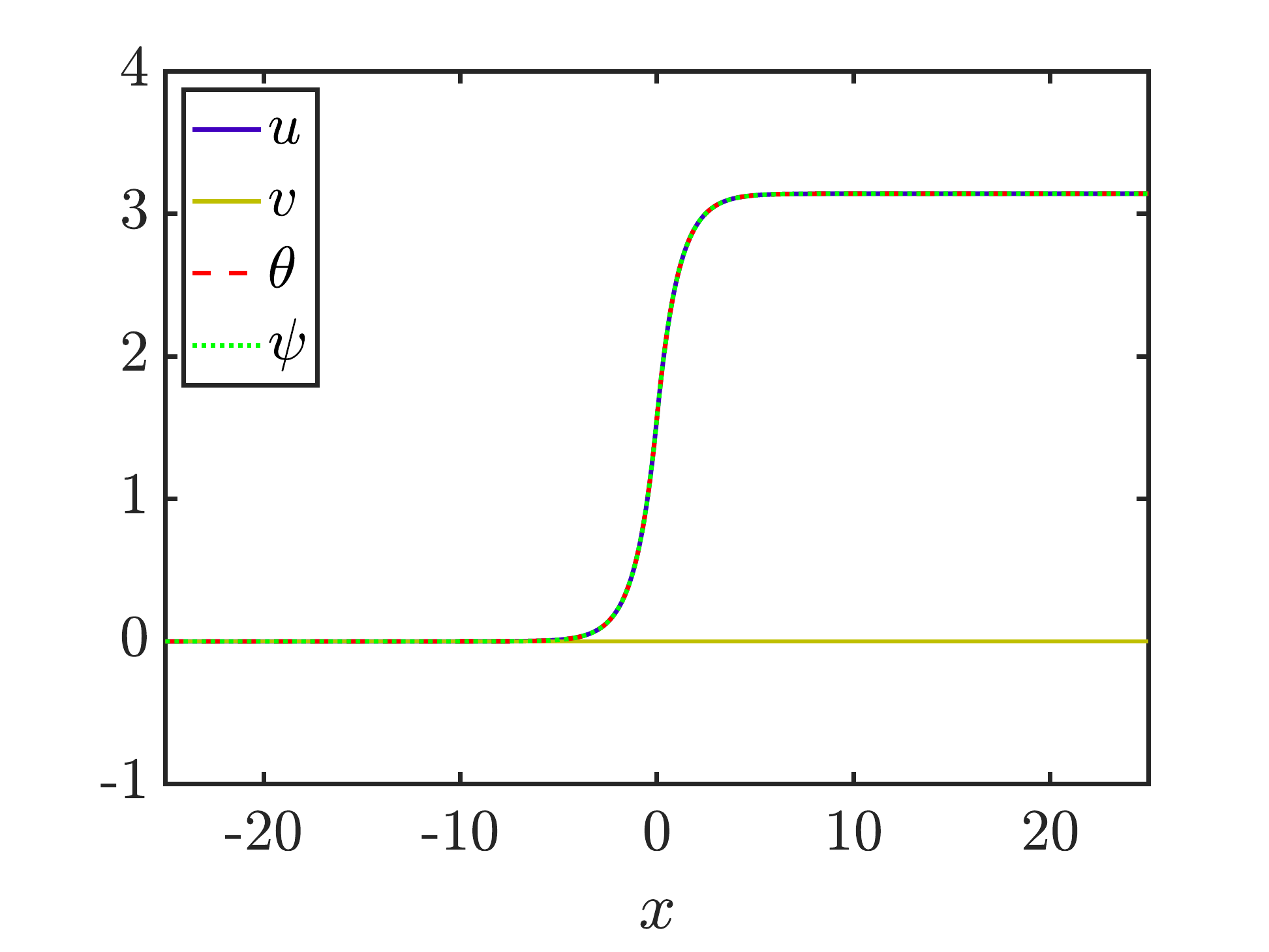}}\vspace{-14pt}\\
	\subfloat[(c)]{\includegraphics[scale=0.4]{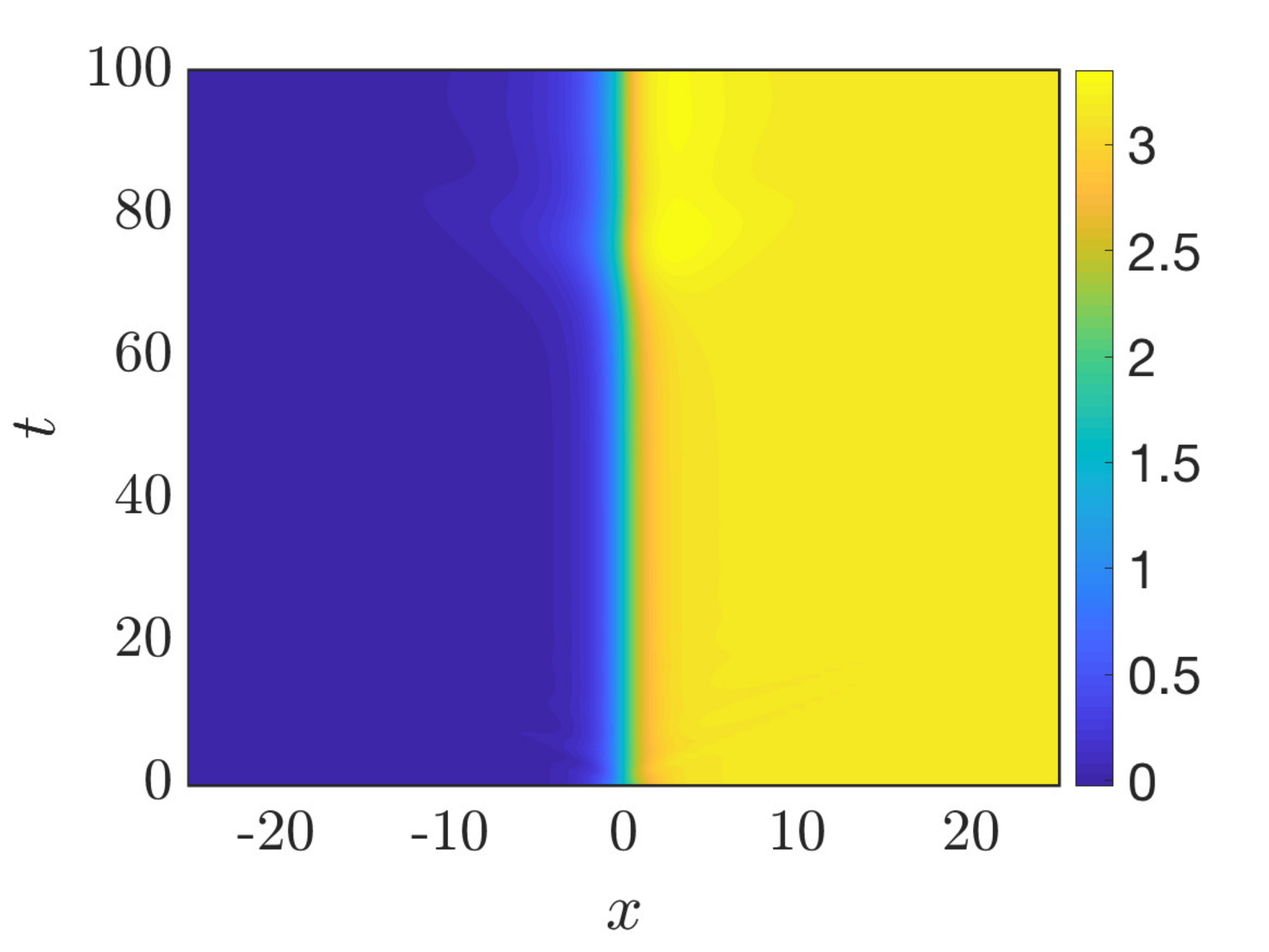}} 
	\subfloat[(d)]{\includegraphics[scale=0.4]{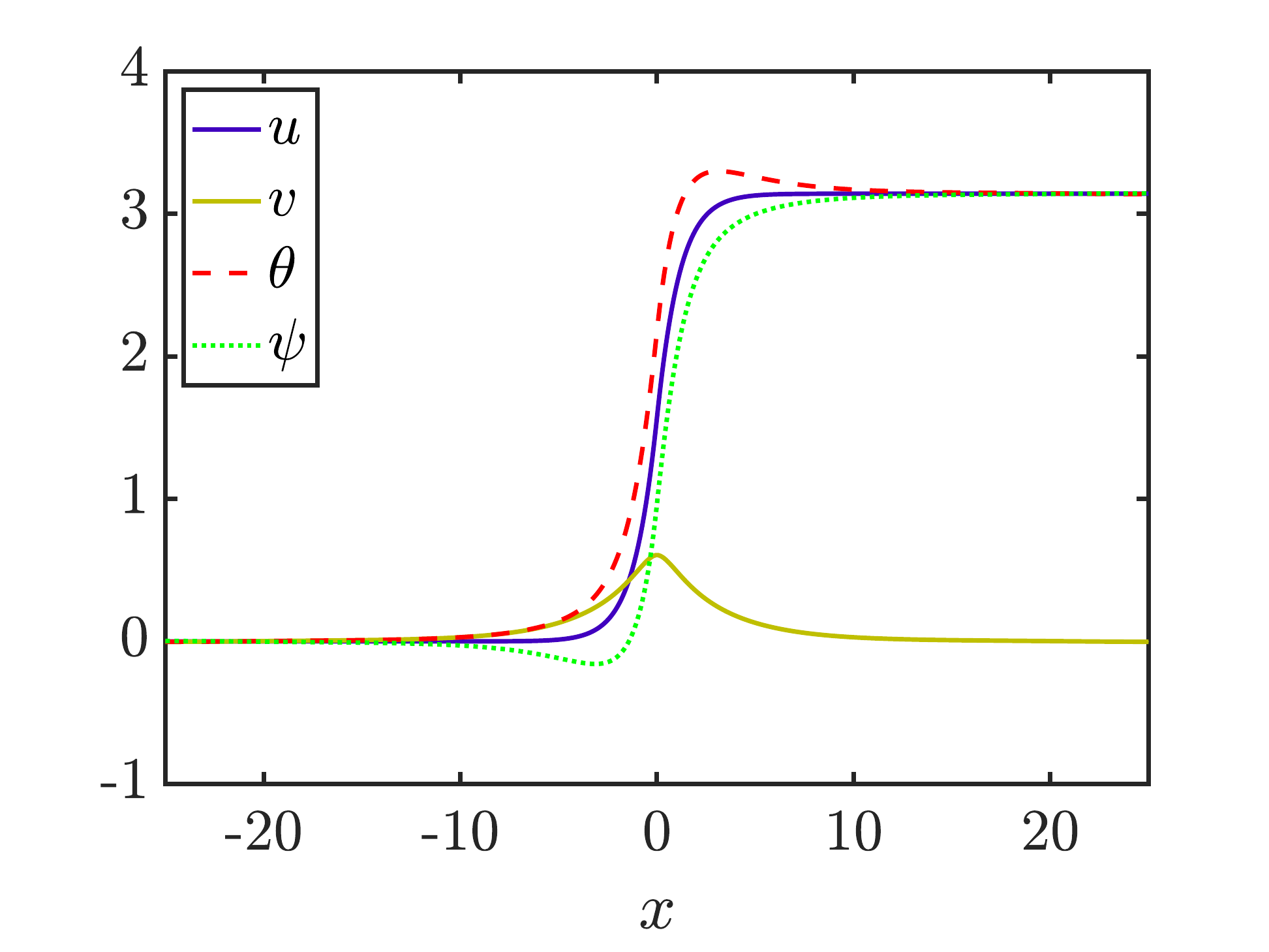}}\vspace{-14pt}\\
	\caption{This figure gives the dynamics of the coupled
		sine-Gordon system with the step
		inhomogeneity~\eqref{onejump} and $d=2$. (a)
		is a space time plot of $u(x,t=0)=u_\pi(x;2)$ with
		$\alpha=0.2$ and (b) is a plot of the
		solution profile in (a) at $t=100$. No bifurcation occurs and
		$u_\pi(x)$ is stable.  (c) is a space time
		plot of $u_\pi(x)$ with $\alpha=0.45$ and (d) is a plot of the solution profile at $t=100$. The
		$u_{\pi}(x)$ solution becomes unstable and bifurcates to a
		new solution. Note we have included a small damping term to
		suppress the additional radiation in the Hamiltonian
		system.\label{f:pi_kink}}
\end{figure} 

Another direction would be to explore the system with other
  inhomogeneities. The results in this paper can be extended to
  consider the existence of fronts in the system (\ref{coupledsystem})
  with a smooth ``step" inhomogeneity of the form
\begin{equation} \label{onesmoothjump} \rho_{\delta}(x) =
  \frac{\tanh\left(x/\delta\right)+1}{2}.
\end{equation}
Unlike the hat-like spatial inhomogeneities studied in this paper, the above
step has only one jump which is centred at $x=0.$ As $\delta \to 0$
the above converges pointwise to
\begin{equation} \label{onejump}
\rho_{0}(x)= 
\begin{cases}
0,&  x  <0,  \\
1,     &  x >0. 
\end{cases} 
\end{equation}
When $d>1$, the system (\ref{alphasys}) with the above piecewise
constant inhomogeneity and boundary conditions
$(u(-\infty),v(-\infty)) = (0,0)$ and
$(u(+\infty),v(+\infty)) = (\pi,0)$ is known to have solutions
$(u_{\pi},0)$ \cite{Derks2007} where
\begin{equation*}
	u_{\pi}(x;d) = 
	\begin{cases}
		4\arctan(e^{x+x_1}) ,&  x < 0, \\
		4\arctan(e^{\sqrt{d-1}x+x_2})-\pi,     &  x > 0,
	\end{cases}
	\label{pikink}
\end{equation*}
and $x_1(d)$ and $x_2(d)$ are matching constants
\begin{equation*}
	x_1(d) = \ln\bigg(\tan\bigg(\frac{\arccos((2-d)/d)}{4}\bigg)\bigg) \quad \text{and} \quad
	x_2(d) = \ln\bigg(\tan\bigg(\frac{\arccos((2-d)/d)+\pi}{4}\bigg)\bigg).
\end{equation*}
The bifurcation points of the solution $(u_\pi(x),0)$ in the full
system (\ref{alphasys}) with spatial inhomogeneity~(\ref{onejump}) are
given exactly by the implicit relation (\ref{BifDeltaLarge}) in
Section \ref{diagramanalysis}. Time simulations shown in
Figure~\ref{f:pi_kink} suggest the existence of a bifurcation whereby
the $v(x)$ component becomes non-zero, similar to the one studied in
this paper. With the explicit front solutions above, one can employ
Lyapunov-Schmidt reduction to show the existence of a pitchfork
bifurcation and the procedure is almost identical to the one completed
in Section~\ref{coupledsys}. Finally, it is possible to show
persistence of solutions for the smooth sharp
inhomogeneity~(\ref{onesmoothjump}) following ideas in Section
\ref{hamslowfast}. Setting $\epsilon = 0$ in
(\ref{Gardner}) means the heteroclinic connections in the fast reduced
system persist when $0<\delta \ll 1.$ One then can consider the flow
along the stable and unstable manifolds on $\mathcal{M}_\delta^l$ and
$\mathcal{M}_\delta^r$ respectively and apply Fenichel's theorems to
prove persistence.

Another extension is the generalisation of the smoothening results for
an arbitrary smooth sharp inhomogeneity. In this paper we restricted
ourselves to using the dynamics of (\ref{Gardner}) to describe
the spatial inhomogeneity. However, the ideas extend to any
Hamiltonian system that has a bifurcation from a heteroclinic to a
homoclinic. The work in Section~\ref{hamslowfast} gives a framework to
generalise the smoothening result and prove persistence with respect to a general
class of perturbations. A further extension would be to generalise the
smoothening result to any system of semi-linear wave equations with spatial
inhomogeneities.

{\bf Acknowledgements} The authors would like to thank Arjen Doelman
for inspiring discussions on this problem and the referees for their constructive feedback. JB acknowledges the EPSRC
whose institutional Doctoral Training Partnership grant (EP/N509772/1)
helped fund his PhD.

The authors confirm that data underlying the findings are available
without restriction. Details of the data and how to request access are
available via the University of Surrey publications repository.

\end{section}


\appendix
\section{Variation of parameters}\label{appA} 

Here we determine an expression for $V_{21}$ as required in (\ref{c}). Upon substituting $V_{11} = V_{22} = 0,$ which were determined at $\mathcal{O}(A),$ into (\ref{linear2}) yields
\begin{equation*}
	(I-Q)\bigg(\mathcal{L}_0\bm{V_2} +\tilde{\alpha}\mathcal{L}_1\bm{V_2}
	+(1-\rho)\left(
	\begin{array}{ c }
	\Psi^2\sin(u_0)/2\\
	0
	\end{array} \right)\bigg) = \mathbf{0}.
\end{equation*} 
We are interested in the first component $V_{21},$ of the vector $\bm{V_2},$ which by the above is governed by 
\begin{equation*}
\frac{d^2Y}{dx^2} - (1-\rho)\cos(u^*)Y = -\frac{1}{2}(1-\rho)\sin(u_0)\Psi^2. 
\end{equation*}
where we have set $Y = V_{21}$ for notational convenience. To solve this second order ODE we are required to use variation of parameters. Since the integral (\ref{c}) is over the interval $[\Delta,\infty)$ it is only necessary to compute $Y(x)$ in the region $x>\Delta.$ Hence we seek to solve,
\begin{equation}
\frac{d^2A}{dx^2} - \cos(u_0)A = -\frac{1}{2}\sin(u_0)\Psi^2. \label{fullODE}
\end{equation}
where 
\begin{equation*}
u_0 = 4\arctan(e^{x-x^*}),
\end{equation*}
\begin{equation*}
\Psi = -R(\tanh(x-x^*)+\sqrt{1-2\alpha^*})\exp(-\sqrt{1-2\alpha^*}(x-x^*)).
\end{equation*}
The second order ODE
\begin{equation*}
\frac{d^2Y}{dx^2} - \cos(u_0)Y = 0
\end{equation*}
has two linearly independent solutions (see e.g.~\cite{Derks2007}) given by 
\begin{equation*}
Y_{1}(x) = \sech(x-x^*) \quad \text{and} \quad Y_{2}(x) = \sinh(x-x^*)+\frac{x-x^*}{\cosh(x-x^*)}.
\end{equation*}
Using the variation of parameters method one determines
\begin{multline*}
Y(x) = \frac{Y_1(x)}{4}\int^{x}Y_2(\bar{x})\sin(u_0(\bar{x}))\Psi^2(\bar{x})d\bar{x}
\\-\frac{Y_2(x)}{4}\int^{x}Y_1(\bar{x})\sin(u_0(\bar{x}))\Psi^2(\bar{x})d\bar{x} 
+AY_1(x)+BY_2(x),
\end{multline*}
where constants $A$ and $B$ are determined. It can be shown from the boundary conditions $Y(+\infty) = Y'(+\infty) = 0$ that $B = 0.$ Finally since $Y(x)$ is an odd function and $Y_1(x)$ is even we must have $A=0.$ Hence,
\begin{equation*}
V_{21}(x) = \frac{Y_1(x)}{4}\int^{x}Y_2(\bar{x})\sin(u_0(\bar{x}))\Psi^2(\bar{x})d\bar{x}
\\-\frac{Y_2(x)}{4}\int^{x}Y_1(\bar{x})\sin(u_0(\bar{x}))\Psi^2(\bar{x})d\bar{x}.
\end{equation*}
\section{Expressions for the eigenfunction}\label{appB}
Here we determine the matching constant $A$ and the rescaling constant $R$ of the eigenfunction (\ref{eigenfunctionq=1}). Matching the eigenfunction (\ref{eigenfunctionq=1}) at $x=-\Delta$ yields
\begin{equation*}
A = -\frac{\bigg(\sqrt{1+\Lambda}+\sqrt{\frac{2-h}{2}}\bigg)\left(\tan\left(\arccos(1-h)/4\right)\right)^{\sqrt{1+\Lambda}}}{\cos\left(\sqrt{-\Lambda}\Delta\right)}.
\end{equation*}
Now we wish to find the rescaling constant such that $\int_{-\infty}^{+\infty}\Psi^2=1.$ Since $\Psi$ is an even function,
\begin{equation*}
\frac{1}{R^2}= A^2\bigg(\frac{\sin(2\sqrt{-\Lambda} \Delta)}{2\sqrt{-\Lambda}}+\Delta\bigg) +2\int_{\Delta}^{\infty}e^{-2\sqrt{1+\Lambda}(x-x^{*})}(\tanh(x-x^{*})+\sqrt{1+\Lambda})^2dx. 
\end{equation*}
Hence
\begin{multline*}
\frac{1}{R^2} = A^2\bigg(\frac{\sin(2\sqrt{2\alpha^*} \Delta)}{2\sqrt{2\alpha^*}}+\Delta\bigg)
 \\ +2\bigg[\frac{(1-2\alpha+(1-\alpha)\sqrt{1-2\alpha})e^{-2\sqrt{1-2\alpha}(x-x^*)}}{(e^{2(x-x1)}+1)\alpha^2(2\alpha-1)}
\\
 \times (\alpha^2e^{2(x+x^*)}+2(\alpha-1)\sqrt{1-2\alpha}+\alpha^2-4\alpha+2)\bigg]_{\Delta}^{\infty}.
\end{multline*}

\bibliographystyle{plain}
\bibliography{stationary}

\begin{thebibliography}{10}

\bibitem{Bour1862}
E.~Bour.
\newblock Th\'eorie de la d\'eformation des surfaces.
\newblock {\em J. Ecole Imperiale Polytechnique}, 19:1--48, 1862.

\bibitem{Braun1988}
O~M Braun, Yu~S Kivshar, and A~M Kosevich.
\newblock Interaction between kinks in coupled chains of adatoms.
\newblock {\em Journal of Physics C: Solid State Physics}, 21(21):3881, 1988.

\bibitem{Derks2007}
G.~Derks, A.~Doelman, S.~A. van Gils, and H.~Susanto.
\newblock Stability analysis of {$\pi$}-kinks in a 0-{$\pi$} {J}osephson
  junction.
\newblock {\em SIAM J. Appl. Dyn. Syst.}, 6(1):99--141, 2007.

\bibitem{Derks2012}
Gianne Derks, Arjen Doelman, Christopher J.~K. Knight, and Hadi Susanto.
\newblock Pinned fluxons in a {J}osephson junction with a finite-length
  inhomogeneity.
\newblock {\em European J. Appl. Math.}, 23(2):201--244, 2012.

\bibitem{Derks2011}
Gianne Derks and Giuseppe Gaeta.
\newblock A minimal model of {DNA} dynamics in interaction with
  {RNA}-polymerase.
\newblock {\em Phys. D}, 240(22):1805--1817, 2011.

\bibitem{AUTO}
E~J Doedel, R~C Paffenroth, A~R Champneys, T~F Fairgrieve, Yu~A Kuznetsov, B~E
  Oldeman, and B~Sandstede.
\newblock {Auto07p: Continuation and bifurcation software for ordinary
  differential equations}.
\newblock Technical report, Concordia University, 2007.

\bibitem{Doelman2009}
Arjen Doelman, Peter van Heijster, and Tasso~J. Kaper.
\newblock Pulse dynamics in a three-component system: Existence analysis.
\newblock {\em Journal of Dynamics and Differential Equations}, 21(1):73--115,
  2009.

\bibitem{Fenichel1979}
Neil Fenichel.
\newblock Geometric singular perturbation theory for ordindary differential
  equations.
\newblock {\em Journal of Differential Equations}, 31:53--98, 1979.

\bibitem{Frenkel1939}
J.~Frenkel and T.~Kontorova.
\newblock On the theory of plastic deformation and twinning.
\newblock {\em Izvestiya Akademii Nauk SSSR, Seriya Fizicheskaya}, 1:137--149,
  1939.

\bibitem{Goatham2011}
S.~W. Goatham, L.~E. Mannering, R.~Hann, and S.~Krusch.
\newblock Dynamics of multi-kinks in the presence of wells and barriers.
\newblock {\em Acta Physica Polonica B}, 42:2087--2106, 2011.

\bibitem{Goh2016}
Ryan Goh and Arnd Scheel.
\newblock Pattern formation in the wake of triggered pushed fronts.
\newblock {\em Nonlinearity}, 29(8):2196--2237, 2016.

\bibitem{Golubitsky1985}
Martin Golubitsky and David~G. Schaeffer.
\newblock {\em Singularities and Groups in Bifurcation Theory}, volume~1.
\newblock Springer, 1985.

\bibitem{Hamdi2011}
Samir Hamdi, Brian Morse, Bernard Halphen, and William Schiesser.
\newblock Analytical solutions of long nonlinear internal waves: Part i.
\newblock {\em Natural Hazards}, 57(3):597--607, Jun 2011.

\bibitem{Homma1984}
Shigeo Homma and Shozo Takeno.
\newblock A coupled base-rotator model for structure and dynamics of dna.
\newblock {\em Progress of Theoretical Physics}, 72(4):679--693, October 1984.

\bibitem{Jones1995}
Christopher K. R.~T. Jones.
\newblock {\em Dynamical Systems}, volume 1609 of {\em Lecture notes in
  Mathematics}, chapter Geometric singular perturbation theory, pages 44--118.
\newblock Springer, 1995.

\bibitem{Mann1997}
E.~Mann.
\newblock {Systematic perturbation theory for sine-Gordon solitons without use
  of inverse scattering methods}.
\newblock {\em J. Phys. A: Math. Gen.}, 30:1227--1241, 1997.

\bibitem{Mielke1995}
A.~Mielke.
\newblock A new approach to sideband-instabilities using the principle of
  reduced instability.
\newblock In {\em Nonlinear dynamics and pattern formation in the natural
  environment ({N}oordwijkerhout, 1994)}, volume 335 of {\em Pitman Res. Notes
  Math. Ser.}, pages 206--222. Longman, Harlow, 1995.

\bibitem{Mielke1997}
Alexander Mielke.
\newblock Instability and stability of rolls in the {S}wift-{H}ohenberg
  equation.
\newblock {\em Comm. Math. Phys.}, 189(3):829--853, 1997.

\bibitem{Piette2007}
Bernard Piette and W~J Zakrzewski.
\newblock Scattering of sine-gordon kinks on potential wells.
\newblock {\em Journal of Physics A: Mathematical and Theoretical},
  40:5995--6010, 2007.

\bibitem{Yakushevich1989}
L.~V. Yakushevich.
\newblock {Nonlinear DNA dynamics: a new model}.
\newblock {\em Phys. Lett. A}, 136:413--417, 1989.

\end{thebibliography}

\end{document}